\documentclass[fleqn]{cas-sc}

\usepackage[numbers,sort&compress]{natbib}

\usepackage{silence}
\usepackage{nag}

\usepackage{amsfonts}
\usepackage{amsmath}
\usepackage{amssymb}
\usepackage{amsthm}

\usepackage{mathtools}
\usepackage{bm}
\usepackage{calc}

\usepackage{tikz}
\usetikzlibrary{decorations.pathreplacing,calligraphy}
\usepackage{nicematrix}

\usepackage{diagbox}
\usepackage{multirow}
\usepackage{siunitx}

\usepackage{algorithm}
\usepackage[noEnd=true]{algpseudocodex}

\usepackage{cleveref}

\makeatletter
\def\input@path{{./Tables}}
\makeatother

\graphicspath{{./figs}}

\allowdisplaybreaks

\usepackage{enumitem}
\setlist[enumerate]{leftmargin=.5in}
\setlist[itemize]{leftmargin=.5in}

\theoremstyle{plain}
  \newtheorem{theorem}{Theorem}[section]
  
  \newtheorem{heuristic}[theorem]{Heuristic}
  \newtheorem{lemma}[theorem]{Lemma}
\theoremstyle{definition}
  \newtheorem{definition}[theorem]{Definition}
  
  \newtheorem{remark}[theorem]{Remark}

\numberwithin{equation}{section} 

\crefformat{equation}{(#2#1#3)}
\Crefformat{equation}{Equation~(#2#1#3)}
\Crefmultiformat{equation}{Equations~(#2#1#3)}{ and~(#2#1#3)}{, (#2#1#3)}{ and~(#2#1#3)}

\crefname{theorem}{Theorem}{Theorems}
\crefname{lemma}{Lemma}{Lemmas}
\crefname{proposition}{Proposition}{Propositions}
\crefname{definition}{Definition}{Definitions}
\crefname{example}{Example}{Examples}
\crefname{remark}{Remark}{Remarks}
\crefname{figure}{Figure}{Figures}
\crefname{algorithm}{Algorithm}{Algorithms}
\crefname{table}{Table}{Tables}
\crefname{section}{Section}{Sections}
\crefname{subsection}{Section}{Sections} 
\crefname{appendix}{Appendix}{Appendices} 
\crefname{heuristic}{Heuristic}{Heuristics} 

\makeatletter
\newcommand{\getcrefname}[1]{\csname cref@#1@name\endcsname}
\newcommand{\getCrefname}[1]{\csname Cref@#1@name\endcsname}
\newcommand{\getcrefnameplural}[1]{\csname cref@#1@name@plural\endcsname}
\makeatother

\DeclarePairedDelimiterX\set[1]\lbrace\rbrace{\def\given{\;\delimsize\vert\;}#1}
\DeclarePairedDelimiterX\event[1][]{\def\given{\;\delimsize\vert\;}#1}
\DeclarePairedDelimiter{\abs}{\lvert}{\rvert}
\DeclarePairedDelimiter{\norm}{\lVert}{\rVert}
\DeclarePairedDelimiter{\floor}{\lfloor}{\rfloor}
\DeclarePairedDelimiter{\round}{\lfloor}{\rceil}

\DeclarePairedDelimiter{\pareno}{(}{)} 
\NewDocumentCommand{\bigo}{s o m}{%
  \mathcal{O}\IfBooleanTF{#1}
    {\pareno*{#3}}%
    {\IfValueTF{#2}{\pareno[#2]{#3}}{\pareno{#3}}}%
}

\DeclareMathOperator*{\lcm}{lcm}
\DeclareMathOperator{\var}{var}

\DeclareMathOperator{\nullity}{nullity}
\newcommand{\bP}{\operatorname{\mathbb{P}}}
\newcommand{\bE}{\operatorname{\mathbb{E}}}
\newcommand{\R}{\mathbb{R}}
\newcommand{\C}{\mathbb{C}}
\newcommand{\Z}{\mathbb{Z}}

\newcommand{\cu}{\mathrm i}

\newcommand{\divs}{\operatorname{\mid}}
\newcommand{\ndivs}{\operatorname{\nmid}}

\newcommand{\uroot}[2]{\eta_{#1}^{#2}}
\newcommand{\idmat}[1]{{\tt I}_{#1}}

\renewcommand{\hat}{\widehat}

\newcommand{\binomdist}[2]{{\rm binom}(#1, #2)}
\newcommand{\normaldist}[2]{\mathcal{N}(#1, #2)}
\newcommand{\chisqdist}[1]{\chi_{#1}^2}
\newcommand{\ncxsqdist}[2]{\chi_{#1}'^2(#2)}

\begin{document}
\let\WriteBookmarks\relax
\shorttitle{Recovery of Integer Signals from Limited DFT Samples: Lattice Methods and Stability Analysis}
\shortauthors{H.W. Levinson and I. Viviano}

\title [mode = title]{Recovery of Integer Signals from Limited DFT Samples: Lattice Methods and Stability Analysis}                      

\author[1]{Howard W. Levinson}[
                        orcid=0000-0003-4933-4645]
\cormark[1]
\ead{hlevinso@oberlin.edu}

\affiliation[1]{organization={Computer Science Department, Oberlin College},
                city={Oberlin},
                state={OH},
                postcode={44074}, 
                country={USA}}

\author[2]{Isaac Viviano}[orcid=0009-0003-2584-1586]
\affiliation[2]{organization={Department of Mathematics, University of Wisconsin -- Madison},
city={Madison},state={WI},postcode={53706},country={USA}}
\ead{iviviano@wisc.edu}

\cortext[cor1]{Corresponding author}

\begin{abstract}
We analyze lattice-based algorithms for recovering integer-valued signals from partial discrete Fourier transform (DFT) measurements. These algorithms formulate signal recovery as the problem of finding short vectors in an appropriately constructed lattice.  We derive parameter estimates that predict the regimes required for successful recovery and quantify how these estimates depend on the signal length, the error of an initial guess, and the number of sampled DFT coefficients. The analysis also reveals a close connection between the lattice parameters and the precision of the measured DFT coefficients, providing insight into the numerical stability of the inversion. Numerical experiments demonstrate close agreement between the theoretical predictions and observed recovery thresholds over a broad range of problem parameters.
\end{abstract}

\begin{keywords}
Discrete Fourier transform \sep Lattice basis reduction \sep LLL algorithm \sep Integer signal recovery \sep Inverse problems
\end{keywords}

\maketitle

\NiceMatrixOptions{xdots/horizontal-labels}

\ExplSyntaxOn
\keys_define:nn { nicematrix / Hbrace } { line-style .code:n = }
\ExplSyntaxOff

\tikzset{
  nicematrix/brace/.style = {
    decoration = {calligraphic brace, amplitude=0.4em, raise=.1em},
    line width = 0.1em, decorate},
  brace1/.style = {
    decoration = {calligraphic brace, amplitude=0.25em, raise=.1em},
    line width = 0.1em, decorate},
}

\section{Introduction}

Reconstructing a signal from incomplete discrete Fourier transform (DFT) measurements is a fundamental problem in signal processing \cite{candes2006robust,candes2006stable}.  When only a limited number of Fourier coefficients are available, the inverse problem typically admits infinitely many solutions.
Prior information is therefore essential for obtaining a meaningful reconstruction.
Many prominent approaches, including sparsity- and regularity-based methods \cite{donoho2006compressed,candes2006robust,candes2006stable}, impose structural assumptions while allowing the signal values to vary continuously.  In contrast, integer-valued constraints restrict the signal to a discrete feasible set,
leading to fundamentally different reconstruction methods and theoretical guarantees.

We consider this reconstruction problem for integer-valued signals.
Our recent work \cite{levinson2025recovery} proved that a small,
carefully chosen subset of DFT coefficients uniquely determines any integer-valued signal 
and introduced a corresponding reconstruction framework based on a sequence of manageable integer linear programs (ILPs).
That work further showed that embedding these inverse problems into integer lattices makes reconstruction substantially more efficient.
Specifically, each lattice is designed so that the solution to the corresponding subproblem is represented by a short lattice vector, which can be recovered using a lattice reduction algorithm
such as Lenstra-Lenstra-Lov\'asz (LLL).  Computational experiments demonstrated dramatic speedups of this lattice formulation over standard ILP-based approaches and enabled the recovery of 
much larger one- and two-dimensional signals from highly incomplete DFT data.

The success of this algorithm raises a deeper mathematical question.
Under what conditions does lattice reduction recover an integer-valued signal from its limited DFT measurements,
and how should the corresponding lattice be optimally designed?  The answer depends on the lattice geometry induced by the partial Fourier constraints.
In particular, the lattice construction contains several scaling parameters that control the separation between the vector representing the true signal and competing vectors whose Fourier coefficients only approximately match the measured data.
Appropriate choices of these parameters are therefore essential for successful recovery and are closely tied to the numerical precision of the available DFT coefficients. Existing implementations, however, rely largely on empirical tuning.

The main contribution of this paper is a theoretical analysis of this lattice-based approach to recovering integer signals from partial DFT measurements.
We relate successful reconstruction to the geometry of the Fourier constraints by estimating the number of competing lattice vectors that are shorter than the vector representing the true signal. This leads to explicit estimates for the lattice parameters required for recovery and shows how they depend on the signal length, the number of sampled Fourier coefficients, the integer range of the signal, and the accuracy of an initial guess.
The analysis also reveals a direct connection between these lattice parameters and the precision of the measured DFT coefficients, providing insight into the numerical stability of the inversion. Extensive numerical experiments demonstrate close agreement between the predicted and observed recovery thresholds over a broad range of problem parameters. Together, these results provide practical guidance for reconstructing integer-valued signals from highly incomplete Fourier data and a mathematical framework for understanding the stability of this discrete Fourier inverse problem.

\subsection{Related Works}
Reconstruction from incomplete Fourier measurements has been studied extensively under a variety of discrete signal models.
Earlier works~\cite{pei2022binary1D,pei2023binary} established uniqueness results for recovering binary signals from partial DFT measurements, which our recent work~\cite{levinson2025recovery} generalized to integer-valued signals of arbitrary dimension.  In particular, \cite{levinson2025recovery} established an algebraic sampling structure underlying recovery of integer-valued signals from partial DFT measurements, leading to sampling sets that uniquely determine an integer-valued signal from a highly incomplete collection of Fourier coefficients.
That work also introduced divide-and-conquer reconstruction methods that reduce the global inverse problem to a sequence of lower-dimensional integer feasibility problems, and demonstrated that lattice-based formulations of these subproblems substantially outperform direct integer linear programming approaches.
The present work develops a theoretical framework for understanding this empirical success.
Related work has considered binary reconstruction in the bandlimited setting~\cite{levinson2021,levinson2023}, binary reconstruction from incomplete Fourier data~\cite{nashold2002synthesis,nashold1989synthesis,mao2012reconstruction}, Fourier recovery under finite-alphabet constraints~\cite{BENDORY2023151}, and algebraic aspects of Fourier inversion and uncertainty principles~\cite{stolk2010algebraic,tao_2005_1}.

Following the development of the LLL algorithm,
lattice basis reduction has frequently been used to reformulate and solve integer feasibility and optimization problems.
Beyond its foundational role in cryptography and computational algebra and number theory~\cite{Simon2010,Lenstra1982,Coppersmith1996},
LLL has been applied to integer programming and combinatorial optimization~\cite{Lenstra1983,Aardal2000}.
Of particular relevance is the lattice basis formulation in \cite{Aardal2000},
which finds the integer solutions to an equality-constrained linear system after reduction with LLL.
Subsequent work analyzed related nullspace,
range-space,
and extended formulations,
showing that lattice basis reduction can produce reformulations whose geometry is substantially more favorable for a subsequent search~\cite{aardal2010lattice,aardal2023study,aardal2004hard,krishnamoorthy2009column}.

Our approach is closely related to this line of work, but differs in several important respects.
First,
the integer programming formulations above begin with systems having integer-valued coefficients,
allowing their analysis to leverage arithmetic properties of integer matrices and their nullspaces.
In the present setting,
the constraints arise from partial Fourier measurements and therefore involve irrational complex roots of unity.
Consequently, the resulting lattice geometry depends on both the algebraic structure of the DFT constraints and the numerical precision of the measurements.
Second,
these classical reformulations use lattice reduction as a preprocessing step to facilitate a subsequent search over the feasible set, as there is no unique solution unless the integer entries are bounded.  
In contrast,
we construct a lattice in which the unique solution vector is distinguished by its length, so that lattice reduction itself serves as the recovery procedure.
Finally,
our lattice basis incorporates an initial guess for the unknown signal and the analysis quantifies how the error of this guess affects recovery success.
To our knowledge, previous lattice reformulations of integer feasibility problems do not incorporate such a guess.

\subsection{Outline and Notation}
\label{sec:outline}
We begin in \cref{sec:background} by describing the reconstruction problem,
including the established uniqueness results and the general algorithmic framework.
In \cref{sec:lattice}, 
we present the lattice formulation, 
and examine the geometry of the lattice construction and the role of scaling parameters.  
This motivates the theoretical analysis developed in \cref{sec:theory},
which derives estimates for the parameter values required for successful recovery
and develops a probabilistic model that makes these estimates practical when the
initial guess error is unknown. In \cref{sec:guesses},
we extend the analysis to two alternative choices of the initial guess.
In \cref{sec:verify}, 
we compare the resulting predictions with extensive numerical experiments, 
demonstrating close agreement between theory and data. 
Finally, 
\cref{sec:full}
applies the results to full one- and two-dimensional reconstruction problems, before a brief discussion in \cref{sec:disc}. \\

Throughout the paper,
vectors are denoted by boldface lowercase letters (e.g., ${\bf x}$),
with entries $x_n$,
and matrices by typewriter-style uppercase letters (e.g., ${\tt X}$),
with entries $X_{mn}$.
We denote the $n\times n$ identity matrix by $\idmat{n}$.
The sets of real numbers,
complex numbers,
and integers are denoted by $\R$, $\C$, and $\Z$,
respectively.
We write $\Re(z)$ and $\Im(z)$ for the real and imaginary parts of $z\in\C$,
and denote the complex conjugate by $z^*$.
The complex unit is denoted by $\cu = \sqrt{-1}$.
For a vector ${\bf x}$,
the Euclidean norm is given by
$\norm{{\bf x}}=\sqrt{\sum_n \abs{x_n}^2}$.
We let $\#S$ be the cardinality of a set $S$.

For integers $m$ and $n$,
$\gcd(m,n)$ denotes their greatest common divisor
and $\lcm(m,n)$ their least common multiple.
We write $m\divs n$ if $m$ divides $n$,
otherwise $m\ndivs n$.
The cyclic group of order $n$ is denoted $\Z_n$,
$\langle k\rangle$ gives the cyclic subgroup generated by $k$,
and $\Z[\alpha]$ is the ring generated by $\alpha$ over $\Z$.
We also write $\tau(n)$ for the number of positive divisors of an integer $n$ and $\phi(n)$ for Euler's totient function,
which counts the number of positive integers less than $n$ and relatively prime with $n$.
Finally,
we define
$\uroot{n}{}=\exp\left(-2\pi i/n\right)$,
as a primitive $n$th root of unity.

\section{Problem Background and Previous Work} \label{sec:background} 

For a fixed dimension $d$,
let $0 < N_1, \dots, N_d \in \Z$, and define
$N = N_1 \cdots N_d$ to be the product of all the $N_j$.
Let
${\tt X} \in \C^{N_1 \times \dots \times N_d}$ be a $d$-dimensional signal.
The discrete Fourier transform (DFT) of ${\tt X}$ is defined by,
\begin{equation} \label{eq:dft}
    \tilde{X}_{\bf k} = 
    \sum_{n_1 = 0}^{N_1 - 1}\cdots \sum_{n_d = 0}^{N_d - 1} X_{n_1\cdots n_d} \uroot{N_1}{k_1n_1} \cdots \uroot{N_d}{k_dn_d}, 
\end{equation}
where ${\bf k}\in\Z^d$ is a multi-index frequency defined by
\begin{equation*}
     {\bf k} = (k_1, \dots, k_d),
    \text{ where } 0\le k_j < N_j 
    \text{ for each } 1 \le j \le d . 
\end{equation*}
The entries $\tilde{X}_{\bf k}$ of the DFT $\tilde{\tt X}$ are called the discrete Fourier coefficients (or DFT coefficients) of ${\tt X}$.
The DFT is a scaled unitary transformation,
which implies that ${\tt X}$ can be recovered entrywise from
\begin{equation} \label{eq:idft}
    X_{n_1 \cdots n_d} = 
    \frac{1}{N}\sum_{k_1 = 0}^{N_1 - 1}\cdots \sum_{k_d = 0}^{N_d - 1} \tilde{X}_{k_1,\dots,k_d} \uroot{N_1}{-k_1n_1} \cdots \uroot{N_d}{-k_dn_d}, 
    \qquad
    \text{for any } 0\le n_j < N_j,
\end{equation}
provided the full set of DFT coefficients is known.
While naive evaluations of \cref{eq:dft,eq:idft} take $\bigo{N^2}$ time,
the fast Fourier transform (FFT) algorithm computes each in $\bigo{N\log N}$ time~\cite{Cooley1965}.

A signal ${\tt X}$ is real
if and only if
its DFT coefficients satisfy the conjugate symmetry condition,
\begin{equation} \label{eq:conj_symmetry}
    \tilde{X}_{\bf k} = 
    \tilde{X}_{-\bf k}^*,
    \qquad
    \text{for all } {\bf k},
\end{equation}
where the notation $-{\bf k}$ indicates taking entrywise additive inverses modulo each $N_j$.
Since the integer-valued signals considered in this work are real,
\cref{eq:conj_symmetry} always holds.
As a consequence,
only about $N/2$ frequencies are required to invert ${\tt X}$ from \cref{eq:idft}.
In \cite{levinson2025recovery},
we show that the integer-valued constraint actually provides a much stronger guarantee on the number of coefficients required to uniquely determine ${\tt X}$.
This uniqueness result is stated in \cref{sec:sample_theory}, while \cref{sec:algs} describes an efficient method for recovering ${\tt X}$ 
from a partial sampling. 

\subsection{Uniqueness Results} \label{sec:sample_theory}
Let ${\tt X}\in\Z^{N_1\times \dots \times N_d}$ be an integer signal. 
For any DFT frequency
${\bf k} = (k_1,\dots,k_d)$,
define $D_j = N_j/\gcd(k_j,N_j)$ and $D=\lcm(D_1,D_2,\dots,D_d)$. 
With this setup, 
we can characterize when two DFT coefficients $\tilde{X}_{{\bf k}}$ and $\tilde{X}_{{\bf k}'}$ are redundant and contain the same information.  
\begin{lemma}[{\cite[\getcrefname{theorem}~3.3]{levinson2025recovery}}] \label{lem:uniq}
    For any frequency ${\bf k}$,
    two integer signals ${\tt X}$ and ${\tt Y}$ satisfy $\tilde{X}_{\bf k} = \tilde{Y}_{\bf k}$ if and only if $\tilde{X}_{{\bf k}'} = \tilde{Y}_{{\bf k}'}$ for every frequency ${\bf k}'$ such that there exists an integer $\lambda$ with $\gcd(\lambda,D)=1$ and
    \begin{equation}
    \label{eq:equivalence}
        k_j' \equiv \lambda k_j \pmod{N_j} \quad \text{for all } j = 1,\dots,d.
    \end{equation}
\end{lemma}
The condition in \cref{eq:equivalence} is equivalent to the frequencies ${\bf k}$ and ${\bf k'}$ generating the same cyclic subgroup of $\Z_{N_1}\times\dots\times\Z_{N_d}$. 
\Cref{lem:uniq} establishes that sampling a generator of each cyclic subgroup is sufficient to guarantee unique recovery of any element in $\Z^{N_1\times\cdots\times N_d}$.  
A slight extension of 
\cite[\getcrefname{theorem}~3.8]{levinson2025recovery}
shows that this sampling condition is also necessary to uniquely recover any integer signal,
as stated in the following lemma.
\begin{lemma}
\label{lem:min}
    Fix $\mathbf{k}$.
    For any ${\tt X}\in\Z^{N_1\times\cdots\times N_d}$, there exists a ${\tt Y}\in\Z^{N_1\times\cdots\times N_d}$ such that, for every frequency $\mathbf{k'}$, $\tilde{X}_{\mathbf{k'}}\neq\tilde{Y}_{\mathbf{k'}}$ if and only if $\langle\mathbf{k'}\rangle = \langle\mathbf{k}\rangle$.
\end{lemma}

\subsection{Reconstruction Methods} \label{sec:algs}

We now describe our reconstruction method for one-dimensional integer signals. 
As shown in \cite{levinson2025recovery},
a higher-dimensional problem may be reduced to a carefully chosen sequence of one-dimensional inversions.
The reduction decomposition is closely related to approaches used in the discrete Radon transform \cite{gertner2002new, hsung1996discrete,kingston2007generalised}, 
as well as in splitting signals \cite{grigoryan-superposition,grigoryan-book}.
Although higher-dimensional inversion requires careful organization of these subproblems, each individual subproblem has the same one-dimensional structure analyzed below. Thus, the subsequent one-dimensional analysis extends directly to the higher-dimensional reconstruction method.

The cyclic subgroups of $\Z_{N}$ are precisely the subgroups $\langle d \rangle \cong \Z_{N'}$, 
for each pair $dN' = N$. 
Denote the divisors of $N$ by $1 = N_{1}' < \dots < N_{\tau}' = N$,
with $d_j = N/N_j'$.
For an $N$-point integer signal ${\bf x}$,
\cref{lem:uniq} implies that
the linear system containing constraints for each sampled DFT coefficient,
\begin{equation} \label{eq:ip_form}
    \begin{NiceArray}{[ccc]c[c]c}
        \uroot{N}{0\cdot d_1} & \cdots & \uroot{N}{(N-1)d_1} & \Block{3-1}{{\bf x} = } & \tilde{x}_{d_1} & \Block{3-1}{, \qquad {\bf x} \in \Z^N,} \\
        \vdots & \ddots & \vdots & & \vdots \\
        \uroot{N}{0\cdot d_{\tau}} & \cdots & \uroot{N}{(N-1)d_{\tau}} & & \tilde{x}_{d_\tau}
    \end{NiceArray}
\end{equation}
has a unique solution over the integers.
We introduce an equivalent notation for the system in \cref{eq:ip_form} which will be used throughout the section,
\begin{equation} \label{eq:ip_notation}
    \begin{bmatrix}
        \uroot{N}{0\cdot d} & \cdots & \uroot{N}{(N-1)d}
    \end{bmatrix}_{d \divs N}
    {\bf x} =  
    \begin{bmatrix}
        \tilde{x}_{d}
    \end{bmatrix}_{d \divs N},
    \qquad {\bf x} \in \Z^N.
\end{equation}
\Cref{eq:ip_notation} indexes the linear constraints based on the divisor structure of $N$,
concatenating them into a matrix linear system.
While the uniqueness theory guarantees a unique solution to \cref{eq:ip_form},
standard approaches for solving this system, 
such as integer linear programming,
reduce to instances of NP-hard problems.  This is true 
even in the simplest case when ${\bf x}$ is known to be binary~\cite{karp1975computational}.

The algorithm proposed in our previous work~\cite[\getcrefname{algorithm}~4.2]{levinson2025recovery}, solves \cref{eq:ip_form} through a sequence of smaller subproblems, leveraging relationships between previously solved subproblems to obtain additional constraints. 
The intermediate subproblems recover frequency-decimated signals of the following form.
\begin{definition} \label{def:dec}
    Let ${\bf x}$ be a signal of length $N$.
    If $dN' = N$, 
    then the frequency decimated signal ${\bf x}^{(N')}$ (equivalently ${\bf x}^{(N/d)}$) is defined by,
    \begin{equation*}
        {\bf x}^{(N')} \coloneqq
        \begin{bmatrix}
        \idmat{N'} & \cdots & \idmat{N'}
        \end{bmatrix} {\bf x} 
        \qquad
        \iff
        \qquad
        \tilde{x}^{(N')}_k = \tilde{x}_{dk},
        \quad
        \text{ for all } 0\le k <N' .
    \end{equation*}
\end{definition}
\noindent
These frequency-decimated signals are equivalently given entry-wise by,
\begin{equation*}
    x^{(N')}_m = \sum_{n = 0}^{d - 1}x_{m + nN'},
    \qquad \text{for}~0\le m < N',
\end{equation*}
with the edge cases ${\bf x}^{(N)} = {\bf x}$ and ${\bf x}^{(1)} = \sum_{n = 0}^{N-1}x_n=\tilde{x}_0$~\cite[\getcrefname{lemma}~2.5]{levinson2025recovery}.

Each ${\bf x}^{(N')}$ is also an integer signal
which may be recovered from an analogous system to \cref{eq:ip_form}, 
where we only keep the rows that correspond to divisors $N''$ of $N'$,
\begin{equation} \label{eq:ip_form_subprob}
\begin{bmatrix}
\uroot{N'}{0 \cdot d} & \cdots & \uroot{N'}{(N'-1)\cdot d}
\end{bmatrix}_{ dN'' = N'}
\mathbf{x}^{(N')}
=
\begin{bmatrix}
\tilde{x}_{d}^{(N')}
\end{bmatrix}_{ dN'' = N'},
\qquad {\bf x}^{(N')} \in \Z^{N'} .
\end{equation}
Note that the DFT coefficients showing up in the RHS of \cref{eq:ip_form_subprob} are directly measured,
as \cref{def:dec} implies,
$\tilde{x}_{d}^{(N')} = \tilde{x}_{(N/N')d}$.
To improve computational efficiency,
we observe that if $dN'' = N'$, $d_1N' = N$ and $d_2N'' = N$,
we can write
\begin{equation*}
    \tilde{x}_{k}^{(N'')}
    = \tilde{x}_{d_2k}
    = \tilde{x}_{(N/N'')k}
    = \tilde{x}_{(N/N')k/(N''/N')}
    = \tilde{x}_{(d_1k/(1/d))}
    = \tilde{x}_{d_1(dk)}
    = \tilde{x}_{dk}^{(N')},
\end{equation*}
which implies from \cref{def:dec} that 
\begin{equation} \label{eq:subprob_dec}
    \begin{bmatrix}
        \idmat{N''} & \cdots & \idmat{N''}
    \end{bmatrix} 
    {\bf x}^{(N')}
    = {\bf x}^{(N'')} . 
\end{equation}
Therefore, 
the linear system in \cref{eq:subprob_dec} determines $\tilde{x}_{d}^{(N')}$, and can replace 
the corresponding row in \cref{eq:ip_form_subprob}.

Note that while this substitution is theoretically equivalent over the integers, it is computationally preferred for two main reasons.  First, it replaces a single Fourier equation with a collection of $N''$ linearly independent constraints, 
thereby increasing the rank of the system. 
Second,
the substitutions yield a linear system with integer coefficients. 
Consequently, 
the error of 
any integer vector ${\bf x}^{(N
')}$ that fails to satisfy the system in \cref{eq:subprob_dec} is bounded below by a constant (in any fixed norm).
In contrast, 
integer linear combinations of the complex irrational coefficients in the original Fourier system are dense in $\C$~\cite{myerson1986small,Buhler2000dense,barber2023small}. 
Thus, 
the modified system has better numerical stability, 
as it separates feasible integer solutions from infeasible ones by a uniform gap.  

Motivated by this substitution, our inversion approach solves the $\tau(N)$ subproblems \cref{eq:ip_form_subprob} in order of increasing length of the decimated signals, 
\begin{equation*} 
    {\bf x}^{(N'_1)}, \dots, {\bf x}^{(N'_\tau)}, 
\end{equation*}
which ensures that when solving for ${\bf x}^{(N')}$, we have already solved for ${\bf x}^{(N'')}$ on the right hand side of \cref{eq:subprob_dec} for every $N'' \divs N'$.
Furthermore,
it is actually sufficient to only consider divisors $N'/p$ a prime $p$. 
A proper divisor $N''$ of $N'$ must be divisible by $p$ for some prime $p \divs N'$,
so $\tilde{x}^{(N')}_{N'/N''}$ is specified by the constraints from ${\bf x}^{(N'/p)}$.
Thus,
an equivalent formulation to \cref{eq:ip_form} is
\begin{equation} \label{eq:ip_better}
    \begin{bmatrix}
        \begin{bmatrix}
            \idmat{N''} & \cdots & \idmat{N''} 
        \end{bmatrix}
    _{\substack{
     pN'' \divs N'\\
    p \text{ prime}
    }} \\
        \begin{bmatrix}
            \uroot{N'}{0} & \cdots & \uroot{N'}{N'-1}
        \end{bmatrix}
    \end{bmatrix}
    {\bf x}^{(N')}
    =
    \begin{bmatrix}
    \left( {\bf x}^{(N'')} \right)_{\substack{
    pN'' = N'\\
    p \text{ prime}
    }} \\
    \tilde{x}_{d}
    \end{bmatrix},
    \qquad {\bf x}^{(N')} \in \Z^{N'} ,
\end{equation}
where $d$ is defined by $N = dN'$ so $\tilde{x}_d=\tilde{x}_1^{(N')}$.
\Cref{alg:memo_1D} outlines this inversion method.  This divide-and-conquer formulation substantially reduces the dimension of the integer feasibility problems relative to solving \cref{eq:ip_form} directly~\cite{levinson2025recovery}.

\begin{algorithm}[!htb]
\caption{1D Inversion}
\label{alg:memo_1D}

\begin{algorithmic}[1]
    \Require DFT coefficients $\tilde{x}_{d}$ for each divisor $d$ of $N$
    \For{$N'\in\Call{Divisors}{N}$} \label{ln:memo1D_loop} \Comment{Includes $N$, iterate in increasing order.}
        \State $\set{p_1, \ldots, p_\omega} \gets \Call{PrimeFactors}{N'}$
        \State ${\bf x}^{(N')} \gets $ the solution to the ILP, \label{ln:memo1D_ip}
        \Statex \makebox{%
            $\begin{NiceArray}{[ccc]c[c]c}
                \idmat{N'/p_1} & \cdots & \idmat{N'/p_1} & \Block{4-1}{{\bf x} = } & {\bf x}^{(N'/p_1)} & \Block{4-1}{, \qquad {\bf x} \in \Z^{N'} .}\\
                \vdots & \ddots & \vdots & & \vdots & \\
                \idmat{N'/p_{\omega}} & \cdots & \idmat{N'/p_{\omega}} & & {\bf x}^{(N'/p_{\omega})} \\
                \cmidrule(lr){1-3}
                \cmidrule(lr){5-5}
                \uroot{N'}{0} & \cdots & \uroot{N'}{(N' - 1)} & & \tilde{x}_{N/N'} \\
            \end{NiceArray}$
        }
    \EndFor
    \State \Return ${\bf x}^{(N)}$ \label{ln:memo1D_ret}
\end{algorithmic}
    
\end{algorithm}
This reconstruction method assumes the minimal-data setting in which only the coefficients $\set{\tilde{x}_d \given d \divs N}$ are available. 
However,
the remainder of this work investigates a more general formulation of the inverse problem,
where,
for any $d \divs N$,
we may sample some number $J \ge 1$ of frequencies $k_j$ satisfying $\gcd(k_j, N) = d$ for each $1 \le j \le J$.
Providing additional DFT measurements beyond the single coefficient required for uniqueness improves stability and practicality of reconstruction for larger dimensions. 
Introducing these extra samples extends the subproblem ILP of \cref{alg:memo_1D} for $N' = \frac{N}{d}$ to,
\begin{equation} \label{eq:ip_better_J}
    \begin{NiceArray}{[ccc]c[c]c}
        \idmat{N'/p_1} & \cdots & \idmat{N'/p_1} & \Block{6-1}{{\bf x} = } & {\bf x}^{(N'/p_1)} & \Block{6-1}{, \qquad {\bf x} \in \Z^{N'} .}\\
        \vdots & \ddots & \vdots & & \vdots & \\
        \idmat{N'/p_{\omega}} & \cdots & \idmat{N'/p_{\omega}} & & {\bf x}^{(N'/p_{\omega})} \\
        \cmidrule(lr){1-3}
        \cmidrule(lr){5-5}
        \uroot{N'}{0 \cdot k_1/d} & \cdots & \uroot{N'}{(N' - 1) \cdot k_1/d} & & \tilde{x}_{k_1} \\
        \vdots & \ddots & \vdots  & & \vdots \\
        \uroot{N'}{0 \cdot k_J/d} & \cdots & \uroot{N'}{(N' - 1) \cdot k_J/d} & & \tilde{x}_{k_J} 
    \end{NiceArray} 
\end{equation}


\section{Lattice Setup} \label{sec:lattice}

This work uses lattice methods to solve the ILPs in the subproblems of \cref{alg:memo_1D}.
The lattice approach offers several benefits over directly applying standard branch-and-cut ILP solvers to \cref{eq:ip_better_J}.
ILP techniques for feasibility problems like \cref{eq:ip_better_J} require tight bounds on the integer variables to run quickly.
However,
even if our signal is constrained to be binary,
many of the subproblems of the algorithm will not be binary.
Additionally,
for many two-dimensional cases,
none of the subproblems of a binary image will have binary constraints.
This setup favors the lattice methods,
which do not incorporate any integer bounds and perform relatively agnostic of the integer values.

The lattice formulation also benefits from the incorporation of a (non-integer) guess for the inverted signal, 
as described below.
In contrast, 
we are not aware of how such a guess could improve the performance of an ILP method, 
beyond prescribing an initial point for a breadth-first exhaustive search.
Finally,
while ILP branch-and-cut algorithms have exponential complexity,
polynomial-time approximation algorithms are available for the lattice-based method.
We discuss our implementation in \cref{sec:implement},
which employs the LLL approximation algorithm for significant speed improvements.
Our previous work provides a benchmark comparison of the performance of ILP and LLL approaches for the subproblems,
and provides strong numerical evidence for the latter~\cite{levinson2025recovery}.

\subsection{Lattices and the LLL Algorithm} \label{sec:lattice_background}

Consider a linearly independent set of $d$ vectors given as the columns of a matrix 
${\tt B} = \begin{bmatrix}
    {\bf b}_0 & \cdots &  {\bf b}_{d-1}
\end{bmatrix} \in \R^{n \times d}$.
The set of all integer linear combinations of columns of ${\tt B}$ defines a $d$-dimensional lattice in $\R^n$,
\begin{equation*}
    \mathcal{L} = \set[\Big]{\sum_{j = 0}^{d-1} \alpha_{j}{\bf b}_j : {\boldsymbol \alpha} \in \Z^d}.
\end{equation*}
The set ${\tt B}$ is called a basis for lattice $\mathcal{L}$.
Since lattices are discrete subsets of $\R^n$~\cite{Galbraith_2012}, 
every lattice contains a shortest nonzero vector with respect to any fixed norm.  
The shortest vector problem (SVP) asks for such a vector in a given lattice.

As a generalization of SVP, 
the $i$th successive minima of $\mathcal{L}$,
denoted $\lambda_i(\mathcal{L})$ for $1 \le i \le d$,
is the radius of the smallest ball containing $i$ linearly independent vectors in $\mathcal{L}$.
In particular, $\lambda_1(\mathcal{L})$ is the length of the shortest nonzero vector in $\mathcal{L}$.  
The successive minima are related to the shortest independent vectors problem (SIVP),
which asks for a set of $i$ linearly independent lattice vectors whose maximum norm is as small as possible.  
Equivalently, SIVP seeks linearly independent vectors attaining the $i$th successive minimum $\lambda_i(\mathcal{L})$.  
The computational complexities of these lattice problems are open topics. In
\cite{ajtai1998shortest}, it is shown that SVP is NP-hard under randomized reductions,
while SVIP is also known to be NP-hard~\cite{blomer1999sivp}.
These complexity results extend to the analogous approximation problems for a fixed constant approximation factor~\cite{khot2005hardness,micciancio2001shortest,blomer1999sivp},
and certain complexity class conjectures imply that this is still true for all approximation factors which scale as a polynomial in $d$~\cite{khot2005hardness,AGGARWAL2021106065,GOLDREICH2000540}.

The basis ${\tt B}$ of a lattice is not unique. 
Although different bases generate the same lattice, 
their geometric properties can vary considerably, 
and many lattice algorithms perform substantially better when given a well-conditioned basis.
Nevertheless, 
the volume of the fundamental region of a lattice is invariant under the choice of basis.
This quantity, 
known as the lattice determinant,
is given by the square root of the determinant of the Gram matrix of any lattice basis,
\begin{equation} \label{eq:lattice_det}
    \Lambda(\mathcal{L}) = \sqrt{\det({\tt B}^T{\tt B})}.
\end{equation}
The length and tractability of a lattice basis are often related to its orthogonality defect,
which is defined as the ratio of the product of the basis vector lengths to the lattice determinant,
\begin{equation} \label{eq:ortho_defect}
    \frac{\prod_{i=1}^d \norm{{\bf b}_i}}{\Lambda(\mathcal{L})}.
\end{equation}
Basis reduction problems take a lattice basis as input and seek a reduced basis which is nearly orthogonal.
Finding a basis which minimizes the orthogonality defect in \cref{eq:ortho_defect} is NP-hard,
so practical lattice reduction algorithms instead seek efficient approximations. 

Lenstra, Lenstra, and Lov\'asz proposed an alternative basis reduction~\cite{Lenstra1982}.
A basis ${\tt B}$ is LLL-reduced for a parameter $0.25 < \delta < 1$ if it satisfies
\begin{equation*} 
    \begin{split}
        \abs{\mu_{ij}} &\le \frac{1}{2}, 
        \qquad \text{for all } 0 \le j < i < d,
        \qquad \qquad \qquad \mu_{ij} \coloneqq \frac{\langle{\bf b}_i, {\bf b}_j^* \rangle}{ \norm{{\bf b}_j^*}^2}\\
        (\delta - \mu_{i,i+1}^2)\norm{{\bf b}_i^*}^2 &\le \norm{{\bf b}_{i+1}^*}^2,
        \qquad \text{for all } 0\le i < d,
    \end{split}
\end{equation*}
where ${\tt B}^*$ is the Gram-Schmidt orthogonalization of ${\tt B}$.
The first projection condition ensures that each reduced basis vector is nearly orthogonal to all preceding Gram-Schmidt vectors. 
Consequently, 
the reduced basis vectors cannot be substantially shortened by subtracting integer multiples of earlier basis vectors.  
The second ordering condition roughly sorts the basis vectors by increasing length,
preventing the basis vector lengths from decreasing too rapidly.
Without such an ordering, a basis could contain unnecessarily long vectors,
as projecting a later short vector onto an earlier long vector would easily satisfy the projection condition. 
Together, these conditions guarantee that the basis is reasonably short and nearly orthogonal.

The LLL reduced basis approximates both orthogonality defect minimization and SIVP 
with approximation factors that depend only on the lattice dimension~\cite[\getcrefname{theorem}~9]{Nguyen2010hermite}.
For SIVP,
the norm of the $i$th LLL-reduced basis vector approximates the $i$th successive minima
\begin{equation} \label{eq:lll_approx}
    \norm{{\bf b}_i}
    \le \left( \frac{4}{4\delta - 1} \right)^{(d-1)/2}
    \lambda_i(\mathcal{L}).
\end{equation}
The primary advantage of working with this basis reduction is that the LLL algorithm computes it in polynomial time.
While other basis reduction problems, 
such as Hermite and Hermite–Korkine–Zolotarev, 
provide tighter guarantees on basis lengths and orthogonality,
no polynomial-time algorithms are known for computing such reductions~\cite{Nguyen2010hermite,hanrot2008worstcasehermitekorkinezolotarevreducedlattice}.

Our work primarily uses the LLL algorithm for numerical reconstructions, allowing \cref{alg:memo_1D} to run in pseudo-polynomial time.  To isolate the intrinsic difficulty of the reconstruction problem, our later analysis focuses on the geometry of the resulting lattice rather than the approximation guarantees of any particular lattice reduction algorithm.  In particular, we estimate how many lattice vectors are shorter than the target vector to be recovered.  This characterizes the stability of the integer linear system \cref{eq:ip_better_J} itself, rather than the stability of applying LLL to the recovery problem. 
We also note that the approximation bound in \cref{eq:lll_approx} is not tight for many common lattices~\cite{Aardal2000}.  Thus it is reasonable to treat lattice reduction as a black-box subroutine and focus instead on the underlying lattice geometry.   

Our analysis is also distinct from SIVP in a subtle but important way.  By estimating the number of lattice vectors shorter than the target vector, we obtain a condition for recovery: the target vector must appear among the $m$ shortest lattice vectors.  Unlike SIVP, this condition does not require the short vectors to be linearly independent.  Therefore, the recovery problem studied here is not directly approximated by LLL.  Nevertheless,
our numerical experiments suggest that the number of lattice vectors shorter than the target vector provides a useful model for the practical behavior of LLL on our reconstruction problem, while  also
highlighting situations in which this approximation breaks down. 

\subsection{Problem Formulation}

\def\coeffs{\begin{bmatrix} {\bm \alpha} & \gamma \end{bmatrix}^T}
\def\coeffsone{\begin{bmatrix} {\bm \alpha} & 1 \end{bmatrix}^T}

We formulate the subproblem ILP of \cref{alg:memo_1D} given in \cref{eq:ip_better_J} as a lattice problem by constructing the basis,
\begin{equation} \label{eq:lattice_basis}
    {\tt B} 
    = 
    \begin{bNiceArray}{ccc|c}
        {\bf b}_{0} & \cdots & {\bf b}_{N - 1} & {\bf b}_{N} \\
    \end{bNiceArray} 
    =\begin{NiceArray}{[ccc|c]c[c]}
        \Block{3-3}{\idmat{N}} & & & \Block{3-1}{-\overline{\bf x}} & \Block{10-1}{=} & \Block{3-1}{{\tt A}} \\
        \\ \\
        \cmidrule(lr){1-4}\cmidrule(lr){6-6}
        0 & \cdots & 0 & \beta_0 & & \beta_0{\tt B}_0 \\
        \cmidrule(lr){1-4}\cmidrule(lr){6-6}
        \beta_1 \idmat{N/p_1} & \cdots & \beta_1 \idmat{N/p_1} & -\beta_1 {\bf x}^{N/p_1} & & \Block{3-1}{\beta_1{\tt B}_1} \\
        \vdots & \ddots & \vdots & \vdots & & \\
        \beta_1 \idmat{N/p_\omega} & \cdots & \beta_1 \idmat{N/p_\omega} & -\beta_1 {\bf x}^{N/p_\omega} & & \\
        \cmidrule(lr){1-4}\cmidrule(lr){6-6}
        \beta_2 \uroot{N}{0\cdot k_1} & \cdots & \beta_2 \uroot{N}{(N - 1)k_1} & -\beta_2 \tilde{x}_{k_1} & & \Block{3-1}{\beta_2{\tt B}_2} \\
        \vdots & \ddots & \vdots & \vdots & & \\
        \beta_2 \uroot{N}{0\cdot k_J} & \cdots & \beta_2 \uroot{N}{(N - 1)k_J} & -\beta_2 \tilde{x}_{k_J} & & \hphantom{\beta_1{\bf x}^{(N/p_\omega)}}
    \end{NiceArray}.
\end{equation}
\Cref{eq:lattice_basis} uses $N$ instead of $N'$ as in \cref{eq:ip_better_J}.
The symmetry of the subproblems ensures that these formulations are equivalent,
so we may focus our lattice exposition on the "top-level" problem where $N' = N$ for simplicity.
In \cref{eq:lattice_basis}, $\beta_0$, $\beta_1$, and $\beta_2$ are parameters, and $\overline{\bf x}$ is an initial guess. 
Various works have successfully applied bases similar to \cref{eq:lattice_basis} to solve subset sum and knapsack problems~\cite{Lagarias1985} and
linear diophantine systems~\cite{Aardal2000}.
However, in these works,  the guess $\overline{\bf x}$ is always set to $\overline{\bf x}={\bf 0}$.

For the remainder of this work,
we refer to the lattice with basis ${\tt B}$ in \cref{eq:lattice_basis} by $\mathcal{L}$,
understanding that $\mathcal{L}$ is parameterized by $N$, ${\bf x}$, $\overline{\bf x}$, the $k_j$s, and the $\beta$s.
For consistency with the real-valued basis vectors in the lattice definition from \cref{sec:lattice_background},
in practice we use the equivalent final basis block,
$\beta_2\begin{bmatrix}
    \Re [{\tt B}_2] \\ \Im [{\tt B}_2]
\end{bmatrix}$,
which stacks the real and imaginary parts of $\beta_2{\tt B}_2$.
However,
the 2-norm of a lattice vector is invariant with respect to these representation choices,
\begin{equation*}
    \abs{\ell_n}^2 = \abs{\Re\ell_n}^2 + \abs{\Im\ell_n}^2,
\end{equation*}
which justifies using the complex shorthand formulation in \cref{eq:lattice_basis} for our analysis.

Any vector ${\bm \ell} \in \mathcal{L}$ in the lattice may be written as ${\tt B}\coeffs$ for coefficients ${\bm \alpha} \in \Z^{N}$ and $\gamma \in \Z$. 
Analogous to the block structure of the basis
${\tt B}$,
we express the lattice vectors in block structure by, 
$\bm{\ell} = \begin{bmatrix} \bm{\ell}^{({\tt A})} & \bm{\ell}^{({\tt B}_0)} & \bm{\ell}^{({\tt B}_1)} & \bm{\ell}^{({\tt B}_2)} \end{bmatrix}$, 
where $\bm{\ell}^{({\tt A})} = {\tt A}\coeffs$ and $\bm{\ell}^{({\tt B}_t)} = \beta_t{\tt B}_t\coeffs$ for $t=0,1,2$.  
The blocks $\bm{\ell}^{({\tt A})}$ and $\bm{\ell}^{({\tt B}_0)}$ may be computed from \cref{eq:lattice_basis},
\begin{equation} \label{eq:lAB0}
{\bm \ell}^{({\tt A})}= {\bm \alpha}-\gamma\overline{\bf x}  \qquad , \qquad
{\bm \ell}^{({\tt B}_0)}= \begin{bmatrix}
    \gamma\beta_{0}
\end{bmatrix} .
\end{equation}
The vector $\bm{\ell}^{({\tt B}_1)}$ itself has a natural block structure based on the prime factors $p_1, \ldots, p_\omega$ of $N$. 
We write 
$\bm{\ell}^{({\tt B}_1)} 
= \begin{bmatrix} \bm{\ell}^{(p_1)} & \cdots & \bm{\ell}^{(p_\omega)} \end{bmatrix}$, 
where each component $\bm{\ell}^{(p_t)}$ is given by
\begin{equation}
 \label{eq:ellP_def}
    {\bm\ell}^{(p_t)}
    = \beta_{1}\left({\bm\alpha}^{(N/p_t)} -\gamma {\bf x}^{(N/p_t)}\right) .
\end{equation}
Lastly, $\bm{\ell}^{({\tt B}_2)}$ is defined entrywise by
\begin{align}
\label{eq:B2_err}
\ell^{({\tt B}_2)}_{j}
= \beta_{2}\left(\tilde{\alpha}_{k_j}-\gamma\tilde x_{k_j} \right),
\qquad \text{for } 1 \le j \le J .
\end{align}
For any lattice vector,
the coefficients $\coeffs$ can easily be obtained from 
${\bm \ell}^{({\tt A})}$ and ${\bm \ell}^{({\tt B}_0)}$ 
using \cref{eq:lAB0}.
Our recovery approach reduces the basis in \cref{eq:lattice_basis} with the LLL algorithm to obtain short,
nearly orthogonal
vectors in the lattice,
and we hope to recover the true signal ${\bf x}$ from one of these short vectors.
The choice of $\beta$ parameters ideally guarantees that the vector with lattice coefficients $\coeffs=\pm\begin{bmatrix} {\bf x} & 1 \end{bmatrix}^T$
is sufficiently short to appear in the LLL reduced basis,
as we can then recover ${\bf x}$ from the lattice coefficients. 
This desired lattice vector is denoted by ${\bm \ell}^*$
and is related to the original signal ${\bf x}$ by,
\begin{equation*} 
    {\bm \ell}^* = {\tt B}
    \begin{bmatrix}
        {\bf x} \\
        \cmidrule(lr){1-1}
        1
    \end{bmatrix} =
    \begin{NiceArray}{[c]c[c]}
        \Block{3-1}{\tt A} & 
        \Block{10-1}{
            \begin{bmatrix}
                {\bf x} \\
                \cmidrule(lr){1-1}
                1
            \end{bmatrix} =
        } & x_0 - \overline{x}_0 \\
        & & \vdots \\
        && x_{N - 1} - \overline{x}_{N - 1} \\
        \cmidrule(lr){1-1}
        \cmidrule(lr){3-3}
        \beta_0{\tt B}_0 & & \beta_0\\
        \cmidrule(lr){1-1}
        \cmidrule(lr){3-3}
        \Block{3-1}{\beta_1{\tt B}_1} & & 0 \\
        & & \vdots\\
        & & 0 \\
        \cmidrule(lr){1-1}
        \cmidrule(lr){3-3}
        \Block{3-1}{\beta_2{\tt B}_2} & & 0 \\
        & & \vdots \\
        \hphantom{x_{N - 1} - \overline{x}_{N - 1}} & & 0 \\
    \end{NiceArray} 
\end{equation*}
We can directly compute the length of ${\bm \ell}^*$ as 
 \begin{equation} \label{eq:shortest_vec}
    \norm{{\bm \ell^*}}
    = \sqrt{\sum_{n=0}^{N - 1}\abs{x_{n}-\overline x_{n}}^{2}+\beta_{0}^{2}}
    = \sqrt{K^{2}+\beta^{2}_{0}} , \qquad \text{where } K=\norm{{\bf x}-\overline{\bf x}}.
\end{equation}
Our intuitive goal is to choose the $\beta$ parameters such that there are not too many lattice vectors shorter than $\sqrt{K^2+\beta_0^2}$.

For a general lattice vector with coefficient
$\gamma = 1$,
$\begin{bmatrix}  \bm{\ell}^{({\tt B}_1)}/\beta_1 & \bm{\ell}^{({\tt B}_2)}/\beta_2 \end{bmatrix}$ is exactly the difference between the two sides of \cref{eq:ip_better_J}.
By \cref{lem:uniq}, 
${\bm \alpha} = {\bf x}$ if and only if this difference is zero.
The parameters $\beta_1$ and $\beta_2$ thus act as penalties for the coefficients ${\bm \alpha}$ not satisfying \cref{eq:ip_better_J} exactly.
Scaling the system by $\beta_1$ and $\beta_2$ also scales this difference,
making any nonzero vector longer.  
In general,
when ${\bm \ell}^{({\tt B}_1)}={\bf 0}$,
we say that ${\bm \ell}$ satisfies the ${\tt B}_1$ constraints (and similarly for ${\tt B}_2$). While the parameters $\beta_1$ and $\beta_2$ have similar roles of enforcing solutions to the linear system,
they are differentiated by the type of constraints in the respective blocks.
As discussed in \cref{sec:algs},
the integer equations of the ${\tt B}_1$ block are numerically stable,
while the complex equations of the ${\tt B}_2$ block are inherently unstable.
Thus,
we expect smaller values of $\beta_1$ will enforce that the ${\tt B}_1$ block of short lattice vectors is exactly 0,
while very large values of $\beta_2$ may be required for the ${\tt B}_2$ block to be sufficiently small.

We will often also specify the corresponding value of the coefficient $\gamma$.  
Although $\abs{\gamma}=1$ is required for ${\bm \ell}=\pm{\bm \ell}^*$, a lattice vector may satisfy the ${\tt B}_1$ or ${\tt B}_2$ constraints for other values of $\gamma$. 
However, 
if ${\bm \ell}$ satisfies both ${\tt B}_1$ and ${\tt B}_2$ constraints, 
then  \cref{lem:uniq} implies that ${\bm \ell}=\gamma{\bm \ell}^*$.  
Such vectors should not appear in the reduced lattice for $\abs{\gamma} > 1$,
since they are longer than ${\bm \ell}^*$.
The last parameter $\beta_0$ acts as a penalty on the size of $\abs{\gamma}$, 
as it directly scales the $\bm{\ell}^{({\tt B}_0)}$ entry of the lattice vector.
A large value of $\beta_0$ prevents $\abs{\gamma}$ from being too large for short lattice vectors,
helping select the desired vector with $\gamma = 1$.
The following sections analyze the values of $\beta_0$, $\beta_1$, and $\beta_2$ required for successful recovery with LLL.

The basis ${\tt B}$ also incorporates a guess $\overline{\bf x}$ for the signal.
For the vectors ${\bm \ell}$ with $\gamma = \pm 1$,
the $\bm{\ell}^{({\tt A})}$ block in \cref{eq:lAB0} has a small norm when the coefficients ${\bm \alpha}$ are close to the guess, 
$\pm \overline{\bf x}$.
An accurate guess thus favors lattice coefficients which are close to the true signal.
A natural guess $\overline{\bf x}$ can be constructed from the limited set of known DFT coefficients.
As discussed in relation to \cref{eq:ip_better_J},
the divide-and-conquer strategy gives the final iteration of \cref{alg:memo_1D} access to all DFT coefficients $\tilde{ x}_{k}$ with $\gcd(k, N) \ne 1$.
These $N - \phi(N)$ DFT coefficients from the decimated signals combine with the $2J$ frequencies $\pm k_1, \ldots, \pm k_J$ to form a set of $N - \phi(N) + 2J$ known frequencies,
\begin{equation} \label{eq:sampled_freq}
    F = \set{\pm k_1, \ldots, \pm k_J}\cup \set{k : \gcd(k, N) > 1}.
\end{equation}
The guess $\overline{\bf x}$ is then defined as the least-norm signal that aligns with the known DFT data.
By the Parseval relation, 
this
corresponds to filling in the unknown DFT coefficients with 0s,
so $\overline{\bf x}$ is given in frequency space by,
\begin{equation} \label{eq:guess}
    \tilde{\overline{x}}_k = 
    \begin{cases}
        \tilde x_k & k \in F \\
        0 & k \notin F .
    \end{cases}
\end{equation}
Note that $\overline{\bf x}$ cannot be an integer signal,
provided $\tilde{x}_1 \ne 0$,
as this would violate \cref{lem:uniq}.
\cref{sec:theory}
quantifies the benefit of incorporating a close guess for $\overline{\bf x}$, 
and \cref{sec:guesses} looks at two additional guess strategies.

\subsection{Implementation Details} \label{sec:implement}

\Cref{alg:memo_1D} was implemented in pure Python leveraging the fast vectorization provided by NumPy~\cite{harris2020array}.
To solve the ILP through the lattice formulation of \cref{eq:lattice_basis}, 
we use the floating point LLL algorithm (FPLLL) \cite{nguyen2009lll,stehle2009floating} included in the fplll library~\cite{fplll} with Python interface fpylll~\cite{fpylll}. 
The FPLLL algorithm replaces exact rational arithmetic with floating-point computations, 
accelerating lattice basis reduction.
This significantly improves speed, but introduces the challenge of accumulated rounding errors. 
FPLLL balances efficiency and stability with adaptive precision and error monitoring, 
achieving a reduced basis quality comparable to exact LLL while being much faster in practice. 
In particular,
the algorithm has runtime bound,
\begin{equation} \label{eq:lll_runtime}
    \bigo{d^4n(d + \log B)\log B},
\end{equation}
when applied to the basis of a $d$-dimensional lattice in $\R^n$ with basis vector norms bounded by $B$,
while
the runtime of exact LLL is cubic in $\log B$~\cite{nguyen2009lll}.
This efficiency allowed us to recover larger signals than were feasible in previous work using lattice methods~\cite{levinson2023}.

The FPLLL libraries operate on integer bases, 
whereas the ${\tt B}_2$ block of the lattice in \cref{eq:lattice_basis} contains irrational real entries.
To handle this discrepancy, 
our algorithm implementation introduces an additional parameter $\beta_3$.
We multiply the entire basis ${\tt B}$ by $\beta_3$ and then truncate the scaled basis, 
before passing $\floor{\beta_3{\tt B}}$ to the FPLLL algorithm.
Truncation limits the effective precision in the data by discarding trailing digits of the DFT coefficients.
As both $\beta_2$ and $\beta_3$ scale ${\tt B}_2$ before truncation, 
the number of preserved digits past the decimal point is approximately $\log_{10}(\beta_2\beta_3)$. 
After sufficient testing, 
we set $\beta_3=\max\set{1 \times 10^2, \frac{1}{\beta_0}}$ in our numerical simulations. 

This interpretation of precision assumes exact knowledge of the DFT coefficients. 
In practice however, the DFT coefficient data has finite precision.
While our theoretical results guarantee that two distinct integer signals cannot have identical sampled spectra, 
the spectra may be indistinguishable at a certain precision level.
Thus, 
standard single or even double precision data may not suffice for unique recovery in some cases.
To explore the limits of the lattice method and enable recovery of larger signals, 
we conduct some numerical experiments with mixed precision floating point arithmetic using the Python module gmpy2 module.
As the lattice entries become larger in magnitude,
we also require more precision in the floating point computations for LLL,
so we leveraged the mixed precision functionality of the fpylll library during lattice reduction.

\Cref{alg:svp} presents our implementation of the lattice-based solution to \cref{eq:ip_better_J}.
The algorithm uses a numerical parameter $\epsilon$ to determine whether a candidate integer signal matches the DFT data to a specified tolerance.
We first check whether lattice reduction is required in Line~\ref{ln:guess_check}.
If $N = 1, 2, 3, 4,$ or 6,
all DFT coefficients are available from the minimal set.
Additionally,
if the sampled DFT coefficients are within the tolerance of 0,
\cref{lem:uniq} immediately applies $\tilde{x}_k = 0$ for all $\gcd(k, N) = 1$.
These cases require no lattice reduction,
since the guess is immediately correct with $\overline{\bf x} = {\bf x}$.

\begin{algorithm}[!htb] \caption{Solving \cref{eq:ip_better_J} with LLL} \label{alg:svp}
\begin{algorithmic}[1] 

\Require Error tolerance $\epsilon$, scale $\beta_3$

\If{$\phi(N) \le 2$ or $\displaystyle\max_{j=1}^J\abs{\tilde{x}_k} < \epsilon$}
    \State \Return $\overline{\bf x}$
    \label{ln:guess_check}
\EndIf

\State Construct the lattice basis ${\tt B}$ in \cref{eq:lattice_basis} \label{ln:basis}
\State ${\tt B}^* \gets$ the LLL reduction of $\floor{\beta_3{\tt B}}$ \label{ln:LLL}

\For{${\bm \ell} \in {\tt B}^*/\beta_3$}
    
    \State ${\bf y} \gets \round{\pm{\bm \ell}^{({\tt A})} + \overline{\bf x}}$ \label{ln:round}
    \If{$\ell_{N} = \pm\beta_0$ and $\displaystyle\max_{j=1}^J\abs{\tilde{y}_{k} - \tilde{x}_{k}} \le \epsilon$} \label{ln:check}
        \State \Return ${\bf y}$
    \EndIf
\EndFor
\end{algorithmic}
\end{algorithm}

The remainder of \cref{alg:svp} recovers ${\bf x}$ using FPLLL.
Lines~\ref{ln:basis} and~\ref{ln:LLL} construct the integer lattice basis and reduce it with FPLLL, 
yielding an LLL-reduced basis with parameter $\delta = 0.9972$.
For each vector in the reduced basis, we round its entries in Line~\ref{ln:round} before checking against the DFT data.
This rounding handles the errors introduced by rescaling and truncating ${\tt B}$ to $\floor{\beta_3{\tt B}}$.  
Our chosen value of $\beta_3$ is sufficiently large to ensure this rounding does not introduce any errors.
Finally, the check step in Line~\ref{ln:check} includes the additional condition $\ell_N = \pm\beta_0$,
as the vector ${\bm \ell} = {\tt B}\coeffs$ satisfies $\ell_N = \pm\beta_0$ if and only if  $\gamma = \pm1$.  
The full implementation of minimal DFT sampling and the one-dimensional and two-dimensional inversion algorithms is available at~\cite{intvert}.

\subsection{Reduced Lattice Structure}
\label{sec:lattice_ex}

The work of \cite{Aardal2000} characterizes the structure of the LLL reduction on a general lattice basis similar in structure to \cref{eq:lattice_basis}. 
However,
their framework omits ${\tt B}_2$ block in \cref{eq:lattice_basis} (and thus the $\beta_2$ parameter), 
and sets the guess to $\overline{\bf x} = {\bf 0}$. 
Consider a general integer-valued ${\tt B_1}$ block,
and let ${\bf b}_{N}^{({\tt B}_1)}$ denote the last column of this ${\tt B}_1$ block, 
and let 
$\overline{{\tt B}}_1 \coloneqq \begin{bmatrix}
    {\bf b}_0^{({\tt B}_1)} & \cdots & {\bf b}_{N-1}^{({\tt B}_1)}
\end{bmatrix}$ 
be the submatrix of ${\tt B}_1$ obtained by removing this last column ${\bf b}_{N}^{({\tt B}_1)}$.  
Let $m$ be the dimension of the null space of $\overline{{\tt B}}_1$.  

For sufficiently large values of $\beta_0$ and $\beta_1$, it is proven in \cite{Aardal2000} that the first $m$ lattice vectors  ${\bm \ell}_0^{({\tt A})},\dots,{\bm \ell}^{({\tt A})}_{m-1}$ in the LLL-reduced basis form an integer basis for the null space of the matrix  $\overline{{\tt B}}_1$. 
Equivalently, 
these lattice vectors
${\bm \ell}_0, \dots, {\bm \ell}_{m-1}$
all satisfy the ${\tt B}_1$ constraints with $\gamma=0$. 
It is also shown that the next lattice vector ${\bm \ell}^{({\tt A})}_{m}$ gives a particular integer solution to the linear system $\overline{{\tt B}}_1{\bf y} = {\bf b}_{N}^{({\tt B}_1)}$, 
or equivalently, 
satisfies the ${\tt B}_1$ constraints with $\gamma = 1$.  

For our specific ${\tt B}_1$ block in \cref{eq:lattice_basis}, 
the dimension of the null space of $\overline{{\tt B}}_1$ is $m=\phi(N)$.
To demonstrate how the addition of the ${\tt B}_2$ block interacts with the aforementioned structure,
we consider an example of the basis in \cref{eq:lattice_basis} for the  signal 
${\bf x} = \begin{bmatrix}
    1 & 2 & 3 & 0 & 1 & 4 & 1
\end{bmatrix}$  of length $N=7$
with no guess ($\overline{\bf x} = {\bf 0}$).
Below is the LLL reduction using the implementation described in the previous section, 
\def\compactcolumns{\setlength{\arraycolsep}{2pt}}
\begin{equation} \label{eq:reduced_no_guess}
\begingroup \compactcolumns
\begin{bNiceArray}[last-col]{S[table-format=+1.2]S[table-format=+1.2]S[table-format=+1.2]S[table-format=+1.2]S[table-format=+1.2]S[table-format=+1.2]|S[table-format=1.2]|S[table-format=+1.2]}
2.00 & 1.00 & -1.00 & 2.00 & 2.00 & -2.00 & 1.00 & 2.00
 & \Vbrace{7}{~{\tt A}}
\\
-2.00 & -2.00 & -2.00 & -2.00 & -3.00 & -1.00 & 2.00 & -1.00
\\
1.00 & 2.00 & 3.00 & 0.00 & 0.00 & 1.00 & 3.00 & -1.00
\\
-1.00 & 0.00 & 0.00 & 2.00 & 3.00 & 2.00 & 0.00 & 2.00
\\
2.00 & -2.00 & -3.00 & -2.00 & -2.00 & -3.00 & 1.00 & 0.00
\\
-2.00 & 2.00 & 2.00 & 1.00 & -1.00 & 0.00 & 4.00 & -1.00
\\
0.00 & -1.00 & 1.00 & -1.00 & 1.00 & 3.00 & 1.00 & 0.00
\\
\cmidrule(lr){1-8}
{0} & {0} & {0} & {0} & {0} & {0} & {\beta_0} & {0} & \Vbrace[line-style=brace1]{1}{\hspace{0.1em}\smash[b]{{\tt B}_0}} \\
\cmidrule(lr){1-8}
{0} & {0} & {0} & {0} & {0} & {0} & {0} & {\beta_1}
 & \Vbrace[line-style=brace1]{1}{\hspace{0.1em}\smash[b]{{\tt B}_1}}
\\
\cmidrule(lr){1-8}
{\phantom{-}0.07\beta_2} & {\phantom{-}0.04\beta_2} & {-0.03\beta_2} & {-0.09\beta_2} & {\phantom{-}0.07\beta_2} & {-0.07\beta_2} & {0} & {\phantom{-}0.02\beta_2}
 & \Vbrace{2}{~\smash[b]{{\tt B}_2}}
\\
{-0.06\beta_2} & {-0.09\beta_2} & {\phantom{-}0.07\beta_2} & {\phantom{-}0.02\beta_2} & {-0.02\beta_2} & {-0.02\beta_2} & {0} & {-0.09\beta_2}
\\
\end{bNiceArray}
\endgroup
\end{equation}

The basis in \cref{eq:reduced_no_guess} shows the same structure as described in~\cite{Aardal2000}.
The vertical delimiters separate the first $\phi(N)$ vectors,
the $(\phi(N) + 1)$-th vector,
and the remaining basis vectors.
We observe that these first $\phi(N) + 1=7$ vectors all satisfy the ${\tt B}_1$ block constraints exactly.  
The ${\tt A}$ blocks of the first $\phi(N)=6$ lattice vectors are in the null space of $\overline{\tt B}_1$,
while the 7th lattice vector has $\gamma = 1$.
Although the basis in \cite{Aardal2000} does not include a ${\tt B}_2$ block, 
we were able to select a value of $\beta_2$ sufficiently large  to recover the signal.  
Consequently, 
the 7th reduced lattice vector is ${\bm \ell}^*$ rather than some other particular solution to the ${\tt B}_1$ constraints with $\gamma =1$. 
This is confirmed by noting that the 7th reduced lattice vector ${\bm \ell}_6$ satisfies ${\bm \ell}_6^{({\tt A})} = {\bf x}$ and is the only vector in the reduced basis to satisfy ${\bm \ell}^{({\tt B}_2)} = {\bf 0}$.
We note that the remaining $N - \phi(N)=1$ basis vector in the reduced lattice does not satisfy the ${\tt B}_1$ constraints.  This follows from 
the linear independence of the basis. 

With the inclusion of a guess $\overline{\bf x}$,
the reduced basis no longer structurally resembles the reduced basis from \cite{Aardal2000}. 
For the same signal and $\beta$ parameter values as in \cref{eq:reduced_no_guess}, 
but using the guess as defined in \cref{eq:guess},
the LLL-reduced basis is given by
\def\compactcolumns{\setlength{\arraycolsep}{2pt}}
\begin{equation} \label{eq:reduced_guess}
\begingroup \compactcolumns
\begin{bNiceArray}[last-col]{S[table-format=+1.2]S[table-format=+1.2]S[table-format=+1.2]S[table-format=+1.2]S[table-format=+1.2]S[table-format=+1.2]|S[table-format=+1.2]|S[table-format=+1.2]}
0.01 & -0.83 & 0.17 & -0.32 & 1.66 & 2.32 & -0.80 & 0.16
 & \Vbrace{7}{~{\tt A}}
\\
1.12 & 0.35 & -0.65 & -0.59 & -0.71 & -1.41 & -0.41 & 0.24
\\
-1.08 & 1.49 & -0.51 & 0.95 & 1.03 & -0.95 & 1.33 & -0.44
\\
-0.18 & -1.53 & 1.47 & -0.12 & -1.94 & 2.12 & -0.89 & 0.65
\\
0.89 & -0.68 & -0.68 & -0.74 & 1.37 & -1.26 & 0.10 & 0.42
\\
-0.17 & 2.14 & -0.86 & 0.55 & 0.72 & 0.45 & -0.21 & -0.69
\\
-0.56 & -0.93 & 1.07 & 0.29 & -2.15 & -1.29 & 0.95 & 0.64
\\
\cmidrule(lr){1-8}
{6\beta_0} & {\beta_0} & {\beta_0} & {4\beta_0} & {-2\beta_0} & {-4\beta_0} & {13\beta_0} & {-5\beta_0} & \Vbrace[line-style=brace1]{1}{\hspace{0.1em}\smash[b]{{\tt B}_0}} \\
\cmidrule(lr){1-8}
{0} & {0} & {0} & {0} & {0} & {0} & {0} & {\beta_1}
 & \Vbrace[line-style=brace1]{1}{\hspace{0.1em}\smash[b]{{\tt B}_1}}
\\
\cmidrule(lr){1-8}
{0} & {0} & {\phantom{-}0.03\beta_2} & {-0.07\beta_2} & {\phantom{-}0.01\beta_2} & {-0.03\beta_2} & {0} & {-0.02\beta_2}
 & \Vbrace{2}{~\smash[b]{{\tt B}_2}}
\\
{\phantom{-}0.03\beta_2} & {0} & {\phantom{-}0.07\beta_2} & {\phantom{-}0.03\beta_2} & {\phantom{-}0.02\beta_2} & {-0.01\beta_2} & {-0.02\beta_2} & {-0.03\beta_2}
\\
\end{bNiceArray}
\endgroup
\end{equation}

The first $\phi(N) + 1$ lattice vectors in \cref{eq:reduced_guess} still satisfy the ${\tt B}_1$ constraints,
while the remaining $N-\phi(N)$ vectors do not. 
However,
there is no longer a clear distinction between null space basis vectors and a particular solution to the ${\tt B}_1$ constraints. 
In fact in \cref{eq:reduced_guess}, every lattice vector has a nonzero coefficient $\gamma$,
as ${\bm \ell}^{({\tt B}_0)}\ne 0$.
The chosen parameter values were again sufficient for correct recovery,
as ${\bm \ell}^*$ appears as the second reduced basis vector ${\bm \ell}_{1}$. 
This is confirmed by noting that it is the only reduced basis vector which satisfies
${\bm \ell}^{({\tt B}_1)}={\bm \ell}^{({\tt B}_2)} = {\bf 0}$.
As we have included a guess (and $\gamma = 1$ for this vector), 
${\bf x}$ can be recovered from 
the corresponding block ${\bm \ell}_{1}^{({\tt A})}={\bf x} - \overline{\bf x}$.

The two reduced bases in \cref{eq:reduced_no_guess,eq:reduced_guess} illustrate several structural properties that are central to our  later analysis.
In particular, 
obtaining sharp estimates for sufficient values of $\beta_1$ and $\beta_2$ requires understanding how and where the desired lattice vector ${\bm \ell}^*$ can appear in the reduced basis. 
Moreover, 
we would like to choose $\beta_2$ as small as possible, 
since excessively large values compound numerical precision issues, 
as discussed in the next subsection.

In \cref{sec:beta1},
we derive a condition on $\beta_1$ that guarantees any lattice vector shorter than 
${\bm \ell}^*$ satisfies 
the $B_1$ constraints.
When a guess $\overline{\bf x}$ is included, 
${\bm \ell}^*$
may appear as any of these first $\phi(N) + 1$ indices in the reduced lattice. 
In contrast, 
without a guess, 
there is exactly one reduced basis vector with $\gamma \ne 0$, and therefore ${\bm \ell}^*$ can only appear as exactly the $(\phi(N) + 1)$-th vector.

The reductions in \cref{eq:reduced_no_guess,eq:reduced_guess} used a value of $\beta_2$ which was close to the minimum required for successful inversion.  
This choice plays an important role in the resulting basis structure.
By \Cref{lem:uniq}, 
$\pm{\bm \ell}^*$ is the only nonzero lattice vector that can satisfy both the ${\tt B}_1$ and ${\tt B}_2 $ constraints.
Thus,
increasing $\beta_2$ lengthens all other lattice vectors while leaving the length of $\pm{\bm \ell}^*$ unchanged.  
Therefore,
we expect $\pm{\bm \ell}^*$ to appear earlier in the lattice basis order as $\beta_2$ increases, as the index of ${\bm \ell}^*$ serves as a proxy for the number of lattice vectors shorter than ${\bm \ell}^*$.
This behavior is illustrated in \Cref{fig:structure},
which plots the average index of ${\bm \ell}^*$ in the reduced lattice as a function of $\beta_2$ for several values of $\beta_0$, both with and without a guess.

\delimitershortfall=5pt 
The three curves in \cref{fig:structure} without a guess nearly coincide,
illustrating that the reduced lattice structure is agnostic to the value of $\beta_0$ in this guess regime.
Below the minimum value of $\beta_2$,
the desired vector ${\bm \ell}^*$ does not appear in the reduced lattice. 
Above this threshold,
${\bm \ell}^*$ appears as exactly the $(\phi(N) + 1)$-th reduced basis vector until $\beta_2$ becomes sufficiently large that ${\bm \ell}^*$ becomes the shortest vector overall.  
This behavior is consistent with the fact that exactly one of the vectors in the reduced basis satisfying the ${\tt B}_1$ constraints has $\gamma = 1$, 
and this is the only reduced basis vector whose length depends on $\beta_0$. 

The inclusion of a nonzero guess $\overline{\bf x}$ makes the reduced lattice dependent on $\beta_0$.  
For the smaller values of $\beta_0$ (the purple and red curves in \cref{fig:structure}),
the location of  ${\bm \ell}^*$ in the reduced lattice shows substanially more variation.
Once $\beta_2$ exceeds the recovery threshold,
the average index of ${\bm \ell}^*$ decreases gradually to 1.
However, we observe that the smallest $\beta_0$ value (red curve) is generally shifted to the right of the middle $\beta_0$ value (purple curve). 
This
indicates that choosing $\beta_0$ too small increases the minimum value of $\beta_2$ required for recovery, as vectors with $\abs{\gamma}>1$ are not penalized enough.   
As seen in the example reduced basis with a guess \cref{eq:reduced_guess},
many short vectors have $\abs{\gamma}$ significantly larger than 1,
so increasing $\beta_0$ penalizes these vectors more heavily than ${\bm \ell}^*$.
However, we note that 
multiple vectors with $\gamma = 1$ show up in \cref{eq:reduced_guess},
so even with a larger value of $\beta_0$, we still need to ensure that $\beta_2$ is sufficiently large.

At the larger $\beta_0$ value,
the brown curve shows the same qualitative behavior as the curves obtained without a guess.  
This suggests that a sufficiently large $\beta_0$ value forces the $\gamma$ coefficients to mimic  the structure of the reduced lattice without a guess.  This is apparent by inspecting \cref{eq:reduced_guess}.
Increasing $\beta_0$ substantially lengthens vectors with nonzero $\gamma$, 
eventually making vectors with $\gamma=0$ comparatively shorter.
Despite the qualitative agreement, 
the brown curve reaches an average index of 1 at a smaller value of $\beta_2$ than the curves without a guess. 
Including a guess reduces the magnitude of the ${\tt A}$ block of ${\bm \ell}^*$,
but does not affect the ${\tt A}$ blocks of vectors with $\gamma = 0$,
so a smaller value is required to make ${\bm \ell}^*$ short relative the vectors $\gamma = 0$ vectors with a guess.

Our analysis in \cref{sec:beta0} will demonstrate that an intermediate value of $\beta_0$ is optimal, 
balancing the need to penalize vectors with  large $\abs{\gamma}$ coefficients while still allowing multiple vectors with $\gamma \ne 0$ in the reduced basis. 
While one can always choose  $\beta_0$ large enough that the lattices in \cref{eq:reduced_no_guess,eq:reduced_guess} exhibit identical structure, 
doing so is suboptimal,
because it requires a larger value of $\beta_2$ (which can propagate numerical precision issues and increases runtime).

\begin{figure}[pos=!htbp]
    \centering
    \includegraphics[width=.67\textwidth]{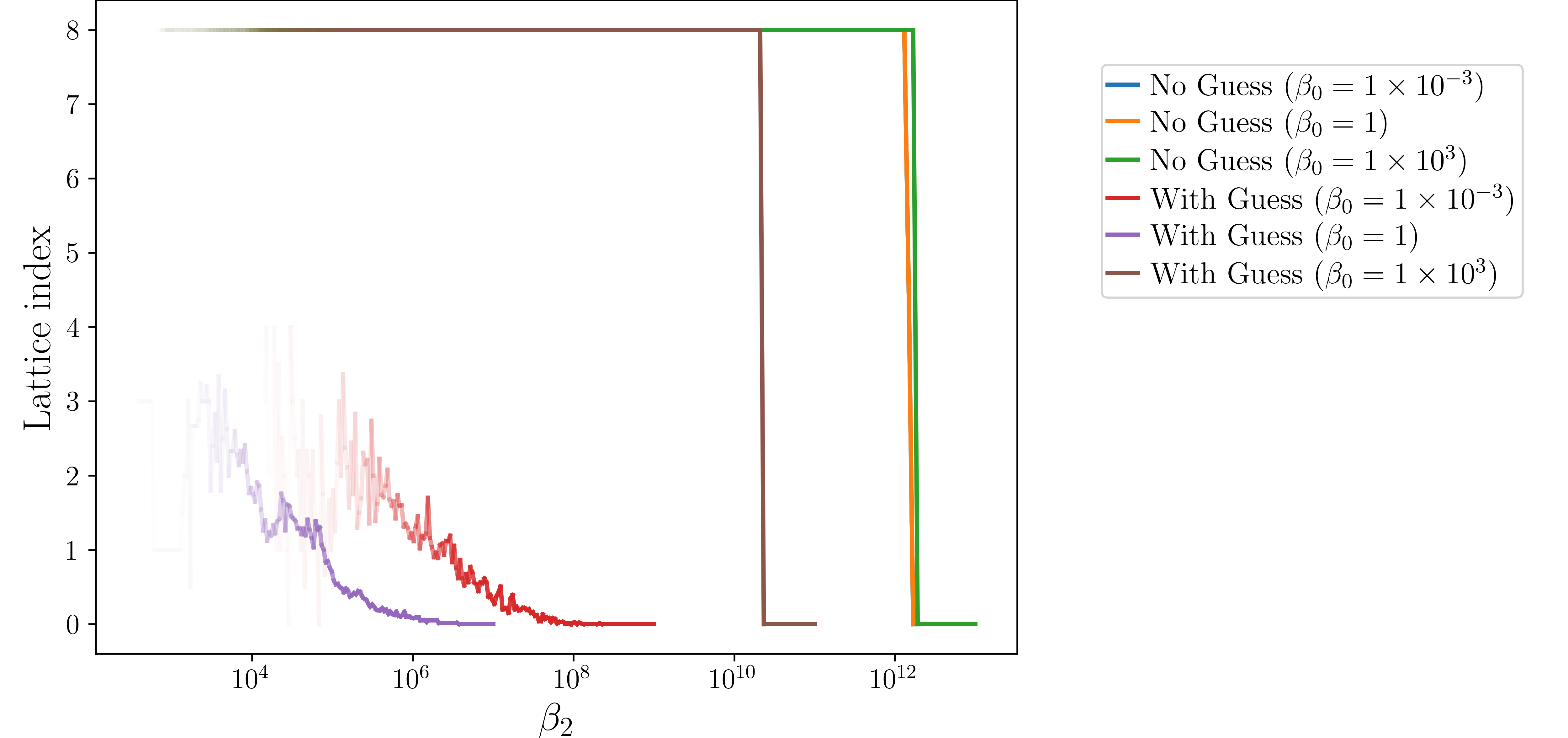}
    \vspace{-0.5pc}
    \caption{The average index of ${\bm \ell}^*$ in the reduced lattice as a function of $\beta_2$ for 100 test signals of length $N=30$ with entries distributed as $\binomdist{1000}{0.5}$, recovered with $J=1$.  The line opacity indicates the fraction of test signals for which ${\bm \ell}^*$ was recovered successfully.  Note that the blue and orange curves visually coincide.  Since $\phi(30)=8$, ${\bm \ell}^*$ appears among the first 9 lattice vectors.  
    }
    \label{fig:structure}
\end{figure}

\subsection{Stability Connection} \label{sec:stability}

The use of integer lattices in the numerical implementation limits the functional precision.
As discussed previously, 
the parameters $\beta_2$ and $\beta_3$ jointly determine how many digits of precision survive the truncation to integers, 
which is approximately $\log_{10}(\beta_2\beta_3)$.
At the same time, 
the sampled DFT coefficients are available only to finite precision in practical applications.
Therefore, 
recovery performance is influenced by the interplay between truncation error arising from the integer lattice representation and finite-precision error in the input data. 

To investigate this interplay, 
we generated a test set of 100 random signals of length $N=25$ with entries distributed as $\binomdist{25}{0.5}$.
For a grid of varying input precisions and $\beta_2$ values (with $\beta_3$ fixed),
we ran \cref{alg:memo_1D} on each test signal.  \cref{fig:prec} shows the number of successful recoveries for each combination of input precision and $\beta_2$.

The orange region of the plot contains the pairs of $\beta_2$ and precision values which generally facilitate successful recovery.
The vertical blue line marks the minimal input precision for which
all test signals were recovered
for an optimal choice of $\beta_2$.
Thus, 
the region to the left of this line
represents the unstable region in which no choice of $\beta_2$ can overcome the insufficient precision in the data.
Similarly, 
the horizontal red line marks the minimal value of $\beta_2$ where all test signals could be recovered. 
The region below this line corresponds to values of $\beta_2$ that do not sufficiently separate the length of the target lattice vector ${\bm \ell}^*$ from competing lattice vectors, preventing correct recovery.

\begin{figure}[pos=!htbp]
    \centering
    \includegraphics[width=\textwidth]{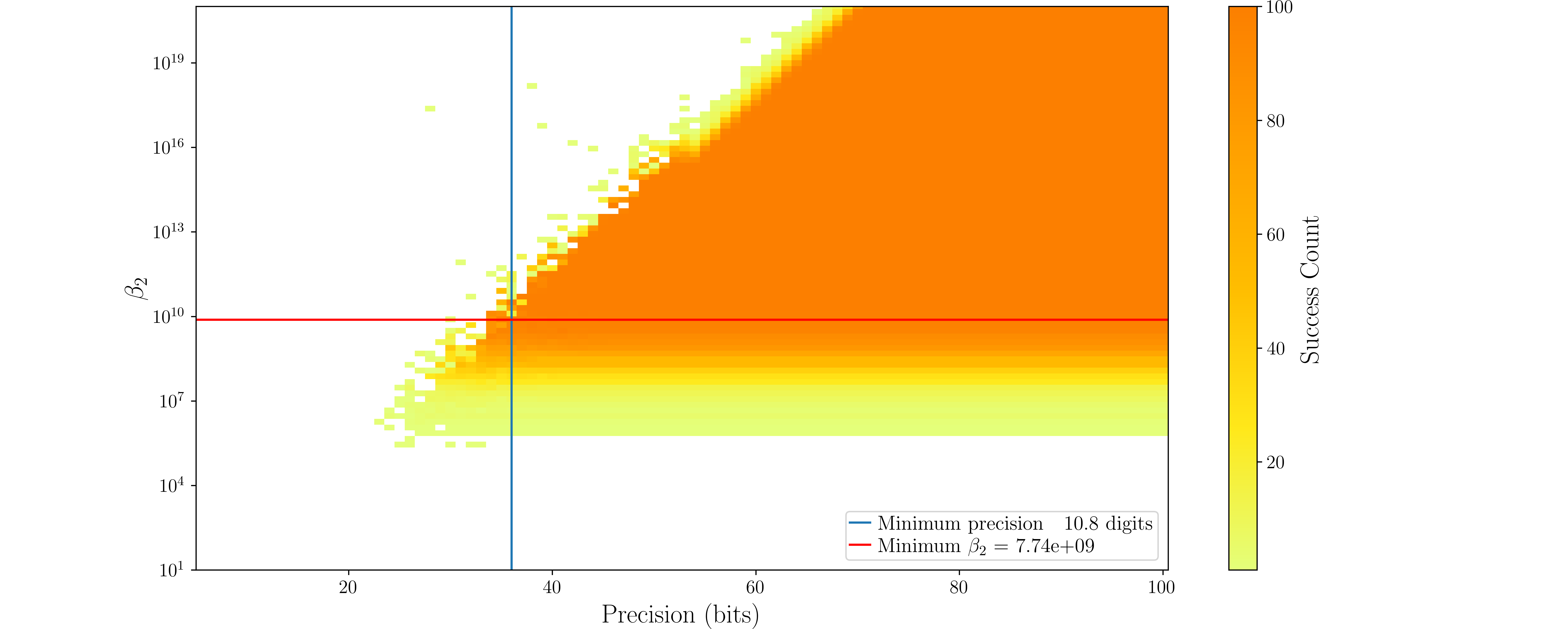}
    \vspace{-1.5pc}
    \caption{Recovery rates for 100 test signals of length $N=25$ with entries distributed as $\binomdist{25}{0.5}$, recovered using $J=1$.  For each data precision ad $\beta_2$ value, the percentage of successful recoveries is shown. The blue and red lines mark the first data precision and $\beta_2$ values, respectively, for which all 100 test signals were recovered successfully.}
    \label{fig:prec}
\end{figure}

This plot also exhibits a triangular region of unsuccessful recovery above both of these minimal threshold values.
In this regime, the input data contain sufficient precision to allow stable recovery of ${\bm \ell}^*$,
but $\beta_2$ is set too large relative to that precision. 
To understand why this leads to recovery failure,
let $\tilde{x}'_k$ be the measured value of $\tilde{x}_k$
with measurement error $\epsilon \coloneqq  \abs{\tilde{x}_k - \tilde{x}_k'}$. 
If $\tilde{x}_k'$ replaces $\tilde{x}_k$ in the lattice basis in \cref{eq:lattice_basis},
the desired lattice vector ${\bm \ell}^* = {\tt B}\coeffs$ will have a nonzero ${\tt B}_2$ block. 
In contrast to \cref{eq:shortest_vec}, 
the length of ${\bm \ell}^*$ now scales with $\beta_2$, as 
\begin{equation*}
    \norm{{\bm \ell}^*}^2
    = \norm{{\bf x} - \overline{\bf x}}^2+ \beta_0^2 + \beta_2^2\abs{\tilde{x}_k - \tilde{x}_k'}^2
     + \beta_2^2\epsilon^2.
\end{equation*}
Furthermore,
the density of the cyclotomic integers in $\C$ implies there exist infinitely many vectors ${\bm \alpha} \in \Z^N$ satisfying
${\tt B}_1 \coeffsone = {\bf 0}$ and $\abs{\tilde{\alpha}_k - \tilde{x}_k'} < \epsilon$.
Therefore, as $\beta_2$ increases, we can eventually find an ${\bm \alpha}$ satisfying 
$\norm{{\tt B}\coeffsone} < \norm{{\bm \ell}^*}$.

To quantify this effect, consider any lattice vector ${\bm \ell} = {\tt B}\coeffs$ satisfying $\norm{{\bm \ell}} \le \norm{{\bm \ell}^*}$.  Since the contribution of the ${\tt B}_2$ block is bounded by the total vector norm,
\begin{equation*}
\beta_2^2\abs{\tilde{\alpha}_k - \gamma \tilde{x}_k}^2
    =\norm{{\bm \ell}^{({\tt B}_2)}}^2
    \le \norm{{\bm \ell}}^2
    = \norm{{\bm \alpha} - \gamma\overline{\bf x}}^2 + \beta_0^2 + \beta_2^2\epsilon^2,
\end{equation*}
we can bound the frequency space error of a recovered vector by,
\begin{equation*}
    \abs{\tilde{\alpha}_k - \gamma \tilde{x}_k}^2
    \le \frac{1}{\beta_2^2}\left(\norm{{\bm \alpha} - \gamma\overline{\bf x}}^2 + \beta_0^2\right) + \epsilon^2
    = \bigo{\beta_2^{-2} + \epsilon^2}.
\end{equation*}
This estimate reveals the tradeoff between $\beta_2$ and measurement precision.  Increasing $\beta_2$ forces short competing lattice vectors to agree more closely with the measured Fourier data, reducing the first error term. However, once $\beta_2^{-2}$ becomes comparable to $\epsilon^2$, further increases in $\beta_2$ primarily amplify the effects of measurement error rather than improving recovery.
We thus expect that the minimum $\beta_2$ value and precision required to recover a signal are related,
with the values possibly differing by a constant.
\Cref{fig:prec} supports this prediction:
the order of magnitude of the minimum value of $\beta_2$ required to recover all test signals,
$\log_{10}\beta_2 = 9.9$ is similar to the minimum number of digits required.

The next section focuses on estimating the minimum value of $\beta_2$ required to recover an integer signal.
When numerically testing these estimates,
we generally work with substantially higher precision than is required for successful reconstruction.
This places us in the regime well to the right of the blue line in \cref{fig:prec},
where picking a value of $\beta_2$ slightly larger than optimal will not impair recovery.
The behavior observed in \cref{fig:prec} suggests that the minimum useful value of $\beta_2$ is independent of elevated precision in this regime,
and should coincide with the optimal $\beta_2$ value when precision is actually limited.

\section{Parameter Analysis} \label{sec:theory}

\newcommand{\Rgamma}{\ensuremath{R_\gamma}}
\newcommand{\Rgammasq}{\ensuremath{R_\gamma^2}}

In this section, we analyze the lattice $\mathcal{L}$ with basis ${\tt B}$ given in \cref{eq:lattice_basis}.  
Our goal is to estimate minimum values of the parameters $\beta_1$ and $\beta_2$ to ensure that ${\bm \ell}^*$ appears in the LLL-reduced basis.
In practice, 
limited precision in the DFT measurements and algorithm runtime considerations prohibit arbitrarily large parameter values,
making precise estimates desirable.
We emphasize that this analysis is an estimate rather than a strict bound.
It provides guidance for parameter selection,
which may require additional tuning in practice.
The parameters $\beta_1$ and $\beta_2$ play the dominant role in ensuring that ${\bm \ell}^*$ appears in the reduced basis,
while the parameter $\beta_0$ plays a less critical role in recovery.
This analysis quantifies the dependence of $\beta_1$ and $\beta_2$ on $\beta_0$,
while the selection of $\beta_0$ is discussed in \cref{sec:beta0}.

\subsection{Estimating \texorpdfstring{$\beta_{1}$}{beta1}}
\label{sec:beta1}

We begin by analyzing $\beta_1$ to determine a value that guarantees that any vector ${\bm \ell}$ which is shorter than ${\bm \ell}^*$ must satisfy ${\bm \ell}^{({\tt B}_1)} = {\bf 0}$.
This reduces the number of lattice vectors which are shorter than ${\bm \ell}^*$,
thereby increasing the likelihood that ${\bm \ell}^*$ appears in the reduced basis.
Moreover,
it enforces structural constraints on the reduced lattice, as seen in \cref{sec:lattice_ex}.

Fix a vector ${\bm \ell} = {\tt B}\coeffs\in \mathcal{L}$ with $\norm{{\bm \ell}} \le \norm{{\bm \ell}^*}$ and chose any prime divisor $p$ of $N$.
Considering only the components $\bm{\ell}^{({\tt B}_0)}$ and one block $\bm{\ell}^{(p)}$ in $\bm{\ell}^{({\tt B}_1)}$ from \cref{eq:ellP_def} gives a trivial lower bound on the length $\norm{{\bm \ell}}$ by
\begin{equation}
\label{eq:beta1_first}
    \abs{\gamma\beta_0}^2 + \beta_1^2\norm{{\bm \alpha}^{(N/p)} - \gamma{\bf x}^{(N/p)}}^2= 
    \norm{{\bm \ell}^{({\tt B_0})}}^2 + \norm{{\bm \ell}^{(p)}}^2
    = \norm{\bm{\ell}^{({\tt B}_0)} \ \bm{\ell}^{(p)}}^2 
    \le \norm{\bm{\ell}}^2 \le  K^2 + \beta_0^2 .
\end{equation}
The final inequality in \cref{eq:beta1_first} comes from the length of ${\bm \ell}^*$ in \cref{eq:shortest_vec}.
Rearranging \cref{eq:beta1_first} yields, 
\begin{equation}
    \label{eq:dec_norm_bound}
    \beta_1^2\norm{{\bm \alpha}^{(N/p)} - \gamma{\bf x}^{(N/p)}}^2 \le K^2 + (1 - \abs{\gamma}^2)\beta_0^2 \le K^2 + \beta_0^2.
\end{equation}
Since all entries of ${\bm \alpha}^{(N/p)} - \gamma{\bf x}^{(N/p)}$ are integers, 
\cref{eq:dec_norm_bound} implies that choosing 
\begin{equation}
    \label{eq:beta1}
    \beta_1 > \sqrt{K^2 + \beta_0^2}
\end{equation} 
guarantees that$\norm{{\bm \alpha}^{(N/p)} - \gamma{\bf x}^{(N/p)}}^2 = 0$.
This choice of $\beta_1$ thus implies that ${\bm \alpha}^{(N/p)} = \gamma {\bf x}^{(N/p)}$.  
Since $p$ was arbitrary, 
\cref{eq:beta1} guarantees that any such ${\bm \alpha}$ will satisfy the ${\tt B}_1$ block constraints. 
In particular,
when $\abs{\gamma}=1$ as desired, 
the  vector ${\bm \alpha}$ will match the DFT data ${\tilde x}_k$ for all $k$ such that $\gcd(k, N) > 1$.

\subsection{Counting Short Vectors}\label{sec:rho}
Determining sufficient conditions on $\beta_2$ is considerably more difficult than for $\beta_1$ as the ${\tt B}_2$ block of the lattice does not consist of integer-valued entries.
The analysis that lead to \cref{eq:beta1} from \cref{eq:dec_norm_bound} relied on the fact that any nonzero integer vector has a norm of at least one,
so
forcing the contribution of the ${\tt B}_1$ block below this threshold guaranteed it was 0.
No analogous argument applies to the ${\tt B}_2$ block,
whose resulting norm can take arbitrarily small values.

Instead, 
we will first compute the number of vectors shorter than ${\bm \ell}^*$ in the lattice, 
\begin{equation*} 
	\rho(\beta_2) \coloneqq 
    \#\set{{\bm \ell} \in \mathcal{L} \given \norm{{\bm \ell}} \le \norm{{\bm \ell}^*}}, 
\end{equation*}
where it is implied that $\rho(\beta_2)$ depends on the choice of $\beta_0$, and  $\beta_1$ is assumed to satisfy  \cref{eq:beta1}. 
Based on the structure of the reduced basis,
we will then apply an appropriate value of $\rho$ to solve for the corresponding value of $\beta_2$ required for successful recovery in \cref{sec:final_beta2}.

We begin by expanding the definition of $\rho(\beta_2)$ in terms of the lattice coefficients and the length of the desired vector ${\bm \ell}^*$ from \cref{eq:shortest_vec},
\begin{align} \label{eq:rho_expand}
	\rho(\beta_2) 
	=  \#\set{{\bm \ell} \in \mathcal{L} \given \norm{{\bm \ell}} \le \norm{{\bm \ell}^*}} 
	=  \#\set[\Big]{({\bm \alpha}, \gamma) \in\Z^{N+1} \given \norm[\big]{{\tt B}\coeffs}^2 \le K^2 + \beta_0^2 }.
\end{align}
Our first goal is to derive a bound on the possible values of ${\bm \alpha}$ by
showing that any choice of ${\bm \alpha}$ yielding a lattice vector shorter than ${\bm \ell}^*$ must belong to a specific bounded region.  

We can bound the length of 
$\bm{\ell} = {\tt B}\coeffs$
below by its components in 
$\bm{\ell}^{({\tt A})}$ and $\bm{\ell}^{({\tt B}_0)}$.
Therefore,
for $\norm{{\bm \ell}} \le \norm{{\bm \ell}^*}$,
we have
\begin{equation}
\label{eq:alpha_bound}
    \norm{{\bm \alpha}-\gamma\overline{\bf x}}^{2} + \abs{\gamma\beta_{0}}^{2}
    = \norm{[\bm{\ell}^{({\tt A})} \ \bm{\ell}^{({\tt B}_0)}]}^2 
    \le \norm{\bm{\ell}}^{2}
    \le K^{2}+\beta_{0}^{2} ,
\end{equation}
which implies that 
\begin{equation*}
     \norm{{\bm \alpha}-\gamma\overline{\bf x}} 
    \le R_\gamma \coloneqq \sqrt{K^2 + (1 -\gamma^2)\beta_{0}^{2}}.   
\end{equation*}
\Cref{eq:alpha_bound} bounds the distance from the coefficients ${\bm \alpha}$ to the scaled guess $\gamma{\bf \overline{x}}$ by \Rgamma{}.
Accordingly,
we define a feasible set of valid lattice coefficients for each $\gamma\ge 0$.
We consider the integer points $D_\gamma^\Z \coloneqq D_\gamma \cap \Z^N$, 
where $D_\gamma$ is defined by,
\begin{equation} \label{eq:D_gamma_def}
    D_\gamma \coloneqq
    \set{ 
        {\bm \alpha} \in \R^{N} 
        \given \norm{{\bm \alpha} - \gamma\overline{{\bf x}}} \le \Rgamma 
        \text{ and } 
        {\bm \alpha}^{(N/p)} = \gamma{\bf x}^{(N/p)} \text{ for each prime } p\divs N
    },
\end{equation}
where the condition on ${\bm \alpha}^{(N/p)}$ comes from \cref{eq:ellP_def} and the fact that ${\bm \ell}^{({\tt B}_1)} = {\bf 0}$.

The set $D_\gamma$ can be described as the intersection of an $N$-dimensional hyperball of radius \Rgamma{} centered at $\gamma\overline{{\bf x}}$,
with the hyperplanes
${\bm \alpha}^{(N/p)} = \gamma{\bf x}^{(N/p)}$ for all prime divisors $p$ of $N$.
These hyperplane constraints equivalently specify, 
by \cref{def:dec}, 
the $N - \phi(N)$ DFT coefficients $\tilde{\alpha}_k = \gamma\tilde{x}_k$ with  $\gcd(k, N) > 1$.
Therefore, 
for any fixed $\gamma$,
the solution space of the linear system obtained by combining the constraints
${\bm \alpha}^{(N/p)} = \gamma{\bf x}^{(N/p)}$ for all prime $p \divs N$ has dimension $N-\left(N-\phi(N)\right)=\phi(N)$.
Moreover, 
as the guess $\overline{{\bf x}}$ was also chosen to satisfy each of these $N-\phi(N)$ DFT coefficients,
the hyperplanes in \cref{eq:D_gamma_def} pass through the center $\gamma\overline{\bf x}$ of the $N$-dimensional hyperball. Therefore, the result of this intersection is a lower dimensional ball of the same radius $R_\gamma$.
For brevity,
we denote the dimension of the hyperball by
$\Phi \coloneqq \phi(N)$.

Additionally,
note that if 
\begin{equation}
\label{eq:gamma_max}
    \abs{\gamma} > \gamma_{\max} \coloneqq \floor*{ \sqrt{(K/\beta_0)^2 + 1}},
\end{equation}
then $\Rgamma < 0$, implying that $D_\gamma = \emptyset$ for all such $\gamma$.  
Thus, the bound $\norm{{\bm \alpha} - \gamma\overline{\bf x}}\le \Rgamma$ implicitly also restricts feasible $\gamma$ values to a finite range.
The set in \cref{eq:rho_expand} can thus be refined to
\begin{equation}
    \label{eq:rho_eq2}
    \rho(\beta_2) 
    = \sum_{\abs{\gamma} \le \gamma_{\max}}\# \set*{{\bm \alpha}\in D_\gamma^\Z \given \norm[\big]{{\tt B}\coeffs}^2 \le K^2 + \beta_0^2 }.
\end{equation}
However, 
an exact count of $\rho(\beta_2)$ is still difficult. 
Instead, for each fixed $\gamma$, 
we introduce the probability space obtained by selecting ${\bm \alpha}$ uniformly at random from $D_\gamma^\Z$.  
Under this distribution, 
we have 
\begin{equation}
\label{eq:prob_intro}
    \# \set*{{\bm \alpha}\in D_\gamma^\Z \given \norm[\big]{{\tt B}\coeffs}^2 \le K^2 + \beta_0^2 }
    = \left(\#D_\gamma^\Z\right) \cdot \ \bP\event[\Big]{\norm[\big]{{\tt B}\coeffs}^2 \le K^2 + \beta_0^2
    \given {\bm \alpha} \in D_\gamma^\Z}.
\end{equation}
Thus, it suffices to approximate the two terms on the right-hand side to obtain an estimate for $\rho(\beta_2)$.
We will first analyze the conditional probability term,
where we approximate the discrete uniform distribution over $D_\gamma^\Z$ by the continuous uniform distribution of $D_\gamma$,
\begin{equation} \label{eq:prob_cont}
    \bP\event*{\norm[\big]{{\tt B}\coeffs}^2 \le K^2 + \beta_0^2
    \given {\bm \alpha} \in D_\gamma^\Z}
    \approx \bP\event*{\norm[\big]{{\tt B}\coeffs}^2 \le K^2 + \beta_0^2 \given {\bm \alpha} \in D_\gamma} ,
\end{equation}
which is a reasonable approximation when $D_\gamma$ contains many lattice points.

As we are choosing $\beta_1$ sufficiently large to ensure that $\norm[\big]{{\tt B}_1\coeffs}=0$, 
expanding the norm of a generic lattice vector by its block structure yields,
\begin{equation}
\label{eq:vector_len}
   \norm[\big]{{\tt B}\coeffs}^2 
   = \norm{{\bm \alpha} - \gamma\overline{\bf x}}^2 + \gamma^2\beta_0^2 + \beta_2^2\norm[\big]{{\tt B}_2\coeffs}^2.
\end{equation}
The first term,
by the discrete Parseval relation,
is given by 
\begin{equation*}
    \norm{{\bm \alpha} - \gamma\overline{\bf x}}^2 = \frac{1}{N}\norm{\tilde{{\bm \alpha}} - \gamma\tilde{\overline{\bf x}}}^2 = \frac{1}{N}\sum_{k = 0}^{N - 1} \abs{\tilde{\alpha}_k - \gamma\tilde{\overline{x}}_k}^2 .
\end{equation*}
Using \cref{eq:B2_err} to rewrite the final term of \cref{eq:vector_len}, 
we can express the norm entirely in terms of the DFT of the coefficients,
\begin{equation} \label{eq:vector_len2}
   \norm[\big]{{\tt B}\coeffs}^2 
   =  \left(\frac{1}{N}\sum_{k = 0}^{N - 1} \abs{\tilde{\alpha}_k - \gamma\tilde{\overline{x}}_k}^2 + \beta_2^2\sum_{j = 1}^{J} \abs{\tilde{\alpha}_{k_j} - \gamma\tilde{x}_{k_j}}^2\right) + \gamma^2\beta_0^2 
   \eqqcolon f(\tilde{\bm \alpha}) + \gamma^2\beta_0^2.
\end{equation}
We thus express the probability term in \cref{eq:prob_cont} as 
\begin{equation} \label{eq:prob_f}
\begin{split}
    \bP\event[\Big]{\norm[\big]{{\tt B}\coeffs}^2 \le K^2 + \beta_0^2
    \given {\bm \alpha} \in D_\gamma}
    &= \bP\event[\Big]{f(\tilde{\bm \alpha})+\gamma^2\beta_0^2 \le K^2 + \beta_0^2
    \given {\bm \alpha} \in D_\gamma} \\
    &= \bP\event[\Big]{f( \tilde{\bm \alpha}) \le R_\gamma^2
    \given {\bm \alpha} \in D_\gamma}.
\end{split}
\end{equation}
Next,
we give a convenient reformulation of the 
uniform probability space over $D_\gamma$ in Fourier coordinates.
However,
as the DFT coefficients of a real-valued signal satisfy a conjugate symmetry relation,
we can impose this symmetry by first replacing the complex Fourier coefficients by an equivalent real-valued representation.
To this end,
we define a map $\widehat{\bf z}\colon\R^N\to\R^N$ by
\begin{equation} \label{eq:z_hat_def}
    \widehat{z}_{k}({\bm \alpha}) =
    \begin{cases}
        \tilde{\alpha}_0 & k = 0 \\
        \sqrt{2}\Re\tilde{\alpha}_k & 0 < k < N/2 \\
        \tilde{\alpha}_{N/2} & k = N/2 \\
        \sqrt{2}\Im\tilde{\alpha}_{N - k} & N/2 < k < N.
    \end{cases}
\end{equation}
Letting $\overline{\bf z} \coloneqq {\bf \widehat{z}}(\overline{\bf x})$, we can also define the shifted coordinates by the map ${\bf z}\colon\R^N\to \R^N$,
\begin{equation}
    \label{eq:z_def}
    {\bf z}(\bm \alpha) =  {\bf \widehat{z}}(\bm \alpha) - \gamma \overline{\bf z} .
\end{equation}
The following lemma demonstrates the utility of the ${\bf z}$ coordinate representation
by providing a simple characterization of the set of representatives,
\begin{equation*}
    F_\gamma 
    \coloneqq
    \set{{\bf z}({\bm \alpha}) \given {\bm \alpha} \in D_\gamma}.
\end{equation*}

\begin{lemma}
\label{lem:prob}
Let $\widehat{\bf z}$ and ${\bf z}$ be defined as in \cref{eq:z_hat_def,eq:z_def}.  Then, 
\begin{enumerate}
    \item The map $\widehat{\bf z}$ is a scaled orthogonal transformation satisfying $\norm{\widehat{\bf z}({\bf x})}=\sqrt{N}\norm{{\bf x}}$ for all ${\bf x}\in\R^N$.
    \item The map ${\bf z}$ is a scaled affine isometry that maps $D_\gamma$ onto the canonical embedding of a $\Phi$-dimensional ball of radius $\Rgamma\sqrt{N}$.
\end{enumerate}
\end{lemma}
\begin{proof}
We first show the $\widehat{\bf z}$ is an orthogonal transformation. 
Let ${\bm \alpha}\in\R^N$, and without loss of generality,
assume $N$ is even (otherwise, simply omit the index $N/2$ term). 
Since ${\bm \alpha}$ is real-valued,
its DFT satisfies the conjugate-symmetry condition $\tilde{\alpha}_{N-k}=\tilde{\alpha}^*_k$. 
By first using the discrete Parseval relation and then applying conjugate symmetry, 
we have 
\begin{align}
    N\norm{{\bm \alpha}}^2 
    = \norm{\tilde{\bm \alpha}}^2
    = \abs{\tilde{\alpha}_0}^2 + \abs{\tilde{\alpha}_{N/2}}^2 + \sum_{k=1}^{N/2-1}\left( \abs{\tilde{\alpha}_k}^2   + \abs{\tilde{\alpha}_{N-k}}^2\right) 
    = \abs{\tilde{\alpha}_0}^2 + \abs{\tilde{\alpha}_{N/2}}^2 + 2\sum_{k=1}^{N/2-1}\abs{\tilde{\alpha}_{k}}^2 .
\end{align}
By expanding the absolute value of $\tilde{\alpha}_k$ into real and imaginary parts and using the definition \cref{eq:z_hat_def} of $\widehat{\bf z}$,
we then obtain
\begin{align}
    N\norm{{\bm \alpha}}^2 
    &= \abs{\tilde{\alpha}_0}^2 + \abs{\tilde{\alpha}_{N/2}^2} + 2\sum_{k=1}^{N/2-1}\left(\abs{\Re\tilde{\alpha}_k}^2+\abs{\Im\tilde{\alpha}_k}^2\right) \\
    &= \abs{\widehat{z}_0({\bm \alpha})}^2 + \abs{\widehat{z}_{N/2}({\bm \alpha})}^2 + \sum_{k=1}^{N/2-1}\abs{\widehat{z}_k({\bm \alpha})}^2 + \sum_{k=N/2+1}^{N-1}\abs{\widehat{z}_k({\bm \alpha})}^2 
    = \norm{\widehat{\bf z}({\bm \alpha})}^2 , 
\end{align}
as desired.  
Moreover,
as the DFT is linear,
so is $\widehat{\bf z}$.

For the second statement,
since $\widehat{\bf z}$ is orthogonal and ${\bf z}$ differs from $\widehat{\bf z}$ only by a translation,
${\bf z}$ is a scaled affine isometry.  
If ${\bm \alpha} \in D_\gamma$,
the orthogonality of $\widehat{\bf z}$ also implies that
\begin{equation}
    \label{eq:cond1}
    \norm{{\bf z}({\bm \alpha})}
    = \norm{{\widehat{\bf z}({\bm \alpha} - \gamma\overline{\bf x}})}
    = \sqrt{N} \norm{{\bm \alpha} - \gamma\overline{\bf x}} ,
\end{equation}
so $\norm{{\bf z}({\bm \alpha})} \le \Rgamma\sqrt{N}$ if and only if $\norm{{\bm \alpha} - \gamma\overline{\bf x}} \le \Rgamma$.
Additionally,
\cref{def:dec} implies that the $D_\gamma$ hyperplane conditions ${\bm \alpha}^{(N/p)} = \gamma{\bf x}^{(N/p)}$ for each prime $p\divs N$ is equivalent to 
\begin{equation} \label{eq:freq_dec_cond}
    \tilde{\alpha}_k = \gamma\tilde{x}_k \text{ for each } \gcd(k, N) > 1.
\end{equation}
We can observe that ${\bf z}({\bm \alpha})$ inherits the $N - \Phi$ coordinate constraints from \cref{eq:freq_dec_cond}.
If $\gcd(k,N)>1$,
then for $0<k<N/2$ we have
\begin{equation*}
    z_k({\bm \alpha}) + \cu z_{N-k}({\bm \alpha}) 
    = \sqrt{2}\left(\tilde{\alpha}_{k}-\gamma\tilde{x}_{k}\right) = 0 
    \qquad \text{for } 0 < k < N / 2,
\end{equation*}
which fixes both $z_k({\bm \alpha})$ and $z_{N-k}({\bm \alpha})$ to 0.
Note that we also have $z_0({\bm\alpha}) = \tilde{\alpha}_0 -\gamma\tilde{x}_0=0$ 
and
$z_{N/2}({\bm \alpha}) = \tilde{\alpha}_{N/2} -\gamma\tilde{x}_{N/2}=0$. 
Thus
the set of these representatives,
\begin{equation*}
    F_\gamma 
    =
    \set{{\bf z}({\bm \alpha}) \given {\bm \alpha} \in D_\gamma} = \set{ {\bf z} \in \R^N \given \norm{{\bf z}} \le \Rgamma\sqrt{N} \text{ and } z_k = 0 \text{ if } \gcd(k, N) > 1} ,
\end{equation*}
is a $\Phi$-dimensional ball of radius $\Rgamma\sqrt{N}$ centered at the origin.
\end{proof}

As seen in the proof of \cref{lem:prob},
under the map ${\bf z}$,
the hyperplane constraints defining $D_\gamma$ become coordinate constraints.  Furthermore, ${\bf z}$ is a scaled affine isometry, so it preserves Euclidean geometry up to a constant scaling factor.  Therefore,  
${\bf z}$ maps the uniform distribution on $D_\gamma$ to the uniform distribution on the $\Phi$ unfixed coordinates of the $\Phi$-dimensional ball $F_\gamma$.

We now rewrite $f(\tilde{\bm \alpha})$ in the ${\bf z}$-coordinates. 
As in \cref{eq:cond1},
the scaled orthogonality of $\widehat{\bf z}$ gives,
$\norm{{\bm \alpha} - \gamma\overline{\bf x}} = 
\frac{1}{\sqrt{N}}\norm{{\bf z}({\bm \alpha})}$.
Moreover, 
for each sampled frequency $k_j$, we have 
\begin{equation} \label{eq:alpha_kj_to_z}
    2\abs{\tilde{\alpha}_{k_j}-\gamma\tilde{x}_{k_j}}^2
    = 2(\Re \tilde{\alpha}_{k_j} - \gamma \Re \tilde{x}_{k_j})^2 
    + 2 (\Im \tilde{\alpha}_{k_j} - \gamma \Im \tilde{x}_{k_j})^2
    = z_{k_j}^2+z_{N-k_j}^2 ,  
\end{equation}
where the factor of 2 comes from the $\sqrt{2}$-scaling in the definition of ${\bf z}$ 
and we used the fact that $\widehat{z}_{k_j}(\overline{\bf x}) = \Re\tilde{x}_{k_j}$
and $\widehat{z}_{N - k_j}(\overline{\bf x}) = \Im\tilde{x}_{k_j}$ from \cref{eq:guess}.
Now let ${\tt P}$ be the orthogonal projection matrix onto the subspace corresponding to the measured DFT coefficients in ${\bf z}$-coordinates, which can be defined entrywise by
\begin{equation}
\label{eq:projection}
    P_{mn} = \begin{cases}
        1 & m = n \in \set{k_1, \dots, k_J} \cup \set{N-k_1,\dots,N-k_J} \\
        0 & \text{otherwise} .
    \end{cases} 
\end{equation}
As ${\tt P}{\bf z}$ selects only the entries of ${\bf z}$ occurring in \cref{eq:alpha_kj_to_z},
we may rewrite the ${\tt B}_2$ block term from \cref{eq:vector_len2} as,
\begin{equation*}
    \sum_{j=1}^J \abs{\tilde{\alpha}_{k_j}-\gamma\tilde{x}_{k_j}}^2
    = \frac{1}{2}\sum_{j=1}^J\left(z_{k_j}^2+z_{N-{k_j}}^2\right) 
    =\frac{1}{2}\norm{{\tt P}{\bf z}}^2.
\end{equation*}
Therefore, we can equivalently express our random variable $f(\tilde{\bm \alpha})$ in ${\bf z}$-coordinates by,
\begin{align}
    Q({\bf z})
    \coloneqq 
    \frac{\beta_2^2}{2}\norm{{\tt P}{\bf z}}^2
    + \frac{1}{N}\norm{{\bf z}}^2 .
\end{align}
Since ${\bf z}$ maps 
the uniform distribution on $D_\gamma$ to the uniform distribution on $F_\gamma$,
the probability from \cref{eq:prob_f} becomes,
\begin{equation} \label{eq:prob_Q}
    \bP\event[\Big]{f( \tilde{\bm \alpha}) \le \Rgammasq
    \given {\bm \alpha} \in D_\gamma}
    = \bP\event[\Big]{Q({\bf z}) \le \Rgammasq
    \given {\bf z} \in F_\gamma}.
\end{equation}

For ${\bf z} \in F_\gamma$,
$\frac{{\bf z}}{\norm{\bf z}}$ is distributed uniformly on the $\Phi$-dimensional unit sphere~\cite[\getcrefname{theorem}~1.5.6]{muirhead1982aspects}.
In general,
for a vector sampled uniformly from the $n$-dimensional sphere,
the distribution of a fixed number of coordinates converges to a scaled standard normal as the dimension $n$ goes to infinity~\cite{stam1982limit}.
We will apply this result to approximate the distribution of $Q({\bf z})$.
In particular,
${\tt P}\frac{{\bf z}}{\norm{\bf z}}$ selects a fixed number ($2J$) of coordinates of a uniform $\Phi$-dimensional sphere sample.
Thus,
for a reasonably large $\Phi \gg J$,
the normal distribution limit approximates,
\begin{equation}
\label{eq:CLT_sphere}
    \sqrt{\Phi}{\tt P}\frac{{\bf z}}{\norm{{\bf z}}} \approx \normaldist{{\bf 0}}{{\tt P}}. 
\end{equation}
We now define
$r\coloneqq\norm{{\bf z}}$ for ${\bf z}$ distributed uniformly in $F_\gamma$,
which is independent of $\frac{{\bf z}}{\norm{{\bf z}}} = \frac{{\bf z}}{r}$~\cite[\getcrefname{theorem}~1.5.6]{muirhead1982aspects}.
To use the approximation in \cref{eq:CLT_sphere}, 
we can express our random variable $Q$ as,
\begin{equation*} 
    Q({\bf z})= r^2 \left(\frac{\beta_2^2}{2}\norm*{\frac{{\tt P}{\bf z}}{r}}^2+\frac{1}{N} \right) .
\end{equation*}
The independence of $r$ and ${\bf z}/r$ 
allows us to condition on $r$ and later introduce its distribution through an expected value.
The squared 2-norm of a vector of $2J$ independent standard normal random variables follows a $\chi^2$ distribution,
$\norm{\normaldist{{\bf 0}}{{\tt P}}}^2 \sim \chisqdist{2J}$,
which gives the following approximate distribution for our quadratic form,
\begin{equation} \label{eq:as_chi2}
    \norm*{\frac{{\tt P}{\bf z}}{r}}^2 \approx 
    \frac{1}{\Phi}\norm{\normaldist{{\bf 0}}{{\tt P}}}^2
    = \frac{1}{\Phi}\chisqdist{2J} .
\end{equation}
Using the model in \cref{eq:as_chi2}, 
we can compute the probability in \cref{eq:prob_Q} by,
\begin{equation} \label{eq:new_prob}
     \bP\event{Q({\bf z}) \le \Rgammasq }
     \approx \bP\event*{\frac{r^2\beta_2^2}{2\Phi}\chisqdist{2J} \le \Rgammasq-\frac{r^2}{N} } .
\end{equation}
For even degrees of freedom,
the $\chi^2$ distribution has CDF,  
\begin{equation} \label{eq:chi2_cdf}
    \bP\event{\chisqdist{2n} \le x} 
    = 1 - \exp\left(-\frac{x}{2}\right) \sum_{j = 0}^{n - 1} \frac{(x/2)^{j}}{j!}
    \approx \frac{(x/2)^n}{n!},
\end{equation}
where the last expression is a first-order approximation for small $x$~\cite[6.5.13,6.5.4,6.5.29]{Abramowitz1965}.
We will
apply the approximate $\chisqdist{2J}$ CDF in \cref{eq:chi2_cdf} to the probability in \cref{eq:new_prob},
which is valid as we expect $\beta_2$ to be large.
As $r \ge 0$ and ${\tt P}{\bf z}/r$ are independent,
conditioning on $r$ yields
\begin{equation} \label{eq:jensen}
     \bP\event{Q({\bf z}) \le \Rgammasq}
     = \bE_r\event[\Big]{\bP \event{Q({\bf z}) \le \Rgammasq \given r} }
     \approx \bE_{r}\event*{\frac{1}{J!}\left( \frac{\Phi\left(N\Rgammasq-r^2\right)}{Nr^2\beta_2^2}\right)^J}
\end{equation}
We now approximate the remaining dependence of the radial variable by evaluating the expression in \cref{eq:jensen} at the mean square radius.
A straightforward calculation shows that the expected value of $r^2$ in a $\Phi$-dimensional ball of radius $\Rgamma\sqrt{N}$ is given by,
\begin{equation} \label{eq:E_r2}
    \bE\event{r^2}
    =\frac{\Phi}{\Phi+2}R_\gamma^2N,
\end{equation}
which also comes from a later, more general computation in \cref{eq:rsq_moments}.
Applying this last approximation and substituting \cref{eq:E_r2} yields our final estimate for the probability in \cref{eq:prob_intro},
\begin{align}
    \bP\event{Q({\bf z}) \le \Rgammasq }
    \approx \frac{1}{J!}\left( \frac{\Phi\left(N\Rgammasq-\bE\event{ r^2}\right)}{N\beta_2^2\bE\event{r^2}}\right)^J 
    &= \frac{1}{J!}\left( \frac{\Phi N\Rgammasq\left(1-\frac{\Phi}{\Phi+2}\right)}{N^2R_\gamma^2\beta_2^2\frac{\Phi}{\Phi+2}}\right)^J
    = \frac{1}{J!}\left( \frac{2}{N\beta_2^2}\right)^J \label{eq:prob_final}.
\end{align}

With this expression for the probability,
we turn to estimating the second term in \cref{eq:prob_intro},
$\#D_\gamma^\Z$.
We handle this with the commonly used Gaussian Heuristic,
which approximates the number of lattice points in some measurable region by the ratio of the volume of that region to the determinant of the lattice~\cite{Nguyen2010hermite,gama2010enum}.
Intuitively,
as the determinant is the volume of the  fundamental region of the lattice,
it represents how much space each lattice point occupies.
Thus,
dividing the total volume of $D_\gamma$ by this quantity reasonably estimates how many integer lattice points lie inside $D_\gamma$.
As previously discussed,
$D_\gamma$ is a $\Phi$-dimensional hypersphere,
whose volume can be given in terms of the unit ball volume $V_n$,
\begin{equation*} 
	{\rm vol}(D_\gamma) = V_\Phi \Rgamma^\Phi,
	\qquad V_n = \frac{\pi^{n/2}}{\Gamma(\frac{n}{2} + 1)} 
    = \frac{\pi^\frac{n}{2}}{(n/2)!}.
\end{equation*}
In this equation,
we used the fact that $\Phi=\phi(N)$ is even (for $N > 2$) to simplify the $\Gamma$ function in the volume of a unit $n$-ball.

We defined the coefficient set $D_\gamma^\Z$ as the intersection of the standard integer lattice $\Z^N$ with $D_\gamma$.
However, 
as $D_\gamma^\Z$ lives in a lower-dimensional affine subspace of $\R^N$,
we must consider the fundamental volume of the intersection between $\Z^N$ and this subspace to derive an accurate count with the Gaussian Heuristic.
The volume of this fundamental region is invariant under the affine shift,
so,
without loss of generality,
we work with the subspace $\set{{\bm \alpha} \in \R^N \given \text{each }{\bm \alpha}^{(N/p)} = {\bf 0}}$ associated with $D_0$.
We denote the determinant of this lattice by $d_N$,
giving the approximation,
\begin{equation} \label{eq:approx_DgammaZ}
    \#D_\gamma^\Z\approx \frac{V_\Phi\Rgamma^\Phi}{d_N}.
\end{equation}
The following lemma provides an explicit formula for $d_N$. The proof of the lemma is deferred to \Cref{ap:det}, where we explicitly construct a basis for the relevant lattice that allows for a direct computation of the determinant from \cref{eq:lattice_det}.
\begin{lemma}
\label{lem:det}
	Fix an integer $N$ and consider the matrix
	\begin{equation} \label{eq:A_def}
		{\tt A} = \begin{bmatrix}
    		{\tt I}_{N/p_{1}} & \cdots & {\tt I}_{N/p_{1}} \\ 
    		\vdots & \ddots & \vdots \\
    		{\tt I}_{N/p_{\omega}} & \cdots & {\tt I}_{N/p_{\omega}} \\ 
		\end{bmatrix}
	\end{equation}
	The lattice $\mathcal{K} = \ker({\tt A}) \cap \Z^{N}$ has determinant 
	\begin{equation*} 
		d_N \coloneqq
		\det \mathcal{K} 
		= \left(\prod_{\text{prime } p \mid N} p^{\frac{\phi(N)}{p-1}}\right)^{1/2}
	\end{equation*}
\end{lemma}

Finally,
by substituting \cref{eq:prob_final,eq:approx_DgammaZ} into \cref{eq:prob_intro}, we obtain our final approximation for $\rho(\beta_2)$ from \cref{eq:rho_eq2}, 
\begin{equation}
\label{eq:rho_final_guess}
    \rho(\beta_2) = \frac{V_\Phi}{d_NJ!}\left( \frac{2}{N\beta_2^2}\right)^J \sum_{\gamma = -\gamma_{\max}}^{\gamma_{\max}}R_\gamma^\Phi 
\end{equation}

\begin{remark}[Moments of $r$ and Jensen's Gap]
Jensen's inequality implies that the approximation in \cref{eq:jensen} is 
actually a lower bound,
as the function $\left(\frac{N^2\Rgammasq - r^2}{r^2}\right)^J$ is convex for $0 \le r \le \sqrt{N}\Rgamma$.
We expect that the gap is negligible for the small values of $J$
that are required to make the asymptotics in \cref{eq:CLT_sphere} accurate.
If one desires a more accurate estimate however,
we note that the expectation in \cref{eq:jensen} may be computed exactly by expanding the expected value argument with the binomial theorem,
\begin{align} \label{eq:binom}
    \bE\event*{\left(\frac{N\Rgammasq - r^2}{r^2}\right)^J}
    = \bE\event*{\left(\Rgammasq r^{-2} - N^{-1}\right)^J} 
    &= \sum_{j = 0}^{J} \binom{J}{j} \Rgamma^{2j}(-N)^{j-J} \bE\event[\Big]{ r^{-2j}} .
\end{align}
The density $p(r)$ of $r$ must be proportional to the surface area of the $\Phi$-dimensional sphere of radius $r$.  Since $r$ ranges from $0$ to $\sqrt{N}\Rgamma$, and $p$ integrates to 1,
we  have $p(r) =\frac{\Phi r^{\Phi-1}}
{N^{\Phi/2}R_\gamma^\Phi}$.
Therefore,
for $2j < \Phi$ (which holds for all $j \le J$ as  $J < \frac{\Phi}{2}$), we have
\begin{equation} \label{eq:rsq_moments}
    \bE\event{r^{-2j}} 
    = \int_0^{\Rgamma} r^{-2j}p(r) \,dr
    = \frac{\Phi}{(\Phi - 2j)N^j\Rgamma^{2j}} ,
\end{equation}
which can be substituted into \cref{eq:binom} to obtain a more accurate value for \cref{eq:jensen}.
However, we elected to use the simpler estimate in this section
as it did not have a significant effect on our numerical results.
\end{remark}

\begin{remark}
As $\gamma_{\max} = \floor*{\sqrt{(K/\beta_0)^2 + 1}}$,
the $2\gamma_{\max} + 1$ terms of the sum in \cref{eq:rho_final_guess} leads to undesirable scaling for small $\beta_0$ and large $K$.
However,
this can be avoided by applying Faulhaber's formula to obtain an equivalent sum with $\bigo{N^2}$ terms,
\begin{align}
    \sum_{\gamma = -\gamma_{\max}}^{\gamma_{\max}} \Rgamma^\Phi
    &= \beta_0^2\sum_{\gamma = -\gamma_{\max}}^{\gamma_{\max}} [(K^2/\beta_0^2+1) - \gamma^2]^{\Phi / 2} \notag{}\\
    &= \beta_0^2\sum_{\gamma = -\gamma_{\max}}^{\gamma_{\max}} \sum_{n=0}^{\Phi/2} \binom{\Phi/2}{n}\left(\frac{K^{2}}{\beta_{0}^{2}} + 1\right)^n(-1)^{\Phi/2-n}\gamma^{\Phi - 2n} \label{eq:binom_thm}\\
    &= \beta_0^2\sum_{n=0}^{\Phi/2} \binom{\Phi/2}{n}\left(\frac{K^{2}}{\beta_{0}^{2}} + 1\right)^n(-1)^{\Phi/2-n}\left(1 + \sum_{\gamma = 1}^{\gamma_{\max}}\gamma^{\Phi - 2n} \right) \label{eq:sum_swap} \\
    &= \beta_0^2\sum_{n=0}^{\Phi/2} \binom{\Phi/2}{n}\left(\frac{K^{2}}{\beta_{0}^{2}} + 1\right)^n(-1)^{\Phi/2-n}\sum_{m=0}^{\Phi - 2n} \binom{\Phi - 2n + 1}{m} B_m\gamma_{\max}^{\Phi - 2n + 1 - m}, \label{eq:faulhaber}
\end{align}
where the coefficients $B_m$ are the Bernoulli numbers.
In \cref{eq:binom_thm},
we apply the binomial theorem to expand the sum argument,
and in \cref{eq:sum_swap},
we swap the order of the sums,
isolating the $\gamma$ dependence in the interior sum.
Finally,
we use Faulhaber's formula in \cref{eq:faulhaber} to eliminate the $\gamma_{\max}$ scaling by replacing the sum over $\gamma$ by a sum with up to $\Phi$ terms~\cite{Graham1994-tk}.
While we will not use this more complicated formulation for our final $\beta_2$ expressions,
it is practical for computational applications in certain regimes  of $K$ and $\beta_0$.  
\end{remark}


\subsection{\texorpdfstring{$\beta_2$}{beta2} Estimate}
\label{sec:final_beta2}

Having computed $\rho(\beta_2)$,
we can estimate the value of $\beta_2$ required to recover an integer signal.
As seen in \cref{sec:lattice_ex}, for sufficiently large $\beta_1$ values,
the reduced lattice contains exactly $\Phi + 1$ vectors which satisfy the ${\tt B}_1$ constraints, while linear independence prevents the remaining $N-\Phi$ basis vectors from doing so. 
Since ${\bm \ell}^*$ satisfies the ${\tt B}_1$ constraints, it should ideally appear in the reduced basis 
if it is among the $\Phi+1$ shortest linearly independent vectors satisfying ${\tt B}_1$.  As $\rho(\beta_2)$ counts both ${\bm \ell}$ and $ -{\bm \ell}$, as well as the zero vector, this corresponds to the condition 
\begin{equation*}
    \rho(\beta_2)\le 2(\Phi+1)+1=2\Phi + 3 . 
\end{equation*}
We denote this value of $\beta_2$ obtained by solving  $\rho(\beta_2) = 2\Phi + 3$ for $\beta_2$ by $\beta_2^{(\Phi+1)}$.

We also consider a second heuristic condition for ${\bm \ell}^*$ to appear in the reduced lattice basis.
While the previous condition ensured that ${\bm \ell}^*$ was  among the $\Phi+1$ short vectors satisfying the ${\tt B_1}$ constraints,
in some cases we may obtain a tighter bound on $\beta_2$ by restricting our attention to lattice vectors with $\gamma \ne 0$.
Since the lattice basis spans vectors with arbitrary integer values for $\gamma$, any basis for the full lattice must contain at least one vector with $\gamma\neq 0$.  Thus, even when many vectors with $\gamma=0$ are shorter than ${\bm \ell}^*$, we still expect ${\bm \ell}^*$ to appear if it is the shortest lattice vector with nonzero $\gamma$.  This can happen, for example, if $\beta_0$ is chosen large, as in \cref{eq:reduced_no_guess}.

While ${\bm \ell}^*$ has lattice coefficient $\gamma=1$,
there is no guarantee that a vector with $\gamma=1$ will appear in the reduced basis. 
Therefore, this alternate condition requires ${\bm \ell}^*$ to be the shortest among lattice vectors with any $\gamma\neq0$.  
We denote the value of $\beta_2$ required for this condition to hold by $\beta_2^{(\gamma\ne0)}$.
To compute $\beta_2^{(\gamma\ne0)}$,
we solve the equation $\rho(\beta_2)=2$ (to account for both ${\bm \ell}^*$ and $-{\bm \ell}^*)$,
but only need to consider the sum over  $\gamma\ne0$ in the formula for $\rho(\beta_2)$.

To obtain our final estimate for $\beta_2$,
we compute $\beta_2^{(\Phi+1)}$ and $\beta_2^{(\gamma\neq0)}$
from the two criteria above and take the smaller value.
Specifically,
$\beta_2^{(\Phi+1)}$ is obtained by solving
$\rho(\beta_2)=2\Phi+3$,
while $\beta_2^{(\gamma\neq0)}$ is obtained by solving
$\rho(\beta_2)=2$ after omitting the $\gamma=0$ contribution.
This gives the following heuristic.

\begin{heuristic}[$\beta_2$ estimate]
\label{heur:beta2_full}
 Defining $\gamma_{\max}$ and $R_\gamma$ by 
 \begin{equation*}
     \gamma_{\max} = \floor*{ \sqrt{\frac{\norm{{\bf x} - \overline{{\bf x}}}^2}{\beta_0^2} + 1}} \qquad , \qquad
        \Rgamma = \sqrt{\norm{{\bf x} - \overline{{\bf x}}}^2 + (1 - \gamma^2)\beta_0^2} , 
 \end{equation*}
 the resulting estimate for $\beta_2$ is 
 \begin{align*}
     \beta_2&=\min\set*{\beta_2^{(\Phi + 1)}, \beta_2^{(\gamma \ne 0)} }\nonumber \\ 
     &=\min\left\{\left( \frac{V_\Phi}{(2\Phi + 3)J!d_N} \left( \frac{2}{N} \right)^J \sum_{\gamma=-\gamma_{\max}}^{\gamma_{\max}} \Rgamma^\Phi  \right)^{1/(2J)} ,\left( \frac{V_\Phi}{2J!d_N} \left( \frac{2}{N} \right)^J \sum_{1\le |\gamma|\le \gamma_{\max}} \Rgamma^\Phi  \right)^{1/(2J)}\right\}
 \end{align*}
\end{heuristic}

We emphasize again that this analysis is a heuristic,
making several approximations to arrive at the final estimates of $\beta_2$.
However,
our numerical simulations in \cref{sec:verify} largely support the accuracy of these estimates.
We also note that this analysis ignores the approximation factor of LLL \cref{eq:lll_approx} when reducing the lattice basis.
The $\beta_2$ analysis could account for the approximation factor
by counting the number of lattice vectors which satisfy,
\begin{equation} \label{eq:approx_ell_bound}
    \norm{{\bm \ell}} 
    \le \left( \frac{4}{4\delta - 1} \right)^{N/2}\norm{{\bm \ell}^*} ,
\end{equation}
instead of $\norm{{\bm \ell}} \le \norm{{\bm \ell}^*}$.
However, 
conducting the analysis with \cref{eq:approx_ell_bound} would drastically increase the estimated value of $\beta_2$, and this upper bound for the approximation is often not tight in practice~\cite{nguyen2006average,Aardal2000}.  The numerical results in \cref{sec:verify} further show that even when the approximation factor becomes significant,
our analysis is more accurate than using \cref{eq:approx_ell_bound}.


\subsection{Guess Error Distribution} \label{sec:K}

The  $\beta_2$ estimate in \cref{heur:beta2_full} depends on the unknown guess error $K = \norm{{\bf x} - \overline{\bf x}}$, through the quantities $\gamma_{\max}$ and $R_\gamma$.
Since the true signal is unknown in practice, \cref{heur:beta2_full} in its current form cannot generally be used directly. 
If, however, the signal is modeled probabilistically, then the induced distribution of $K$ can be used to obtain corresponding estimates for $\beta_2$.  

In this section, we analyze the case where the entries of ${\bf x}$ are independent binomial distributions,
\begin{equation*}
    x_n \sim \binomdist{L}{p},
    \qquad 0 \le n < N.
\end{equation*}
This is a natural model,
as it appears in the one-dimensional subproblems that arise from inverting a binary image (or more generally when the image entries are themselves i.i.d. binomial).
The same approach can, in principle, be adapted to other signal distributions.

Under the current model, 
the zero-frequency DFT coefficient is also binomial, 
with 
\begin{equation*}
  T \coloneqq \tilde{x}_0 
  = \sum_{n = 0}^{N - 1} x_{n} 
  \sim \binomdist{NL}{p} .
\end{equation*}
We condition the distribution of the entries $x_n$ on the realized value of $T$. 
This conditioning is reasonable as we always sample $\tilde{x}_0$, 
and all remaining Fourier coefficients depend strongly on its value. 
Conditioning fixes the total signal mass, 
and we will see that the remaining Fourier coefficients can be accurately modeled by normal random variables.
Conditioning on $T$, we can thus view $\mathbf{x}$ as a sample from a
multivariate hypergeometric distribution: the $T$ total successes are
distributed among $N$ groups of size $L$, with $x_n$ counting the
successes in group $n$.

Our goal is to characterize the distribution of the initial guess error $K$.  
 For $\overline{\bf x}$  defined in \cref{eq:guess} and the sampled-frequency set $F$ defined in \cref{eq:sampled_freq},
we have 
\begin{equation} \label{eq:K2}
K^{2} 
= \frac{1}{N}\sum_{\substack{k<N \\ k \notin F}} \abs{\tilde{x}_{k}}^{2}
= \frac{1}{N}\sum_{\substack{k<N/2 \\ k \notin F}}(\abs{\tilde{x}_{k}}^{2} + \abs{\tilde{x}_{N-k}}^2)
= \frac{1}{N}\sum_{\substack{k<N/2 \\ k \notin F}}(\abs{\tilde{x}_{k}}^{2} + \abs{\tilde{x}_{k}^*}^2)
= \frac{1}{N}\sum_{\substack{k<N/2 \\ k \notin F}}2\abs{\tilde{x}_{k}}^{2},
\end{equation}
where we have applied the discrete Parseval relation and used the conjugate symmetry of the DFT.
Note that this calculation use that $k = 0$ and, if $N$ is even, $k = \frac{N}{2}$ are both in $F$,
which holds as the uniqueness guarantee requires these frequencies.

As each $\tilde{x}_k$ in \cref{eq:K2} is the sum of random variables
\begin{equation*} 
    \tilde{x}_{k} = \sum_{n = 0}^{N-1} x_{n} \uroot{N}{kn},
\end{equation*}
we would like to model its real and imaginary parts of as approximately normal random variables.  The conditioned entries of ${\bf x}$, however, follow a multivariate hypergeometric distribution and are therefore not independent.  We handle this issue by first analyzing the corresponding multinomial model, for which the central limit theorem applies directly, and then incorporarting the finite population correction.  

A multinomial random vector admits a highly related construction. 
If we draw $T$ independent samples $n_1, \dots, n_{T}$ uniformly from $\set{0,\dots,N-1}$ (with replacement), 
and let $y_n$ denote the number of times index $n$ is selected, 
then 
\begin{equation*}
    \Re[\tilde{y}_{k}] 
    = \sum_{n=0}^{N-1} y_n \cos\left( \frac{2\pi kn}{N} \right)
    = \sum_{m=1}^T \cos\left( \frac{2\pi kn_m}{N} \right)
\end{equation*}
is the sum of $T$ i.i.d. random variables.
Since $\bE\event{y_n}=T/N$, the expected value of $\tilde{y}_{k}$ can be computed by
\begin{equation}
\label{eq:Ex_k}
    \bE\event{\tilde{y}_{k}} 
    = \bE\event*{\sum_{n=0}^{N-1} y_{n}\uroot{N}{nk}}
    = \sum_{n=0}^{N-1} \bE\event{y_{n}} \uroot{N}{nk} 
    = \sum_{n=0}^{N-1} (T/N) \uroot{N}{nk} 
    = 0,
\end{equation}
where we have used the orthogonality of the discrete Fourier basis and the fact that $k \ne 0$ as $k=0$ is in the sampled set $F$.
Therefore, applying the central limit theorem to $\Re[\tilde{y}_{k}]$ gives the approximate distribution
\begin{equation} \label{eq:y_tilde_clt}
    \Re[\tilde{y}_k] \sim \normaldist{0}{T\var\left[ \cos\left( \frac{2\pi kn}{N} \right) \right]}.
\end{equation}
The variance can easily be computed analytically,
as
\begin{align*}
    \var\left[ \cos\left( \frac{2\pi kn}{N} \right) \right]
    = \bE\event*{\cos^{2}\left( \frac{2\pi kn}{N} \right)}
    - \bE\event*{\cos\left( \frac{2\pi kn}{N} \right)}^2 
    = \frac{1}{N} \sum_{n=0}^{N-1} \cos^{2}\left( \frac{2\pi kn}{N} \right) 
    = \frac{1}{2}.
\end{align*}
Now, to model $\Re[\tilde{x}_{k}]$, 
we note that the covariance matrices of the multivariate hypergeometric distribution and the multinomial distribution are related by a factor of the finite population correction~\cite{hypergeom-book}
\begin{equation*}
    \frac{NL-T}{NL-1}.
\end{equation*}
By applying this rescaling of the variance to the distribution in \cref{eq:y_tilde_clt}, and computing $\bE\event{\tilde{x}_k}=0$ as in \cref{eq:Ex_k},
we obtain the model,
\begin{equation*}
    \Re[\tilde{x}_{k}] \sim \normaldist{0}{\frac{T}{2} \cdot \frac{NL-T}{NL-1}} .
\end{equation*}
Applying an identical argument to the imaginary part yields 
\begin{equation*}
    \Im[\tilde{x}_{k}] \sim \normaldist{0}{\frac{T}{2} \cdot \frac{NL-T}{NL-1}} .
\end{equation*}
Now, 
we make the further approximation that 
\begin{equation*}
    \set{\Re[\tilde{x}_{k}], \Im[\tilde{x}_{k}] \given 0\le k < \floor{N/2} \text{ and } k \notin F}
\end{equation*}
is a set of independent random variables.  Numerical experiments indicate only weak dependence, 
supporting this approximation. 
By the expression for $K^2$ in \cref{eq:K2}, 
after dividing by the variance, 
this model approximates the random variable 
\begin{equation*}
    \frac{K^{2}}{\frac{T}{N} \cdot \frac{NL-T}{NL-1}}
\end{equation*}
as the sum of the squares of independent standard normal random variables.

To count the number of independent normal random variables in the sum in \cref{eq:K2} , we first note that the set $F$ in \cref{eq:sampled_freq} contains all but $\Phi - 2J$ coefficients.
Each term $\abs{\tilde{x}_k}^2 = \abs{\Re[\tilde{x}_k]}^2+\abs{\Im[\tilde{x}_k]}^2$ contributes two independent random variables. 
However,
since the sum in \cref{eq:K2} includes only one representative from each conjugate pair,
we count only half of the unsampled Fourier coefficients. 
These two factors cancel, 
so the total number of standard normal random variables is $\Phi-2J$.  
This specifies the number of degrees of freedom for the $\chi^2$ distribution of the sum,
giving
\begin{equation} \label{eq:K_dist}
    \frac{K^{2}}{\frac{T}{N} \cdot \frac{NL-T}{NL-1}}
    \sim \chisqdist{\Phi - 2J},
    \qquad
    T \sim \binomdist{NL}{p} .
\end{equation}
As $\bE\event{\chisqdist{n}} = n$, \cref{eq:K_dist} yields, conditional on $T$
\begin{equation*} 
    \bE\event{K^2} 
    \approx
(\Phi-2J)\frac{T}{N}\frac{NL-T}{NL-1}.
\end{equation*}
Taking expectation over $T\sim\operatorname{binom}(NL,p)$ and using $\mathbb E[T(NL-T)]=NL(NL-1)p(1-p)$, we obtain
\begin{equation}\label{eq:K_mean}
    \mathbb E[K^2]\approx(\Phi-2J)Lp(1-p).
\end{equation}
\Cref{eq:K_dist} provides an approximate distribution for the full-guess error $K$,
while \cref{eq:K_mean} gives the practical estimate
\begin{equation}\label{eq:K_mean_approx}
    K \approx \sqrt{\bE\event{K^2}}
    \approx \sqrt{(\Phi-2J)Lp(1-p)}.
\end{equation}
Substituting this value into the definitions of $\gamma_{\max}$ and $R_\gamma$ in \cref{heur:beta2_full}
yields a directly computable estimate for $\beta_2$ in terms of the known problem parameters
$N$, $\Phi$, $J$, $L$, $p$, and $\beta_0$.
More generally,
quantiles of the distribution in \cref{eq:K_dist} can be propagated through \cref{heur:beta2_full}
to obtain more conservative parameter choices corresponding to a desired probability level. Figure~\cref{fig:K_dist} compares the predicted distribution in \cref{eq:K_dist}
with Monte Carlo simulations for several choices of $J$, $L$, and $p$,
and shows close agreement with the model.

\begin{figure}[pos=!htb]
    \centering
    \includegraphics[width=\textwidth]{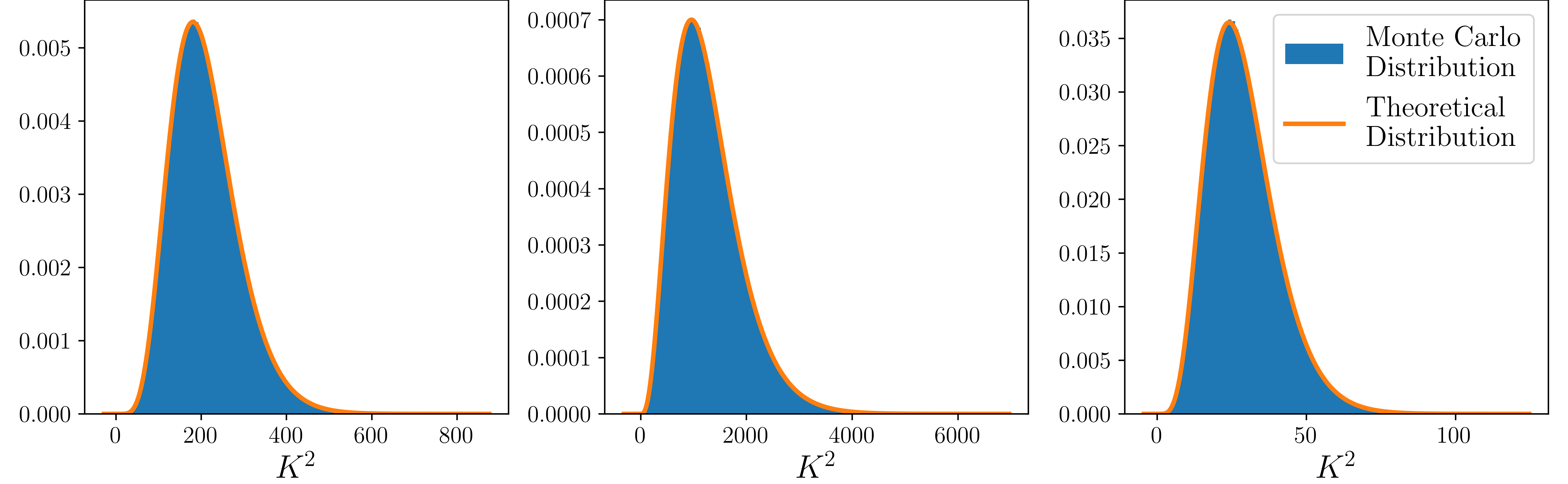}
    \vspace{-1.5pc}
    \caption{
    Distributions of $K$ for signals of length $N=60$ with entries distributed as $\binomdist{L}{p}$ for different values of $L$, $p$, and $J$.
    {\bf Left}: $L = 60$, $p = 0.5$, $J = 1$;
    {\bf Middle}: $L = 1000$, $p = 0.2$, $J = 4$;
    {\bf Right}: $L = 10$, $p = 0.6$, $J = 2$; The orange curve shows the theoretical distribution from \cref{eq:K_dist}, while the blue density estimate shows the results of a Monte Carlo simulation.
    }
    \label{fig:K_dist}
\end{figure}


\section{Alternative Guess Strategies}
\label{sec:guesses}
The analysis in \cref{sec:theory} assumes that the initial guess
$\overline{\bf x}$ incorporates all known DFT coefficients, as in \cref{eq:guess}.
It is also useful to consider alternative choices of the guess,
both for comparison with more standard lattice formulations and because the choice of
$\overline{\bf x}$ can affect the dependence of the reconstruction on the lattice parameters.
In particular,
the standard lattice formulation for integer feasibility problems uses the zero guess,
$\overline{\bf x}={\bf 0}$.
We also consider a modified Fourier-based guess which, as we will see later,
is less sensitive to the choice of $\beta_0$ than the full guess.
This section considers these two alternatives and describes the corresponding modifications to the estimate of $\beta_2$.

\subsection{Modified Fourier Guess} 
Our modified Fourier guess is similar to the guess in \cref{eq:guess}, but it does not incorporate every sampled coefficient.
To avoid confusion, we denote this modified guess by $\overline{{\bf x}}'$.  The modified guess
is defined in frequency space by
\begin{equation} \label{eq:guess_wo}
    \widetilde{\overline{x}}'_k
    = \begin{cases}
        0 & \gcd(k, N) = 1 \\
        \tilde{x}_k & \gcd(k, N) \ne 1,
    \end{cases}
\end{equation}
which differs from the full guess by ignoring the sampled DFT coefficients $\tilde{{x}}_{\pm k_j}$.
The analysis proceeds as in the full-guess case, but the omitted sampled coefficients shift the center of the feasible set. We therefore redefine the feasible sets and their radii using the modified guess.
Let 
\begin{equation} \label{eq:D_gamma_prime_def}
    D_\gamma' \coloneqq
    \set{ 
        {\bm \alpha} \in \R^{N} 
        \given \norm{{\bm \alpha} - \gamma\overline{{\bf x}}'} \le \Rgamma' 
        \text{ and } 
        {\bm \alpha}^{(N/p)} = \gamma{\bf x}^{(N/p)} \text{ for each prime } p\divs N
    },
\end{equation}
with radii $\Rgamma'$ derived from the length of the desired lattice vector which now depends on the modified guess error,
\begin{equation*}
    R_\gamma' 
    \coloneqq \sqrt{\norm{{\bm \ell}^*}^2 - \beta_0^2\gamma^2}
    = \sqrt{ \norm{{\bf x} - \overline{\bf x}'}^2 + (1 - \gamma^2)\beta_0^2} .
\end{equation*}
In this case, analogous to \cref{eq:prob_f},
we want to estimate the probability,
\begin{equation}
\label{eq:prob_guess1}
    \bP\event[\bigg]{\beta_2^2\sum_{j = 1}^{J} \abs{\tilde{\alpha}_{k_j} - \gamma\tilde{x}_{k_j}}^2 + \norm{{\bm \alpha} - \gamma\overline{\bf x}'}^2 \le (R_\gamma')^2
    \given {\bm \alpha} \in D_\gamma'}.
\end{equation}
We will reuse the map $\widehat{\bf z}$ from \cref{eq:z_hat_def},
and define a new coordinate transformation ${\bf z}'$ that is appropriately shifted by the modified guess,
\begin{equation*} 
    {\bf z}'(\bm \alpha) 
    = {\bf \widehat{z}}(\bm \alpha) - \gamma \overline{\bf z}' ,
    \qquad
    \overline{\bf z}' 
    = {\bf \widehat{z}}(\overline{\bf x}').
\end{equation*}
Note that just as in \cref{lem:prob},
${\bf z}'$ is an affine isometry that maps $D_\gamma'$ to the canonical embedding of a $\Phi$-dimensional sphere,
as it still translates the constraints ${\bm \alpha}^{(N/p)} = \gamma{\bf x}^{(N/p)}$ to coordinate constraints.  
Likewise,
the new coordinate representation again transforms the uniform distribution on $D_\gamma'$ to the uniform distribution on $F_\gamma' = {\bf z}'(D_\gamma')$,
which is spherically symmetric.

The first term in $\cref{eq:prob_guess1}$ can be rewritten in ${\bf z}'$-coordinates as
\begin{equation} \label{eq:B2_to_z_prime}
\begin{split}
    \beta_2^2\sum_{j = 1}^{J} \abs{\tilde{\alpha}_{k_j} - \gamma\tilde{x}_{k_j}}^2 
    &= \beta_2^2\sum_{j = 1}^{J} \abs{\Re[\tilde{\alpha}_{k_j}] - \gamma\Re[\tilde{x}_{k_j}]}^2 + \abs{\Im[\tilde{\alpha}_{k_j}] - \gamma\Im[\tilde{x}_{k_j}]}^2 \\
    &= \frac{\beta_2^2}{2}\sum_{j=1}^J \left[\left(\widehat{z}_{k_j}({\bm \alpha}) - \gamma\widehat{z}_{k_j}({\bf x})\right)^2 + \left(\widehat{z}_{N-k_j}({\bm \alpha}) - \gamma\widehat{z}_{N-k_j}({\bf x})\right)^2\right]\\
    &= \frac{\beta_2^2}{2}\sum_{j=1}^J \left[\left(z'_{k_j}({\bm \alpha}) - \gamma \widehat{z}_{k_j}({\bf x})\right)^2 + \left(z'_{N-k_j}({\bm \alpha}) - \gamma \widehat{z}_{N-k_j}({\bf x})\right)^2\right],
\end{split}
\end{equation}
where the last line follows from ${ z}_{k_j}'({\bm \alpha}) = \widehat{z}_{k_j}({\bm \alpha})$,
as excluding $\tilde{x}_{k_j}$ from the guess in \cref{eq:guess_wo} implies $\widehat{z}_{k_j}(\overline{\bf x}')=0$. 
The scaled orthogonality of $\widehat{\bf z}$ from \cref{lem:prob} gives the following expression for the second term of \cref{eq:prob_guess1} in ${\bf z}'$ coordinates,
\begin{equation} \label{eq:A_to_z_prime}
    \norm{{\bm \alpha} - \gamma\overline{\bf x}'}^2 
    = \frac{1}{N}\norm{\widehat{\bf z}({\bm \alpha} - \gamma\overline{\bf x}')}^2
    = \frac{1}{N} \norm{{\bf z}'({\bm \alpha})}^2 . 
\end{equation}
Recalling the definition of the projection matrix ${\tt P}$ from \cref{eq:projection},
we apply \cref{eq:B2_to_z_prime,eq:A_to_z_prime} to rewrite the random variable in \cref{eq:prob_guess1} as
\begin{equation*}
   Q({\bf z}')
   \coloneqq 
   \frac{\beta_2^2}{2} \norm*{{\tt P}\left({\bf z}' - \gamma\widehat{\bf z}({\bf x})\right) }^2 + \frac{1}{N}\norm{{\bf z}'}^2,
\end{equation*}
to obtain the equivalent probability distribution
\begin{equation*}
    \bP\event[\bigg]{\beta_2^2\sum_{j = 1}^{J} \abs{\tilde{\alpha}_{k_j} - \gamma\tilde{x}_{k_j}}^2 + \norm{{\bm \alpha} - \gamma\overline{\bf x}'}^2 \le (R_\gamma')^2
    \given {\bm \alpha} \in D_\gamma'} 
    = \bP\event[\bigg]{Q({\bf z}')\le (\Rgamma')^2 \given {\bf z}' \in F_\gamma'}.
\end{equation*}
We can now adapt the asymptotic normal argument from the previous section,
accounting for one additional term.  
Letting $r\coloneqq\norm{{\bf z}'({\bm\alpha})}$ and ${\bf c}\coloneqq\gamma\widehat{\bf z}({\bm x})/r$,
we rewrite $Q$ as
\begin{equation*}
    Q({\bf z}') = r^2\left(\frac{\beta_2^2}{2}\norm*{\frac{{\tt P}{\bf z}'}{r} - {\tt P}{\bf c} }^2 + \frac{1}{N}\right) .
\end{equation*}
Again,
\cref{eq:CLT_sphere} gives the approximation, 
${\tt P}{\bf z}'/r\approx\normaldist{{\bf 0}}{{\tt P}/\Phi}$.
Therefore,
conditioning on $r$ gives,
\begin{equation*}
    \norm*{\frac{{\tt P}{\bf z}'}{r} - {\tt  P}{\bf c} }^2
    \approx \frac{1}{\Phi}\norm{\normaldist{{\bf 0}}{{\tt P}}-\sqrt{\Phi}{\tt P}{\bf c}}^2 
    = \frac{1}{\Phi}\norm{\normaldist{-\sqrt{\Phi}{\tt P}{\bf c}}{{\tt P}}}^2 
\end{equation*}
Unlike the full-guess case, the sampled Fourier coordinates are no longer centered at the measured data. Consequently, the projected Gaussian approximation has a nonzero mean, leading to a noncentral chi-squared distribution.
We thus have,
\begin{equation} \label{eq:ncx2}
    \norm{\normaldist{-\sqrt{\Phi}{\tt P}{\bf c}}{{\tt P}}}^2 
    \sim \ncxsqdist{2J}{\lambda_\gamma},
    \qquad
    \lambda_\gamma = \Phi\norm{{\tt P}{\bf c}}^2 
    = \frac{\Phi\gamma^2}{ r^2} \sum_{j=1}^J \abs{\tilde{x}_{k_j}}^2 ,
\end{equation}
where $\lambda_\gamma$ is the noncentrality parameter.

In general,
the $\chi'^2$ CDF does not have a closed form.
However,
the first-order approximation,
\begin{equation*}
    \bP\event{\ncxsqdist{2n}{\lambda}\le x} 
    \approx \exp\left( -\frac{\lambda}{2} \right) \frac{(x/2)^n}{n!}
\end{equation*}
holds for small arguments~\cite{baricz2021ncx2},
generalizing the approximation in \cref{eq:chi2_cdf} to non-zero $\lambda$ through a factor of $\exp(-\lambda/2)$. 
Therefore,
as $r$ and ${\tt P}{\bf z}'/r$ are independent,
we can approximate the probability as 
\begin{align}
    \bP\event*{Q({\bf z}')\le  \norm{{\bf x} - \overline{\bf x}'}^2 + (1 - \gamma^2)\beta_0^2}
    &\approx \bP\event*{\frac{r^2\beta_2^2}{2\Phi}\chi'^2(2J,\lambda_\gamma)\le  \Rgammasq-\frac{r^2}{N}}\\ 
    &\approx \bE_r\event*{\exp\left( -\frac{\lambda_\gamma}{2} \right)\frac{1}{J!}\left(\Phi\frac{ \Rgammasq-\frac{r^2}{N}}{r^2\beta_2^2} \right)^J} \\
    &\approx \exp\left( -\frac{\lambda_\gamma}{2} \right)\frac{1}{J!}\left(\frac{2}{N\beta_2^2} \right)^J,
\end{align}
where we again applied the expected value of $r^2$ from \cref{eq:E_r2} throughout the expression in the last line.  
This includes substituting the expected value of $r^2$ into the noncentrality parameter in \cref{eq:ncx2} to obtain  
\begin{equation*}
    \lambda_\gamma \approx  \frac{2(\Phi+2)\gamma^2}{ \Rgammasq N} \sum_{j=1}^J \abs{\tilde{x}_{k_j}}^2 .
\end{equation*}
Substituting this probability into the analogous expression to \cref{eq:prob_intro}, we obtain 
our final expression for $\rho(\beta_2)$
\begin{equation} \label{eq:rho_final_guess_wo}
    \rho(\beta_2) 
    = \frac{V_\Phi}{J! d_N \beta_2^{2J}} \left( \frac{2}{N} \right)^J\sum_{\gamma = -\gamma'_{\max}}^{\gamma'_{\max}} (\Rgamma')^\Phi \exp\left( -\frac{\lambda_\gamma}{2} \right) , 
\end{equation}
where $\gamma'_{\max}$ is defined in terms of the modified guess
$\gamma'_{\max} = \floor[\big]{\sqrt{\norm{{\bf x}-\overline{\bf x}'}^2/\beta_0^2 + 1}}$.  As before, we estimate $\beta_2$ by taking the smaller of
$\beta_2^{(\Phi+1)}$, obtained from $\rho(\beta_2)=2\Phi+3$, and
$\beta_2^{(\gamma\neq0)}$, obtained from $\rho(\beta_2)=2$ after
omitting the $\gamma=0$ contribution. This gives the following heuristic.

\begin{heuristic}[$\beta_2$ estimate for modified Fourier guess]
\label{heur:beta2_modified}
 Defining $\gamma_{\max}'$, $R_\gamma'$, and $\lambda_\gamma'$ by 
 \begin{equation*}
     \gamma_{\max}' = \floor*{ \sqrt{\frac{\norm{{\bf x} - \overline{{\bf x}}'}^2}{\beta_0^2} + 1} } \qquad , \qquad
        \Rgamma' = \sqrt{\norm{{\bf x} - \overline{{\bf x}}'}^2 + (1 - \gamma^2)\beta_0^2} \qquad , \qquad \lambda_\gamma' = \frac{2(\Phi + 2)\gamma^2}{(\Rgamma')^2N} \sum_{j=1}^{J} \abs{\tilde{x}_{k_j}}^2, 
 \end{equation*}
 the resulting estimate for $\beta_2$ is 
 \begin{align*}\label{eq:beta2_guesswo}
     \beta_2&=\min\set*{\beta_2^{(\Phi + 1)}, \beta_2^{(\gamma \ne 0)} }\nonumber \\ 
     &=\min\left\{\left( \frac{V_\Phi}{(2\Phi + 3)J!d_N} \left( \frac{2}{N} \right)^J \sum_{\gamma=-\gamma'_{\max}}^{\gamma'_{\max}} (\Rgamma')^\Phi e^{-\lambda_\gamma'/2} \right)^{1/2J} ,\left( \frac{V_\Phi}{2J!d_N} \left( \frac{2}{N} \right)^J \sum_{1\le |\gamma|\le \gamma_{\max}'} (\Rgamma')^\Phi e^{-\lambda_\gamma'/2}\right)^{1/(2J)}\right\}
 \end{align*}
\end{heuristic}

As in \cref{sec:K}, we can model the unknown guess error $K'=\norm{{\bf x} - \overline{{\bf x}}'}$ in the modified-guess heuristic when the signal entries are binomial.
By again conditioning on $T=\tilde{x}_0$, 
the analysis gives the distribution of the
error of the modified guess as
\begin{equation} \label{eq:partial_guess_dist}
    \frac{\norm{{\bf x} - \overline{\bf x}'}^2}{\frac{T}{N} \cdot \frac{NL - T}{NL - 1}}
    \sim \chisqdist{\Phi} ,
    \qquad
    T \sim \binomdist{NL}{p}.
\end{equation}
This expression is analogous to \cref{eq:K_dist}, 
but there are $2J$ more degrees of freedom for the $\chi^2$ distribution as each $\tilde{x}_{\pm k_j}$ do not contribute to the guess.
We can also adapt this analysis to the sampled coefficients that contribute to the noncentrality parameter $\lambda_\gamma$ in \cref{heur:beta2_modified}.
This yields
\begin{equation} \label{eq:lam_dist}
    \frac{\sum_{j = 1}^J \abs{\tilde{x}_{k_j}}^2}{\frac{T}{2} \cdot \frac{NL - T}{NL - 1}}
    \sim \chisqdist{2J} ,
    \qquad
    T \sim \binomdist{NL}{p},
\end{equation}
where the extra factor of $\frac{N}{2}$ in the denominator occurs because we are excluding conjugate frequencies and start with the error in frequency space.
As the $\chi^2$ random variables in \cref{eq:partial_guess_dist,eq:lam_dist} both depend on the sampled coefficients $\tilde{x}_{\pm k}$,
they are dependent.
Thus,
we cannot simply substitute both $\chi^2$ models into \cref{heur:beta2_modified} to derive the full probabilistic $\beta_2$ estimate.

However,
the dependence between $\norm{{\bf x}-\overline{\bf x}'}$  and $\sum\abs{\tilde{x}_{k_j}}^2$ has a particularly simple form,
since the sampled coefficients form part of the total modified guess error by
\begin{equation*}
    \norm{{\bf x} - \overline{\bf x}'}^2
    = \frac{1}{N} \sum_{\gcd(k, N) = 1} \abs{\tilde{x}_k}^2
    = \frac{2}{N} \sum_{j=1}^J \abs{\tilde{x}_{k_j}}^2
    + \frac{1}{N} \sum_{ \substack{\gcd(k, N) = 1 \\ \pm k \notin \set{k_1, \dots, k_J}} } \abs{\tilde{x}_k}^2
    = \frac{2}{N} \sum_{j=1}^J \abs{\tilde{x}_{k_j}}^2
    + \norm{{\bf x} - \overline{\bf x}}^2 .
\end{equation*}
The first term is exactly the contribution from the sampled Fourier
coefficients, while the second is the error of the original guess
$\overline{\bf x}$.
Since these two sums involve disjoint sets of Fourier coefficients (which are modeled as independent),
the corresponding $\chi^2$ random variables are independent as well.  Therefore,
we define the independent random variables
\begin{equation*}
    X_1\sim\chi^2(2J)
\quad \text{and} \quad
X_2\sim\chi^2(\Phi-2J),
\end{equation*}
which,
by \cref{eq:partial_guess_dist,eq:lam_dist}, give the relevant distributions
\begin{equation} \label{eq:mod_guess_dist_refined}
    \norm{{\bf x}-\overline{\bf x}'}^2
    = \frac{T}{N}\frac{NL-T}{NL-1}(X_1+X_2) ,
    \qquad
    \sum_{j=1}^J\abs{\tilde x_{k_j}}^2
    = \frac{T}{2}\frac{NL-T}{NL-1}X_1 .
\end{equation}
Substituting these expressions into the equation for the  noncentrality parameter $\lambda_\gamma'$ yields
\begin{equation} \label{eq:lam_dist_ugly}
    \lambda_\gamma'
    = \frac{2(\Phi + 2)\gamma^2\sum_{j = 1}^{J} \abs{\tilde{x}_{k_j}}^2}{(\norm{{\bf x} - \overline{\bf x}'}^2 + (1 - \gamma^2)\beta_0^2)N}
    \sim \frac{
        (\Phi + 2)\gamma^2 T \frac{NL - T}{NL - 1} X_1
    }{
        \left( \frac{T}{N} \cdot \frac{NL - T}{NL - 1} (X_1 + X_2) + (1 - \gamma^2)\beta_0^2 \right)N
    } .
\end{equation}
Applying the same substitution to every occurrence of
$\norm{{\bf x}-\overline{\bf x}'}^2$
in \cref{heur:beta2_modified}
and replacing $\lambda_\gamma'$ by \cref{eq:lam_dist_ugly} gives a probabilistic model for $\beta_2$ in terms of the independent random variables $X_1, X_2$ and $T$.

For practical parameter selection, we can replace the random quantities in
\cref{heur:beta2_modified} by their expected values,
\begin{equation*}
    \bE\event*{(K')^2}  = \Phi Lp(1-p),
    \qquad
    \bE\event*{\sum_{j=1}^J \abs{\tilde{x}_{k_j}}^2}= JNLp(1-p).
\end{equation*}
Note that we can again use the approximation $K' \approx \bE[K'] \approx \sqrt{\bE[(K')^2]}$ when substituting into $\gamma'_{\max}$ and $R'_\gamma$. 

\subsection{Zero Guess}
We next consider the standard lattice formulation obtained by setting the initial guess to
$\overline{\bf x}^{(0)}={\bf 0}$.  In this case, 
we will show that the ${\tt B}_1$ constraints still determine affine spaces that are naturally centered at the modified Fourier guess
$\gamma\overline{\bf x}'$,
allowing much of the preceding analysis to be reused with a different family of radii.

With a zero guess,
the desired lattice vector has length
$\norm{{\bm \ell}^*}^2 = \norm{{\bf x}}^2 + \beta_0^2$.
Thus, for each $\gamma$,
the new set of feasible coefficients $D_\gamma^{(0)}$ is defined by 
\begin{equation}
\label{eq:D_gamma_0}
	D_\gamma^{(0)}
    \coloneqq \set*{{\bm \alpha} \in \R^N \given \norm{{\bm \alpha}} \le \sqrt{\norm{{\bf x}}^2 + (1 - \gamma^2)\beta_0^2 } \text{ and } {\bm \alpha}^{(N/p)} = \gamma{\bf x}^{(N/p)} \text{ for each prime } p \divs N} . 
\end{equation}
Letting ${\tt A}$ be the matrix from \cref{eq:A_def} which stacks the linear constraints ${\bm \alpha}^{(N/p)} = \gamma{\bf x}^{(N/p)}$ for each prime $p \divs N$, we have
\begin{equation}
\label{eq:A_system}
    {\tt A}{\bm \alpha} = \gamma {\tt A} {\bf x},
\end{equation}
for all ${\bm \alpha} \in D_\gamma^{(0)}$.
We write the corresponding affine solution space to the linear system in \cref{eq:A_system} as
\begin{equation*}
    H_\gamma \coloneqq \set{{\bm \alpha}\in\R^N \given {\tt A}{\bm \alpha} = \gamma {\tt A} {\bf x} } ,
\end{equation*}
and note that $D_\gamma^{(0)}\subset H_\gamma$.
We first note that, 
by the construction of the modified guess $\overline{\bf x}'$ in \cref{eq:guess_wo},
$\gamma\overline{\bf x}'$ is the unique least-norm solution to the linear system \cref{eq:A_system} for any fixed $\gamma$.
Equivalently, it is the unique element of $H_\gamma\cap\ker({\tt A})^\perp$.  We can thus rewrite any ${\bm \alpha}\in H_\gamma$ by
\begin{equation*}
    {\bm \alpha} = \gamma \overline{\bf x}' + ({\bm \alpha} - \gamma\overline{\bf x}') ,  
\end{equation*}
where the first term is in $\ker({\tt A})^\perp$ and the second term is in $\ker({\tt A})$.  As this is an orthogonal decomposition of ${\bm \alpha}$ we have 
\begin{equation*}
	\norm{{\bm \alpha}}^2
    = \norm{{\bm \alpha} - \gamma\overline{{\bf x}}'}^2 + \norm{\gamma\overline{{\bf x}}'}^2 .
\end{equation*}
This implies that the condition 
$\norm{{\bm \alpha}} \le \norm{{\bm \ell}^*}$ 
on the norm of ${\bm \alpha}$ 
in \cref{eq:D_gamma_0} is equivalent to the condition 
$\norm{{\bm \alpha} - \gamma\overline{\bf x}'}^2 \le \norm{{\bf x}}^2 + (1 - \gamma^2)\beta_0^2 - \gamma^2\norm{\overline{\bf x}'}^2$ on the distance from ${\bm \alpha}$ to the scaled modified guess.
Thus,
we can rewrite the feasible set $D_\gamma^{(0)}$ as
\begin{equation} \label{eq:D_gamma_0_better}
    D_\gamma^{(0)}
    = \set*{{\bm \alpha} \in \R^N \given \norm{{\bm \alpha} - \gamma\overline{{\bf x}}'} \le R_\gamma^{(0)} \text{ and } {\bm \alpha}^{(N/p)} = \gamma{\bf x}^{(N/p)} \text{ for each prime } p \divs N} ,
\end{equation}
for the radius $R_\gamma^{(0)}$ defined by
\begin{equation*}
    R_\gamma^{(0)} = \sqrt{\norm{{\bf x}}^2 + (1-\gamma^2)\beta_0^2 - \gamma^2\norm{\overline{\bf x}'}^2}
    = \sqrt{\norm{{\bf x} - \overline{{\bf x}}'}^2 + (1 - \gamma^2)(\beta_0^2 + \norm{\overline{{\bf x}}'}^2)} ,
\end{equation*}
where the last expression was obtained through the discrete Parseval relation.
\Cref{eq:D_gamma_0_better} expresses $D_\gamma^{(0)}$ as the intersection of a ball with an affine space containing its center.
As the nullity of ${\tt A}$ is $\Phi$,
this shows that $D_\gamma^{(0)}$ is a $\Phi$-dimensional ball centered at $\gamma\overline{\bf x}'$ with radius $R_\gamma^{(0)}$.  

Thus,
the zero-guess feasible set has exactly the same center and affine constraints as
$D_\gamma'$ in the previous section for the modified Fourier guess.  The only different is its radius.
Consequently,
the derivation of the noncentral chi-squared approximation carries over directly
after replacing
$R_\gamma'$ by $R_\gamma^{(0)}$
and $\gamma_{\max}'$ by an appropriate value of $\gamma_{\max}^{(0)}$.
This gives the following corresponding estimate for $\beta_2$.
\begin{heuristic}[$\beta_2$ estimate for zero guess]
\label{heur:beta2_noguess}
 Defining $\gamma_{\max}^{(0)}$, $R_\gamma^{(0)}$, and $\lambda_\gamma^{(0)}$ by 
 \begin{equation*}
     \gamma_{\max}^{(0)} =  \floor*{\sqrt{\frac{\norm{{\bf x}}^2 + \beta_0^2}{\norm{\overline{{\bf x}}'}^2 + \beta_0^2}}} \qquad , \qquad
        \Rgamma^{(0)} = \sqrt{\norm{{\bf x} - \overline{{\bf x}}'}^2 + (1 - \gamma^2)(\beta_0^2 + \norm{\overline{{\bf x}}'}^2)} \qquad , \qquad \lambda_\gamma^{(0)} = \frac{2(\Phi + 2)\gamma^2}{(\Rgamma^{(0)})^2N} \sum_{j=1}^{J} \abs{\tilde{x}_{k_j}}^2, 
 \end{equation*}
 the resulting estimate for $\beta_2$ is 
 \begin{align*}
     \beta_2&=\min\set*{\beta_2^{(\Phi + 1)}, \beta_2^{(\gamma \ne 0)} }\nonumber \\ 
     &=\min\left\{\left( \frac{V_\Phi}{(2\Phi + 3)J!d_N} \left( \frac{2}{N} \right)^J \sum_{\gamma=-\gamma^{(0)}_{\max}}^{\gamma^{(0)}_{\max}} (\Rgamma^{(0)})^\Phi e^{-\lambda_\gamma^{(0)}/2} \right)^{1/(2J)},\left( \frac{V_\Phi}{2J!d_N} \left( \frac{2}{N} \right)^J \sum_{1\le |\gamma|\le \gamma_{\max}^{(0)}} (\Rgamma^{(0)})^\Phi e^{-\lambda_\gamma^{(0)}/2}\right)^{1/(2J)}\right\}
 \end{align*}
\end{heuristic}
Note that when $\abs{\gamma} = 1$, $\Rgamma'=\Rgamma^{(0)}$, while when $\gamma = 0$,
$\Rgamma'<\Rgamma^{(0)}$, and
for $\abs{\gamma} > 1$, $\Rgamma'>\Rgamma^{(0)}$.
Accordingly,
the cutoff $\gamma_{\max}^{(0)}$ 
is typically significantly smaller $\gamma_{\max}'$.

It remains to determine how the unknown signal-dependent quantities in
\cref{heur:beta2_noguess}
should be modeled for practical parameter selection.
In contrast to the previous analyses, 
the radii $\Rgamma^{(0)}$ depend on both the error of the modified guess $\norm{{\bf x} - \overline{\bf x}'}$ and its length $\norm{\overline{\bf x}'}$.
A complete probabilistic analysis would therefore require adapting the previous calculations to model their joint distribution.
Rather than pursue this additional analysis,
we observe that it offers little practical benefit as the previous results already encompass the practically significant terms.

In particular,
the dominant contribution to $\beta_2$ in \cref{heur:beta2_noguess} comes from the $\abs{\gamma}=1$ terms of $\beta_2^{(\gamma\neq0)}$.
This is particularly convenient,
as these terms coincide exactly with the modified guess.
For these terms with $\abs{\gamma}=1$, 
\begin{equation*}
    R_1^{(0)} = \norm{{\bf x} - \overline{\bf x}'} = R_1',
\end{equation*}
so the model distributions in \cref{eq:mod_guess_dist_refined,eq:lam_dist_ugly} also apply to $\Rgamma^{(0)}$ and $\lambda_\gamma^{(0)}$.
The remaining $\abs{\gamma}\ne1$ terms have negligible impact on the final $\beta_2$ estimate for two reasons.
First,
$R_\gamma^{(0)}$ decreases rapidly with $\abs{\gamma}$,
so contributions from $\abs{\gamma}>1$ are relatively small.
Second,
as in 
\cref{eq:reduced_no_guess},
the longer ${\bm \ell}^*$ in the zero-guess case implies that there are many short vectors with $\gamma=0$, while we intend for ${\bm \ell}^*$ to be the shortest vector with $\gamma\neq0$.  Therefore, we expect 
$\beta_2^{(\gamma \ne 0)}<\beta_2^{(\Phi+1)}$. Since the final estimate is the minimum of these two quantities,
the $\gamma=0$ contribution, which appears only in
$\beta_2^{(\Phi+1)}$, does not contribute to the final $\beta_2$ estimate in \cref{heur:beta2_noguess}.

This reasoning is supported by the numerical results in
\cref{fig:each_beta2}.  The figure plots the two estimates $\beta_2^{(\Phi + 1)}$ and $\beta_2^{(\gamma \ne 0)}$ for a set of test signals with $N=30$,  entries identically and independently distributed as $\binomdist{N}{0.5}$,
using $J = 1$ with each guess.
For the zero guess,
$\beta_2^{(\gamma\ne 0)}$ is always the minimum.  Moreover, for the parameter ranges considered in \cref{fig:each_beta2},
all test signals had $\gamma_{\max}=1$.

For the full and modified guesses, in contrast, the criterion determining the final value of $\beta_2$ depends on $\beta_0$. For small $\beta_0$,
$\beta_2^{(\Phi+1)}$ is smaller,
whereas for larger $\beta_0$, as expected,
$\beta_2^{(\gamma\neq0)}$ becomes the smaller quantity.

\begin{figure}[pos=!htbp]
    \centering
    \includegraphics[width=.8\textwidth]{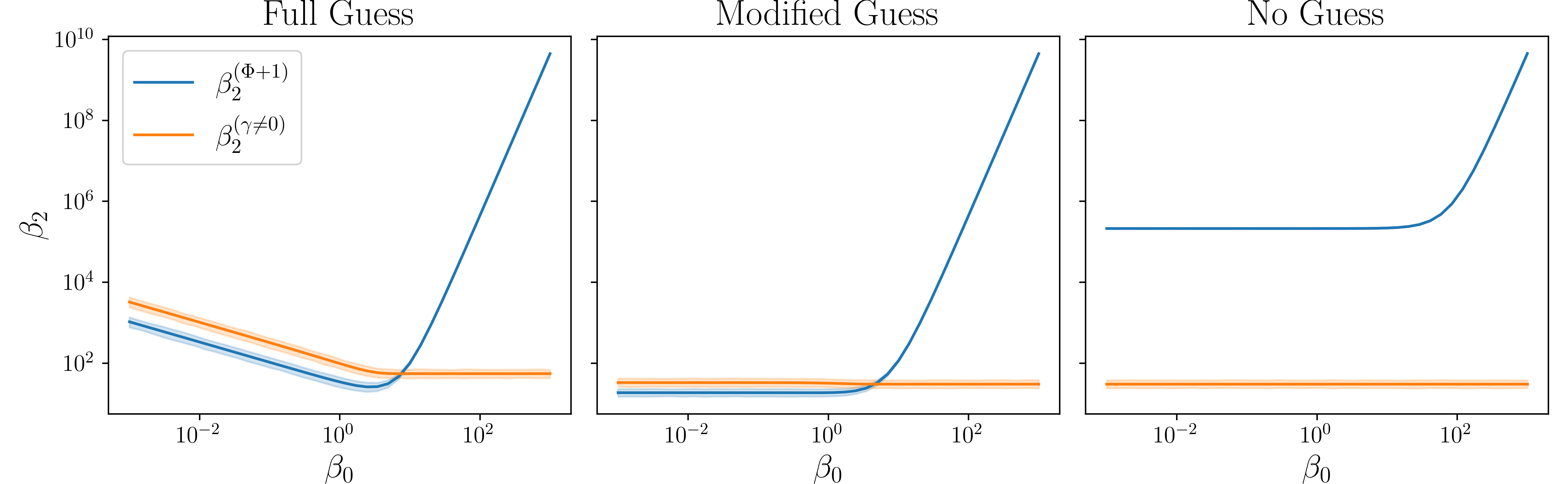}
    \vspace{-0.5pc}
    \caption{Values of $\beta_2^{(\Phi+1)}$ and $\beta_2^{(\gamma\neq0)}$ as functions of $\beta_0$ for 100 test signals of length $N=30$ with entries distributed as $\binomdist{300}{0.5}$, recovered using $J=1$.  For each of the three guess strategies, the plotted curves represent the average values, with shaded 95th percentile confidence intervals, over the 100 test signals.}
    \label{fig:each_beta2}
\end{figure}


\section{Analysis Verification} \label{sec:verify}

We now present numerical results supporting the analysis in \cref{sec:theory,sec:guesses}.
Our main objective is to evaluate how the parameter $\beta_2$ influences the recovery of random integer signals and to compare the theoretical predictions with experimental performance.  
In particular, we investigate the dependence of $\beta_2$ on the lattice parameter $\beta_0$ as well as the problem parameters $N$, $L$, and $J$, demonstrating that the heuristic estimates developed in \cref{sec:theory,sec:guesses} accurately predict the observed trends.  

Numerous experimental studies of LLL have shown that the rigorous bound in \cref{eq:lll_approx} significantly overestimates the approximation error encountered in practice~\cite{gama2008predictinglll,Aardal2000,nguyen2006average,nguyen1999crypto,backes2002heuristics}.
We will show that our numerical tests exhibit the same behavior. While the approximation factor influences the reconstruction performance, the effective approximation factor appears to be significantly smaller than the worst-case bound in \cref{eq:lll_approx}.  The results also demonstrate that these effects are much more significant when $\phi(N)$ is large. 

\subsection{\texorpdfstring{$\beta_0$}{beta0} and Guess Method Dependence} \label{sec:beta0}

We first investigate how the empirically required value of $\beta_2$ depends on the choice of $\beta_0$.
For $N=30$,
we generated a test set of 100 random signals with entries distributed as $\binomdist{10N}{0.5}$.
For each fixed value of $\beta_0$,
an exponential search over $\beta_2$ was performed for each test signal to determine the minimum value of $\beta_2$ required for successful reconstruction.
\Cref{fig:guess_comp} plots the mean empirically required values of $\beta_2$ as a function of $\beta_0$,
with the 95\% confidence interval illustrating the variability across the test set.
The columns of \cref{fig:guess_comp} consider each of the three guess cases: 
the full guess $\overline{\bf x}$ given in \cref{eq:guess}, 
the modified guess $\overline{\bf x}'$ given in \cref{eq:guess_wo},
which does not use $\tilde{x}_{\pm k_j}$, 
and the no-guess case.
For comparison,
each plot includes the theoretical distribution
obtained by calculating the theoretical estimate of  $\beta_2$ in \cref{heur:beta2_full,heur:beta2_modified,heur:beta2_noguess} for each test signal over the same range of $\beta_0$ values.

\begin{figure}[pos=!htb]
    \centering
    \includegraphics[width=\textwidth]{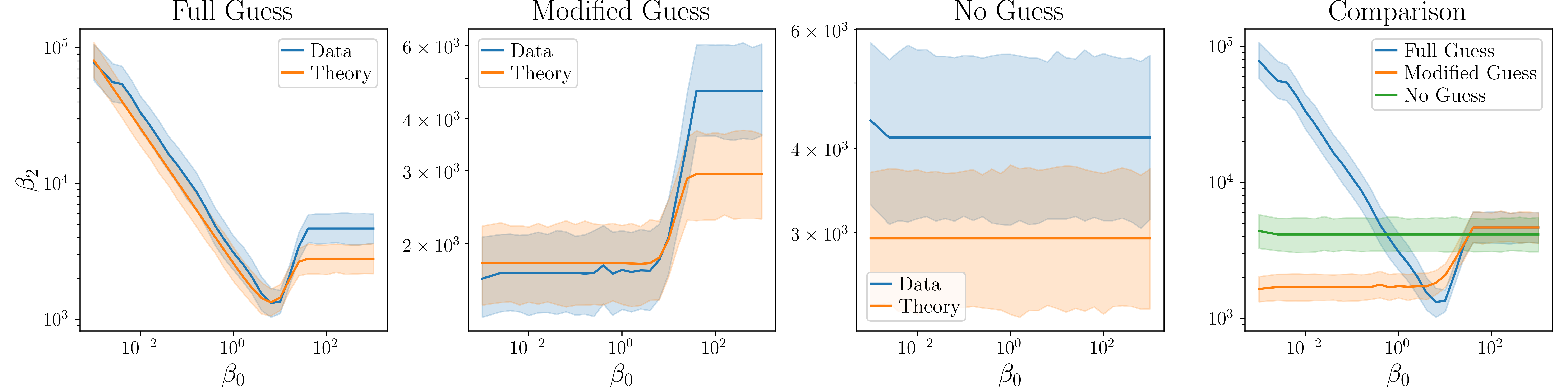}
\vspace{-1.5pc}
    \caption{Simulation of 100 test signals of length $N=30$ with entries distributed as $\binomdist{300}{0.5}$, using $J=1$.  The first three panels compare the theoretical and empirical distributions for each guess strategy, while the last panel compares the empirical distributions across all three guess strategies. Each curve plots the mean, with the shaded region representing the 95\% confidence interval.}
    \label{fig:guess_comp}
\end{figure}

Overall,
the theoretical predictions agree remarkably well with the experimental results.
The heuristic estimates accurately capture the shape of the empirical $\beta_2$ curves for all three guess strategies,
while providing a close approximation to the numerical values.
Among the guess strategies,
the full guess $\overline{\bf x}$ is by far the most sensitive to the choice of $\beta_{0}$,
as selecting a small value of $\beta_0$ makes the corresponding required value of $\beta_2$ quite large. 
In contrast,
the modified guess $\overline{\bf x}'$ performs best for smaller values of $\beta_0$,
while having only slight variation in $\beta_2$ across the entire range.
The no guess case is almost completely agnostic to the choice of $\beta_0$.

The final plot of \cref{fig:guess_comp} compares all three guess methods.
For sufficiently large values of $\beta_0$,
all three methods require similar values of $\beta_2$.
If $\beta_{0}$ is optimized,
the full guess $\overline{\bf x}$ performs best,
as it achieves the smallest required value of $\beta_2$ overall.
However,
the full guess $\overline{\bf x}$ is the most sensitive to choice of $\beta_0$,
and performs noticeably worse if $\beta_{0}$ is too small.
The minimum value of $\beta_2$ achieved by the modified guess $\overline{\bf x}'$ is only slightly worse,
and has close to optimal performance for small $\beta_{0}$ values. 
To avoid computationally expensive joint parameter optimization,
we default to the modified guess $\overline{\bf x}'$ for the remainder of the paper,
with a reasonably small value of $\beta_0 = 1 \times 10^{-1}$.
These experiments motivate our use of the modified Fourier guess in the remaining numerical studies:
although the full guess achieves a slightly smaller optimal value of $\beta_2$,
the modified guess remains close to optimal over a much wider range of $\beta_0$.

\subsection{Role of LLL Approximation} \label{sec:approxLLL}
We next repeated the numerical experiment from \cref{fig:guess_comp} with $N=31$.
The first column of \cref{fig:beta0} compares the required empirical and theoretical values of $\beta_2$ for the modified guess case $\overline{\bf x}'$ when $N=30$ and 31.  
For $N = 31$,
the empirical and theoretical curves show the same qualitative dependence on $\beta_0$,
but also display a significant numerical gap in $\beta_2$ magnitude.  
One possible explanation is that the LLL approximation factor plays a more substantial role for $N=31$.
To investigate this possibility,
we first consider how the theoretical analysis changes when the approximation factor of LLL is taken into account.

Motivated by the work in \cite{nguyen2006average},
which provides heuristic analysis and empirical support for a smaller average-case approximation factor bound than the worst-case in \cref{eq:lll_approx},
we assume that there exists some effective approximation factor $C$ (which depends on the dimension $N+1$) associated with applying LLL to the family of lattice bases in \cref{eq:lattice_basis}.
Our choice of $\beta_{1}$ can be easily modified so that the  approximation factor does not affect which reduced basis vectors satisfy ${\bm \ell}^{{({\tt B}_{1})}} = {\bf 0}$.
Specifically,
in comparison with \cref{eq:beta1},
we choose $\beta_{1}$ to satisfy 
\begin{equation*} 
	\beta_1 > C\norm{{\bm \ell}^*} = C\sqrt{\norm{{\bf x} - \overline{\bf x}'}^2+\beta_0^2},
\end{equation*}
so that every vector ${\bm \ell}$ with 
$\norm{{\bm\ell}} \le C\norm{{\bm \ell}^*}$ satisfies 
${\bm \ell}^{{({\tt B}_{1})}} = {\bf 0}$.
Since this elevated choice of $\beta_1$ is still significantly smaller than the estimate for $\beta_2$,
it produces no observable runtime effects.

To account for the approximation factor $C$,
we could apply our analysis  in \cref{sec:theory,sec:guesses} to count the number of lattice vectors with length bounded by $C\norm{{\bm \ell}^*}$.
As we can still assume that all such vectors satisfy 
${\bm \ell}^{{({\tt B}_{1})}} = {\bf 0}$,
the only change to the analysis would be the larger values of \Rgamma, which scales approximately by $C$.
This would introduce an additional factor $C^{\Phi}$ on the estimate for $\rho$ in \cref{eq:rho_final_guess},
in addition to expanding the range of $\gamma$ summed over. 
This $C^\Phi$ scaling explains why the approximation factor effects scale with $\phi(N)$, 
instead of directly increasing with $N$,
and suggests that they will be most apparent for prime values of $N$,
such as $N=31$.

To determine whether the discrepancy in \cref{fig:beta0} is indeed caused by the LLL approximation factor,
we repeated the experiments using exact lattice solvers that eliminate the approximation introduced by LLL.
Since the theoretical estimate in \cref{heur:beta2_modified} is determined by different terms in  the small- and large-$\beta_0$ regimes (see \cref{fig:each_beta2}),
we investigate these regimes separately using two different exact reduction formulations. 

Since $\beta_2^{(\Phi+1)}$ determines the theoretical bound for sufficiently small values of $\beta_0$,
we first isolate this regime by 
computing the exact $\phi(N) + 1$ shortest nonzero lattice vectors.
This exact solver was also implemented using the fpylll library,
by iteratively increasing a search radius
until calling the \texttt{Enumeration.enumerate} method produced $\phi(N) + 1$ vectors~\cite{hanrot2011algorithms,kannan1983enumeration}.
The theoretical value of $\beta_2^{(\Phi + 1)}$ from \cref{heur:beta2_modified} estimates the value of $\beta_2$ required to recover ${\bm \ell}^*$ by solving this exact lattice problem,
as it is the condition for the true solution to be one of the $2\Phi+3$ (including ${\bf 0}$ and negation) shortest lattice vectors.
The second column of \cref{fig:beta0} compares the empirically required value of $\beta_2$ for this exact solver with the corresponding theoretical prediction.
For both $N=30$ and $N=31$,
the theory accurately predicts the required values of $\beta_2$ and their dependence on $\beta_0$.
Since this experiment removes the approximation factor introduced by LLL, the close agreement strongly suggests that the discrepancy observed in the first column for $N=31$ at smaller values of $\beta_0$ is not a failure of the theoretical prediction, but rather reflects the increasing influence of the LLL approximation factor.

For larger values of $\beta_0$,
the $\beta_2$ value is determined by the $\gamma\neq0$ case.  
Although we are not aware of a  computationally tractable method to find the shortest lattice vector over all $\gamma \ne 0$,
we can compute the exact shortest vector in the lattice subject to the desired constraint  $\gamma = 1$.
This can be formulated as an exact closest vector problem to the lattice vector with coefficients ${\bm \alpha}={\bm 0}$ and $\gamma=1$,
which was implemented with the fpylll \texttt{closest\_vector} method, using the \texttt{method="proved"} setting~\cite{hanrot2011algorithms}.
The third column of \cref{fig:beta0} compares the empirically required values of $\beta_2$ for this exact $\gamma = 1$ solver with the theoretical prediction $\beta_2^{(\gamma=1)}$,
which is obtained by modifying the expression for  $\beta_2^{(\gamma\ne0)}$ to retain only the $\gamma = \pm1$ terms in the summation.
As expected,
neither the exact solver nor its theoretical prediction depends on $\beta_0$,
since every lattice vector with $\gamma=1$ has ${\bm \ell}^{({\tt B}_0)}=\beta_0$.
Again, there is close agreement between the theory and this exact solver for both $N=30$ and $N=31$.

The last column of \cref{fig:beta0} directly compares the LLL solver with these two exact solvers.
For $N = 30$,
the LLL solver effectively tracks the exact shortest vectors solver when $\beta_0$ is small,
and the exact $\gamma = 1$ solver when $\beta_0$ is large. 
For $N = 31$,
the LLL solver performed significantly worse than either exact solver across both regimes.  

\begin{figure}[pos=!htb]
    \centering
    \includegraphics[width=\textwidth]{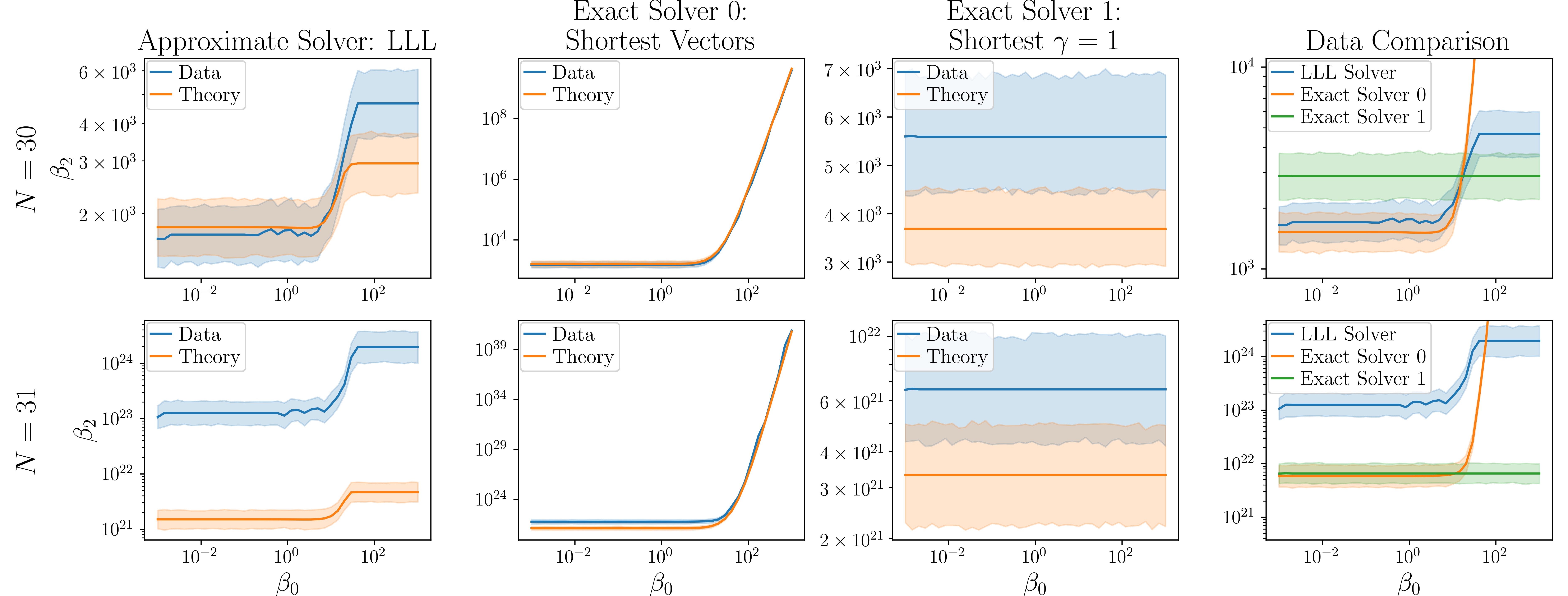}
    \vspace{-1.5pc}
    \caption{Simulation of 100 test signals of length $N=30$ and $N=31$ with entries distributed as $\binomdist{10N}{0.5}$ recovered using $J=1$.  The first row corresponds to $N=30$, and the second row to $N=31$. In each row, the first three panels compare the theoretical and empirical distributions for three different solvers: LLL, an exact solver that finds the $\Phi+1$ shortest lattice vectors, and an exact solver restricted to $\gamma=1$. The fourth panel compares the empirical distributions of the three solvers. Each curve plots the mean, with the shaded region representing the 95\% confidence interval.}
    \label{fig:beta0}
\end{figure}

To further investigate the role of the LLL approximation factor, 
we performed an additional test for $N = 30$ and $N = 31$.
With $\beta_0 = 1\times 10^{-1}$ fixed,
we measured 
the fraction of test signals recovered by the LLL solver and the exact solver over a range of $\beta_2$ values.
The resulting recovery curves are shown in \cref{fig:solvers} alongside theoretical $\beta_2$ curves.
As expected,
the theoretical,
LLL,
and exact solver recovery curves are nearly identical for $N = 30$.
For $N = 31$,
the theory and exact solver curves still agree very well,
further confirming the accuracy of our estimate. 
However,
the LLL recovery curve is shifted to the right of these two curves,
showing that larger $\beta_2$ values were required with LLL.

To distinguish failures caused by the LLL approximation from those caused by an insufficient value of $\beta_2$,
we classified an unsuccessful LLL reconstruction as an approximation-factor failure whenever ${\bm \ell}^*$ was one of the $\phi(N) + 1$ shortest lattice vectors (recovered with the first exact solver),
but all of the vectors in the LLL-reduced basis were longer than ${\bm \ell}^*$.  
This is indicated by the green curve in \cref{fig:solvers}.
For $N = 30$,
this  curve stays close to 0 for all $\beta_2$ values,
indicating that the approximation factor is insignificant.
For $N = 31$,
the green curve initially increases at a similar rate to the exact solver recovery.
In this regime,
the LLL recovery lies far below the exact recovery,
so the greens curve values indicate that most of these early failures are caused by the LLL approximation factor. 

We note that while this influence of the approximation factor causes the theory to underestimate the required values of $\beta_2$ for the $N = 31$ plot in \cref{fig:solvers},
it is still reasonably accurate.
One might instead attempt to account for the LLL approximation factor directly by incorporating the worst-case bound from \cref{eq:lll_approx} into the analysis as the effective approximation factor from above,
setting $C=(2 / \sqrt{4\delta - 1})^N$ as described in \cref{eq:approx_ell_bound}.
However, 
for $N = 31$,
this produces a theoretical $\beta_2$ estimate on the order of $10^{52}$,
which is many orders of magnitude larger than the empirically required value. Although rigorous,
this bound is far too pessimistic for practical parameter selection,
whereas our approximation remains comparatively accurate.

\begin{figure}[pos=!htb]
    \centering
    \includegraphics[width=\textwidth]{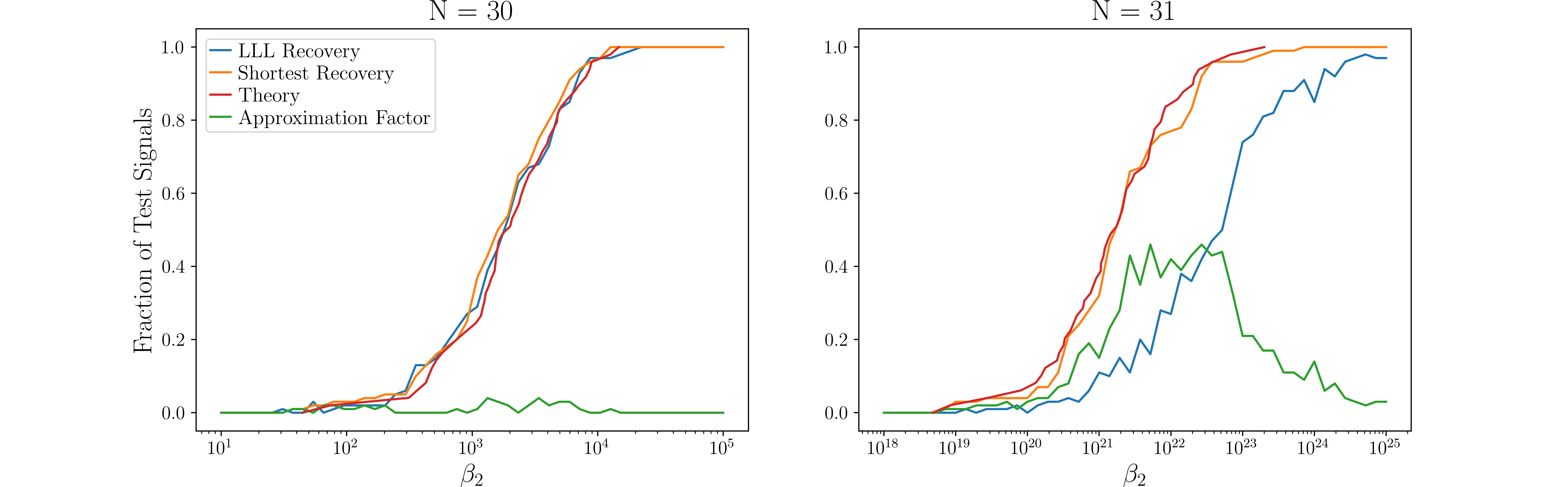}
        \vspace{-1.5pc}
    \caption{Simulation of 100 test signals of length $N=30$ and $N=31$ with entries distributed as $\binomdist{10N}{0.5}$, recovered using $J=1$.  For each value of $\beta_2$, the plots show: (1) the fraction of test signals successfully recovered using LLL; (2) the fraction successfully recovered using Exact Solver 0; (3) the theoretical prediction of the recovery probability; and (4) the fraction of instances in which the correct vector was shorter than the vectors returned by LLL, serving as a proxy for failures due to the LLL approximation factor.}
    \label{fig:solvers}
\end{figure}

\subsection{Problem Parameter Dependence}

With our choice of $\beta_0=1\times10^{-1}$ and using the modified guess $\overline{\bf x}'$,
we now investigate how the value of $\beta_2$ required for inversion varies with the problem parameters $N$, $L$, $K$, and $J$.

The divisor structure of the signal length $N$ determines both the number of unknowns and the amount of data available from subproblems.  
For example,
if $N$ is prime,
then there is only one measurement, $\tilde{x}_0$,
available from subproblems,
while composite $N$ may provide many additional coefficients (at least $N/2$ when $N$ is even).
The theoretical $\beta_2$ estimate incorporates this dependence through the totient function $\phi(N)$.
\Cref{fig:N} plots both experimental and theoretical $\beta_2$ values over a range of signal lengths with $J=1$ fixed.
The experimental values were computed as in \cref{sec:beta0} with $J = 1$ sampled coefficient,
using test sets of 100 randomly generated signals with entries distributed as $\binomdist{N}{0.5}$ for each $N$.

The theoretical estimates accurately capture the dependence of $\beta_2$ on $N$,
particularly for values with smaller totient $\phi(N)$.
As in \cref{sec:approxLLL},
we attribute the gap for larger $\phi(N)$ to the LLL approximation factor.
The gap is most apparent for the large prime values $N = 47$, 53, and 59,
where the LLL approximation factor is expected to have the greatest impact.  
Regardless,
the theory correctly predicts the relative size of $\beta_2$ across different signal lengths
and continues to provide a useful practical estimate.

\begin{figure}[pos=!htb]
    \centering
    \includegraphics[width=\textwidth]{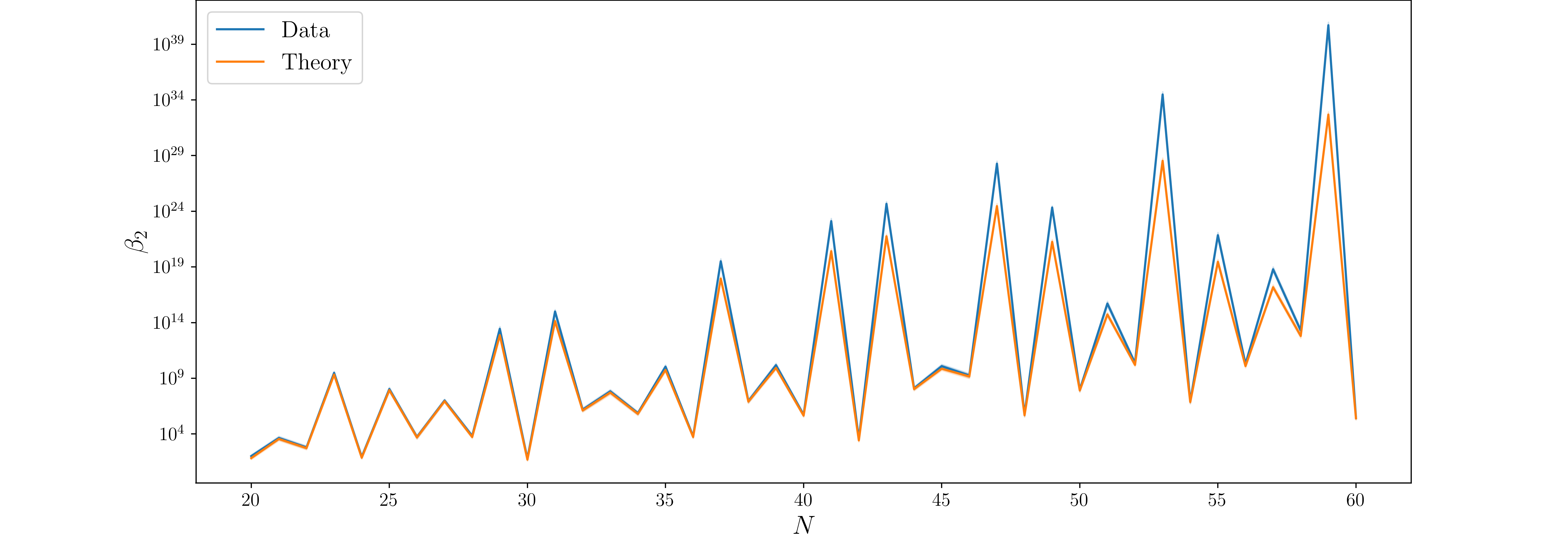}
    \vspace{-1.5pc}
    \caption{Simulation of 100 test signals for each $N\in\set{20,\dots,60}$, with entries distributed as $\binomdist{N}{0.5}$, recovered using $J=1$ and the modified guess $\overline{\bf x}'$  and $\beta_0 = 1\times 10^{-1}$.  For each value of $N$, the plot shows the average $\beta_2$ value (with 95\% confidence intervals)
  required for successful recovery, as determined empirically and predicted theoretically.}
    \label{fig:N}
\end{figure}

Next,
we investigate how the required value of $\beta_2$ depends on the amount of top-level data available, $J$.
For $N = 41$ and $N = 42$,
we computed the empirical value of $\beta_2$ required to invert 100 randomly generated test signals with entries distributed as $\binomdist{10N}{0.5}$.
\Cref{fig:J} compares the average across the test set of these empirical $\beta_2$ values  with the corresponding theoretical predictions for $1 \le J < \frac{\phi(N)}{2}$.
For $N=42$,
the curves are nearly indistinguishable,
while for the prime case $N=41$,
the theoretical curve lies slightly below the empirical values across all $J$,
consistent with the stronger influence of the LLL approximation factor observed for prime signal lengths.
In both cases,
the required value of $\beta_2$ decreases rapidly as $J$ increases,
indicating that reconstruction stability can be substantially improved by sampling more than the minimal data set.
We explore this idea further in \cref{sec:full}, where additional top-level measurements enable the recovery of larger signals and images.

\begin{figure}[pos=!htb]
    \centering
    \includegraphics[width=.8\textwidth]{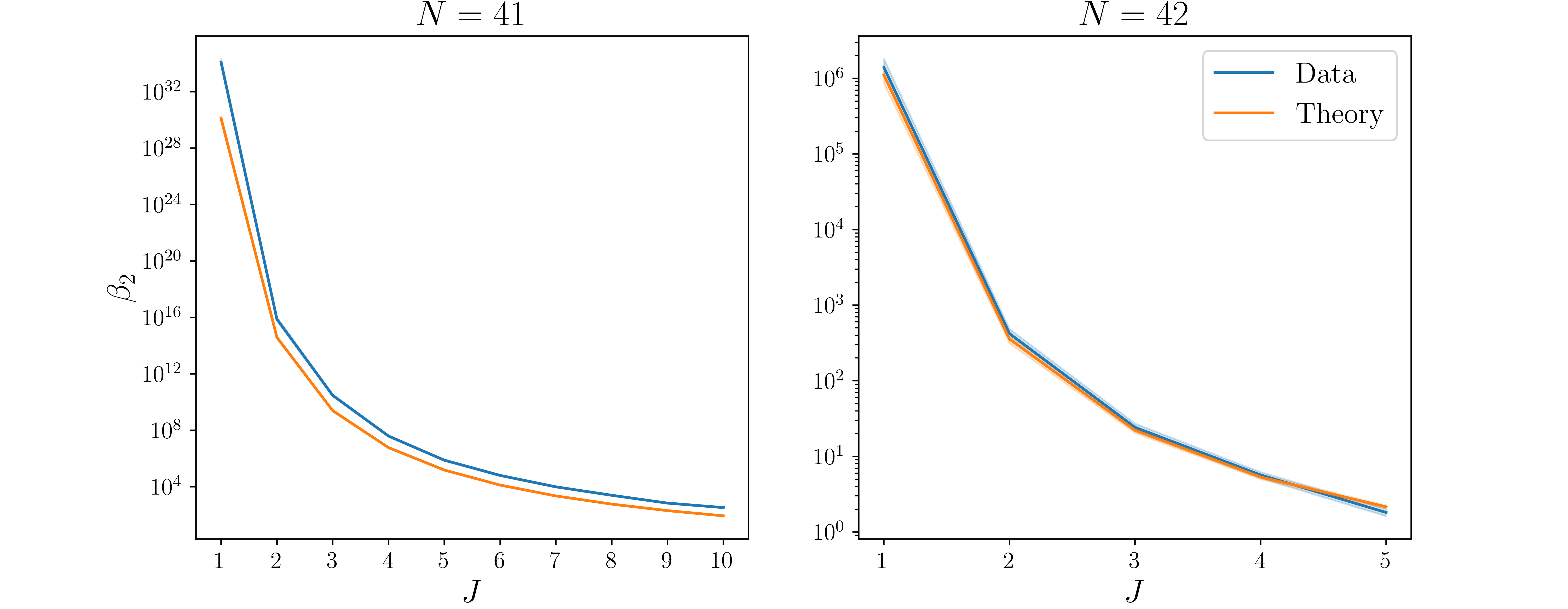}
    \vspace{-0.5pc}
    \caption{Simulation of 100 test signals for each $J$ from 1 to $\floor{\phi(N)/2} - 1$ with $N=41$ (left) and $N=42$ (right).  Signal entries were distributed as $\binomdist{10N}{0.5}$, and recovered using the modified guess $\overline{\bf x}'$ and $\beta_0 = 1\times 10^{-1}$.  For each value of $J$, the plot shows the average $\beta_2$ value (with 95\% confidence intervals) required for successful recovery, as determined empirically and predicted theoretically.}
    \label{fig:J}
\end{figure}

Finally,
we investigate the dependence of $\beta_2$ on the binomial parameter $L$ and the error of the guess $K = \norm{{\bf x} - \overline{\bf x}}$.
For this numerical test,
we returned to using the full guess $\overline{\bf x}$,
as it depends directly on $K$ as seen in \cref{heur:beta2_full}.
In contrast,
for the modified guess $\overline{\bf x}'$,
the noncentrality parameter $\lambda_\gamma$ not only depends on the modified guess error,
but how $K$ is split between the sampled and non-sampled coefficients.
With $J = 1$ fixed,
we generated 100 test signals for each of $N=36$ and 37, 
with entries distributed as $\binomdist{L}{0.5}$ for each $L=100, 1{,}000$, and  $10{,}000$. 
For every test signal,
we computed both the empirical $\beta_2$ value and the initial guess error $K$.  
\cref{fig:K} displays these paired data on a scatter plot.
Since $N$ and $J$ are fixed,
the theoretical $\beta_2$ depends only on the guess error,
so we plotted the theoretical curve as a function of $K$.
We also plot a horizontal line showing the result of substituting the estimate $\bE\event{K}$ in \cref{eq:K_mean_approx} to the theoretical expression for $\beta_2$.

The experimental results support the predicted dependence of $\beta_2$ on $K$.
For $N=36$,
the theoretical curve passes through the center of the empirical distributions over the full range of $K$.
For the prime case $N = 37$,
the curve captures the overall trend but consistently underestimates the required $\beta_2$,
again likely due to the LLL approximation factor.  In both cases, the average theory estimate obtained from $\bE\event{K}$ provides a useful practical approximation.
We note that for both $N$,
the empirical $\beta_2$ values exhibit some non-monotonicity as functions of $K$.
However,
we observed this behavior only with the LLL solver and not with the exact solver,
and thus attribute it to minor approximation factor effects.

\begin{figure}[pos=!htb]
    \centering
    \includegraphics[width=\textwidth]{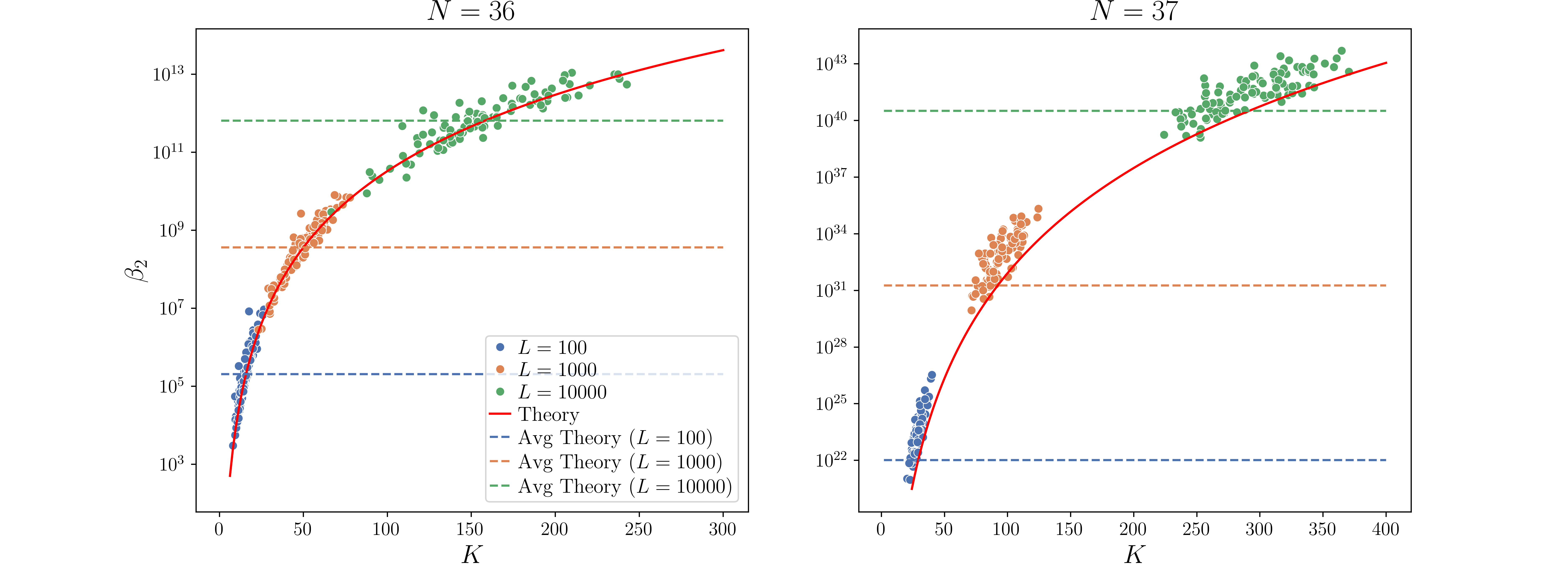}
    \vspace{-1.5pc}
    \caption{Simulation of 100 test signals with $N=36$ (left) and $N=37$ (right).  For each $L\in\set{100, 1{,}000, 10{,}000}$, signal entries were distributed as $\binomdist{L}{0.5}$, and recovered using $J=1$ with the full guess $\overline{\bf x}$ and $\beta_0 = 1\times 10^{-1}$.  For each signal in the test set, we plot its actual $K$ value against the  empirical value of $\beta_2$
    required for successful recovery.  The red curve plots the theoretical required value of $\beta_2$ against $K$.}
    \label{fig:K}
\end{figure}


\section{Full Inversion} \label{sec:full}

All of the preceding numerical experiments considered a single subproblem of \cref{alg:memo_1D},
matching the scope of the theoretical analysis.
In practice,
however,
reconstructing an integer signal requires solving every subproblem.
In this section,
we first examine the collection of subproblems arising in the inversion of a $210\times210$ image with entries from $0$ to $L = 255$,
and then present two practical image inversion examples.
Our numerical tests catalog the associated values of $L$,
$N$,
and $J$,
and then demonstrate image reconstruction from both the minimal sampling set and expanded sampling sets sufficient for double- and single-precision arithmetic.

\subsection{1D Inversion Subproblems} \label{sec:subprob}

Consider the application of \cref{lem:uniq} to an $N \times N$ integer image ${\tt X}$.
The group $\Z_N \times \Z_N$ contains maximal cyclic subgroups of order $N$,
as well as cyclic subgroups of order $N'$ for every divisor $N' \divs N$.
Therefore,
recovering ${\tt X}$ requires solving one-dimensional subproblems of size $N'$ for every $N' \divs N$.
Throughout this section,
we continue to model the image entries as independent binomial random variables,
\begin{equation*}
    X_{mn} \sim \binomdist{L}{0.5},
    \qquad 0 \le m, n < N.
\end{equation*}
Each entry of a subproblem of length $N'$ is the sum of $\frac{N^2}{N'}$ entries of ${\tt X}$ and is therefore distributed as $\binomdist{\frac{N^2L}{N'}}{0.5}$~\cite{levinson2025recovery,grigoryan-book}.

For the $210\times210$ image,
we take $N=210$ and $L=255$.
For each subproblem size $N'\divs210$,
we generated 100 one-dimensional test signals of length $N'$ with entries distributed as $\binomdist{(N/N')L_0}{0.5}$,
where $L_0 = 255N$ is the binomial parameter for the largest subproblems of length $N$.
We omitted the subproblems of lengths $1$, $2$, $3$, and $6$ from the numerical experiments,
since the minimal sampling set (together with conjugate symmetry) contains every DFT coefficient for these cases,
making reconstruction immediate.

\Cref{tab:subprobs} presents selected percentiles of the minimum $\beta_2$  values required for successful recovery.
For each test signal,
we computed the minimal empirical value of $\beta_2$ required for successful recovery,
as well as the theoretical $\beta_2$ value from \cref{heur:beta2_modified}.
The percentiles were computed using a log-linear interpolation from the associated distribution of $\beta_2$ values~\cite{hyndman1996quantile}. 
For each subproblem size,
we include the results for recovery from the minimal sampling set ($J=1$),
together with the smallest values of $J$ that make single- or double-precision recovery feasible,
when these differ from $J=1$.
These precision labels are based on the 99th percentile of the required $\beta_2$ value.
Values below $10^7$ are classified as single precision,
while values below $10^{14}$ are classified as double precision.
These thresholds were chosen after sufficient testing,
although they could likely be refined. However,
they are consistent with the stability analysis of \cref{sec:stability},
which predicts that the largest usable value of $\beta_2$ is determined by the available measurement precision.
Since single- and double-precision arithmetic provide approximately 7 and 16 decimal digits of accuracy,
respectively,
the chosen thresholds are of the expected magnitude. 

\begin{table}
    \centering
    \setlength{\tabcolsep}{4pt}

\begin{tabular}{crc|S[table-format=1.1e2, scientific-notation=true]S[table-format=1.1e2, scientific-notation=true]|S[table-format=1.1e2, scientific-notation=true]S[table-format=1.1e2, scientific-notation=true]|S[table-format=1.1e2, scientific-notation=true]S[table-format=1.1e2, scientific-notation=true]|c}
\multirow{2}{*}{$N$} & \multicolumn{1}{c}{\multirow{2}{*}{$L$}} & \multirow{2}{*}{$J$} & \multicolumn{2}{c|}{50th Percentile} & \multicolumn{2}{c|}{90th Percentile} & \multicolumn{2}{c|}{99th Percentile} & \multirow{2}{*}{Precision}\\
& & & {Data} & {Theory} & {Data} & {Theory} & {Data} & {Theory}\\
\midrule\midrule
5 & $42L_0$ & 1 & 5.5e+05 & 6.9e+05 & 1.6e+06 & 1.6e+06 & 3.5e+06 & 3.0e+06 & Single \\
\midrule
7 & $30L_0$ & 1 & 8.4e+08 & 8.9e+08 & 2.7e+09 & 2.2e+09 & 4.7e+09 & 5.1e+09 & Double \\
 & & 2 & 9.4e+03 & 1.4e+04 & 2.3e+04 & 2.4e+04 & 4.2e+04 & 3.7e+04 & Single \\
\midrule
10 & $21L_0$ & 1 & 9.7e+04 & 1.4e+05 & 3.0e+05 & 3.1e+05 & 6.1e+05 & 4.9e+05 & Single \\
\midrule
14 & $15L_0$ & 1 & 7.3e+07 & 6.8e+07 & 2.7e+08 & 2.5e+08 & 5.3e+08 & 6.2e+08 & Double \\
 & & 2 & 2.7e+03 & 3.5e+03 & 7.2e+03 & 6.8e+03 & 1.2e+04 & 8.3e+03 & Single \\
\midrule
15 & $14L_0$ & 1 & 6.2e+10 & 6.5e+10 & 2.9e+11 & 2.4e+11 & 1.3e+12 & 8.3e+11 & Double \\
 & & 2 & 1.2e+05 & 1.1e+05 & 2.6e+05 & 2.1e+05 & 3.9e+05 & 2.7e+05 & Single \\
\midrule
21 & $10L_0$ & 1 & 3.2e+16 & 2.6e+16 & 1.3e+17 & 9.6e+16 & 3.0e+17 & 2.3e+17 & Extended \\
 & & 2 & 6.7e+07 & 6.1e+07 & 1.5e+08 & 1.2e+08 & 2.1e+08 & 1.7e+08 & Double \\
 & & 3 & 9.2e+04 & 8.0e+04 & 1.6e+05 & 1.2e+05 & 2.2e+05 & 1.7e+05 & Single \\
\midrule
30 & $7L_0$ & 1 & 3.2e+09 & 2.8e+09 & 1.3e+10 & 1.3e+10 & 6.3e+10 & 3.6e+10 & Double \\
 & & 2 & 2.0e+04 & 1.9e+04 & 4.7e+04 & 4.0e+04 & 6.9e+04 & 5.2e+04 & Single \\
\midrule
35 & $6L_0$ & 1 & 3.1e+34 & 1.7e+33 & 2.9e+35 & 1.4e+34 & 1.3e+36 & 4.1e+34 & Extended \\
 & & 3 & 5.4e+10 & 2.7e+10 & 1.2e+11 & 5.6e+10 & 2.1e+11 & 7.2e+10 & Double \\
 & & 5 & 1.2e+06 & 6.7e+05 & 2.1e+06 & 1.0e+06 & 2.8e+06 & 1.3e+06 & Single \\
\midrule
42 & $5L_0$ & 1 & 5.2e+14 & 4.3e+14 & 2.2e+15 & 1.5e+15 & 9.2e+15 & 3.2e+15 & Extended \\
 & & 2 & 7.6e+06 & 7.0e+06 & 1.5e+07 & 1.2e+07 & 2.3e+07 & 2.0e+07 & Double \\
 & & 3 & 1.9e+04 & 1.6e+04 & 3.1e+04 & 2.4e+04 & 4.7e+04 & 3.1e+04 & Single \\
\midrule
70 & $3L_0$ & 1 & 4.9e+30 & 2.8e+29 & 5.6e+31 & 2.6e+30 & 3.6e+32 & 1.1e+31 & Extended \\
 & & 3 & 2.2e+09 & 1.3e+09 & 7.1e+09 & 2.8e+09 & 1.2e+10 & 5.0e+09 & Double \\
 & & 5 & 1.9e+05 & 9.3e+04 & 3.3e+05 & 1.5e+05 & 5.0e+05 & 2.4e+05 & Single \\
\midrule
105 & $2L_0$ & 1 & 5.1e+65 & 1.2e+59 & 8.9e+66 & 2.0e+60 & 5.0e+67 & 6.1e+60 & Extended \\
 & & 5 & 1.2e+12 & 6.8e+10 & 2.9e+12 & 1.2e+11 & 4.5e+12 & 1.6e+11 & Double \\
 & & 9 & 1.8e+06 & 2.3e+05 & 3.0e+06 & 3.2e+05 & 4.6e+06 & 3.8e+05 & Single \\
\midrule
210 & $L_0$ & 1 & 1.0e+58 & 3.3e+51 & 2.4e+59 & 4.2e+52 & 2.3e+60 & 4.5e+53 & Extended \\
 & & 5 & 2.9e+10 & 1.6e+09 & 7.4e+10 & 2.6e+09 & 1.1e+11 & 3.9e+09 & Double \\
 & & 8 & 1.2e+06 & 1.4e+05 & 1.8e+06 & 1.9e+05 & 3.2e+06 & 2.6e+05 & Single \\
\end{tabular}

\setlength{\tabcolsep}{6pt}

    \caption{Selected percentiles of the minimum $\beta_2$ values required for successful recovery of one-dimensional subproblems arising in the inversion of a $210\times210$ image. For each divisor $N' \divs 210$, $100$ test signals with entries distributed as $\binomdist{(210/N')L_0}{0.5}$ were generated. Empirical percentiles are compared with the theoretical predictions of \cref{heur:beta2_modified}. Precision labels indicate whether the $99^\mathrm{th}$ percentile of the required $\beta_2$ lies within the range of single- or double-precision.}
    \label{tab:subprobs}
\end{table}

Overall,
the experimental data in \cref{tab:subprobs} support the theoretical $\beta_2$ estimates. 
For subproblems of length $N\le70$,
the theoretical and empirical $\beta_2$ values agree remarkably well across the reported percentiles and values of $J$.
For the larger subproblems of lengths $105$ and $210$,
the empirical values are consistently larger than the theoretical predictions,
reflecting the increased impact of the LLL approximation factor.

The table also highlights the inherent instability of reconstruction from the minimal sampling set.
For $J=1$,
the largest subproblems ($N=105$ and $210$) require extremely large values of $\beta_2$,
corresponding to 60 to 70 decimal digits of precision in the sampled DFT coefficients.
Fortunately,
the results also 
demonstrate that inversion at single or double precision requires only a modest increase in the number of measurements.
For example,
only four additional DFT samples are required at each of the $N=105$ and $N=210$ levels for double-precision recovery.
This still represents only a small fraction of the full DFT data,
so the reconstruction problem remains highly underdetermined.

One surprising feature of \cref{tab:subprobs} is that the smaller subproblem of size 105 requires a larger value of $\beta_2$ than the full problem of size 210.
More generally,
the ordering of the required $\beta_2$ values is not determined solely by the subproblem size.
While smaller subproblems recover shorter signals, 
they also involve larger integer bounds, and these competing effects help determine the required value of $\beta_2$. 

Consider a signal of size $N$ with an immediate subproblem of size $N'=N/p$, 
for some prime factor $p$ of $N$.
Since the estimates of $\beta_2$ depend primarily on $\phi(N)$,
the guess error $K$,
and the lattice determinant $d_N$,
we compare these quantities between the two subproblems.
Euler's identity 
\begin{equation} \label{eq:totient_id}
    \phi(n) = n\prod_{\text{prime }q \divs n} \left(1 - \frac{1}{q}\right),
\end{equation}
implies that, if $p \divs N'$, then 
\begin{equation} \label{eq:sub_totient_divs}
    \phi(N') = N'\prod_{q \divs N'}\left(1 - \frac{1}{q}\right) 
    = \frac{N}{p}\prod_{q \divs N}\left(1 - \frac{1}{q}\right) 
    = \frac{\phi(N)}{p} .
\end{equation}
On the other hand, if $p \ndivs N'$, then the products in \cref{eq:totient_id} for $N$ and $N'$ differ by a factor of $1 - 1/p$, giving
\begin{equation} \label{eq:sub_totient_ndivs}
    \phi(N') = N' \prod_{q \divs N'}\left(1 - \frac{1}{q}\right)
    = \frac{N}{p} \cdot \frac{\prod_{q \divs N}\left(1 - 1/q\right)}{(1 - 1/p)}
    = \frac{\phi(N)}{p(1-1/p)}
    = \frac{\phi(N)}{p-1}.
\end{equation}
If the entries of ${\bf x}^{(N)}$ are distributed as $\binomdist{L}{0.5}$,
then the decimated signal ${\bf x}^{(N')}$ has entries distributed as $\binomdist{L'}{0.5}$, where
$L' = N/N' \cdot L = pL$.
Substituting these values into the approximation of \cref{eq:K_mean} gives
\begin{equation} \label{eq:sub_K}
        \bE\event{K_{N'}^2}
    \approx
    \frac{1}{2}(\phi(N')-2J)L'
    = \cdots . \begin{cases}
        \frac{1}{2}(\phi(N) - 2pJ)L & \text{if }p \divs N' \\
        \frac{1}{2}(\phi(N) \frac{p}{p - 1}-2pJ)L & \text{if }p \ndivs N' .
    \end{cases}
\end{equation}
Note that for this calculation,
we used \cref{eq:K_mean} which was for the full guess case.
For this comparison we use the full-guess estimate from \cref{sec:K},
whose dependence on the subproblem parameters is substantially simpler to analyze.
Although the numerical experiments use the modified Fourier guess,
\cref{fig:guess_comp} shows that the three guess strategies produce
$\beta_2$ values of comparable magnitude.
We can also relate the lattice determinants of the subproblems by,
\begin{equation} \label{eq:sub_det}
    d_N^2 = \begin{cases}
        d_{N'}^{2p} & \text{ if } p \divs N' \\
        p^{\phi(N) / (p - 1)}d_{N'}^{2(p-1)} & \text{ if } p \ndivs N',
    \end{cases}
\end{equation}
from the definition of $d_N$ in \cref{lem:det} and the totient formulas in \cref{eq:sub_totient_divs,eq:sub_totient_ndivs}.

These relations in \cref{eq:sub_totient_divs,eq:sub_totient_ndivs,eq:sub_K,eq:sub_det} reveal the competing effects governing the required value of $\beta_2$. 
Due to the $\Rgamma^\Phi$ term in the $\beta_2$ estimates, which depends on
the relative sizes of $\phi(N)$ and $K$, we expect that, in general, larger subproblems require larger values of $\beta_2$.
One notable exception comes from the case when $p = 2$ and $p \ndivs N'$. 
Here,
\cref{eq:sub_totient_ndivs} implies the subproblems have the same totient values,
$\phi(N) = \phi(N')$,
so $V_{\phi(N)} = V_{\phi(N')}$.
Combining this with the larger integer range for the smaller subproblem,
$L' > L$,
\cref{eq:sub_K} implies
$\\bE\event{K_N^2}<\bE\event{K_{N'}^2}$.
As $d_{N} > d_{N'}$ always holds by \cref{eq:sub_det},
this shows that the estimated $\beta_2$ in \cref{heur:beta2_full} for $N'$ is necessarily larger than the estimated $\beta_2$ for $N$. 

This immediately explains the data in \cref{tab:subprobs} suggesting that the subproblem of size 105 requires a larger value of $\beta_2$ than the full problem of size 210.
In fact,
the same phenomenon also occurs for every divisor pair $(N,N'=N/2)$ satisfying $2\ndivs N'$,
which includes 
$N = 10, 30, 70$, and 210.
In each case, both the theoretical and empirical values of $\beta_2$ are larger for the smaller subproblem,
confirming the prediction of the parameter analysis.

\subsection{2D Inversion Examples} 

Finally, we demonstrate the reconstruction algorithm on structured two-dimensional images.
We use the lattice implementation of the two-dimensional inversion algorithm from \cite{levinson2025recovery}. 
We consider two images:
${\tt X}_1$,
a Version 22 QR Code,
and ${\tt X}_2$,
the classic boat image from the USC-SIPI Image Database~\cite{set12}. 
As a QR code, 
${\tt X}_1$ has binary entries,
and its version specifies the dimensions $105 \times 105$~\cite{tiwari2016}.
The latter image ${\tt X}_2$ was rescaled for computational practicality from $512\times512$ to $210\times210$,
as the divisor structure of 210 is favorable to the inversion algorithm.
The original intensity range $0 \le {\tt X}_2 < 256 = L$ was preserved.

As discussed in \cref{sec:subprob}, reconstruction requires solving one-dimensional subproblems of various lengths $N'\divs N$.
Although the subproblems of both ${\tt X}_1$ and ${\tt X}_2$ have the same dimensions as those considered in \cref{tab:subprobs},
the required values of $\beta_2$ and $J$ will be different in this case.
For ${\tt X}_1$,
reconstruction should be possible with smaller values of $\beta_2$ and $J$,
since the $N=105$ subproblems have integer bound $L=105$,
compared with $L = 2\cdot255$ in \cref{tab:subprobs}.
While ${\tt X}_2$ has the same integer bounds,
the results may still differ because \cref{tab:subprobs} was generated using randomly sampled binomial images,
whereas ${\tt X}_2$ is a structured image.

First,
we reconstructed the images at single precision.
Based on the experiments in the previous section,
we used a global value $\beta_2$ of $10^7$ for every subproblem.
As in \cref{tab:subprobs}, the minimal sampling set
$(J = 1)$ was insufficient to solve most subproblems at this precision,
so additional DFT coefficients were sampled until each subproblem became solvable.   
For ${\tt X}_1$, we needed to sample 1006 DFT coefficients (9.12\% of the total).  For ${\tt X}_2$, we required 8791 (19.33\%) coefficients.
The corresponding least-norm reconstructions,
obtained by setting all but the sampled DFT coefficients (and their conjugates) to 0,
are shown in \cref{fig:images}.
The inversion algorithm exactly recovered the original image in 31.76s for ${\tt X}_1$ and 983.86s for ${\tt X}_2$. 

Repeating this experiment at double precision with $\beta_2=10^{14}$ and again choosing the smallest feasible $J$ for each subproblem reduced the required sampling to 625 (5.67\%) coefficients for ${\tt X}_1$ and 5021 (11.39\%) coefficients for ${\tt X}_2$.  
\Cref{fig:images} shows the least-norm reconstruction from these samples.
As the $\beta_2$ value was larger,
the recovery time increased to 107.99s for ${\tt X}_1$ and 5806.69s for ${\tt X}_2$. 

Finally,
we reconstructed ${\tt X}$ from the minimal sampling set. 
This consists of 315 coefficients (2.86\%) for ${\tt X}_1$ and 1260  coefficients (2.86\%) for ${\tt X}_2$.
\Cref{fig:images} shows the least-norm reconstruction from a minimal set of DFT coefficients.
Since no additional measurements were used, 
we experimentally determined the minimal value of $\beta_2$ required for each subproblem.  For ${\tt X}_1$, the largest $\beta_2$ value across all subproblems was $5 \times 10^{35}$, so we anticipate that about 35-40 digits of precision are required for the most difficult subproblems.  For the larger and more difficult image ${\tt X}_2$, the largest $\beta_2$ value was $5 \times 10^{97}$, which requires about 100 digits of precision.
With these larger $\beta_2$ values, 
the reconstructed image coincided exactly with the original image, but increased the reconstruction time to 168.43s for ${\tt X}_1$ and 10938.53s for ${\tt X}_2$.   

\begin{figure}[pos=!htb]
    \centering

    \includegraphics[width=\textwidth]{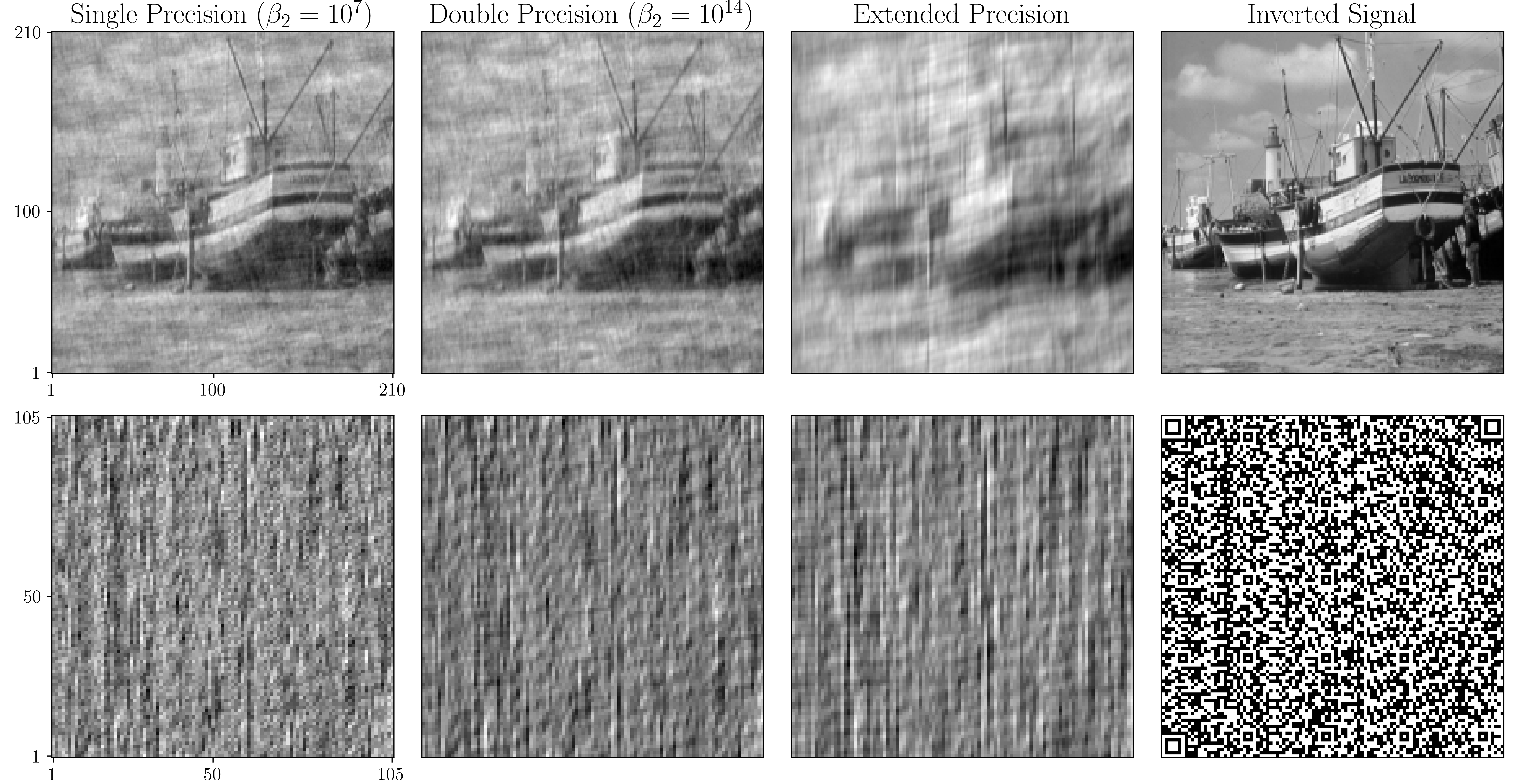}
    \vspace{-1.5pc}
    \caption{The first three panels in each row show the least-norm reconstruction obtained from the minimum number of DFT coefficients that was empirically found to be sufficient for Algorithm~2 to recover the original image. The first panel limits to single precision, the second panel limits to double precision, while the third panel allows extended precision so that the theoretically minimal data set can be used. The fourth panel shows the reconstructed image in each case, which exactly matches the original. The first row corresponds to the $210\times210$ boat image with integer values from 0 to 255, while the second row corresponds to a $105\times105$ QR code with binary values 0 or 1.}
    \label{fig:images}
\end{figure}

\section{Discussion}
\label{sec:disc}
This work provides a theoretical understanding of the lattice constructions introduced in
\cite{levinson2025recovery} for recovering integer-valued signals from incomplete DFT measurements.
The analysis explains how the geometry induced by the Fourier constraints,
measurement precision,
and lattice scaling parameters interact to determine successful recovery.
The resulting parameter estimates closely match the observed behavior of the algorithm across a wide range of experiments
and provide practical guidance for selecting lattice parameters.

The analysis also clarifies the role of the initial guess used in the lattice construction.
Incorporating all available Fourier measurements produces the smallest guess error and,
after appropriate tuning of $\beta_0$,
can yield the smallest required value of $\beta_2$.
However,
this full guess is comparatively sensitive to the choice of $\beta_0$.
The modified Fourier guess sacrifices a small amount of this optimal performance
while remaining close to optimal over a much wider range of $\beta_0$ values,
making it a more robust practical choice.
The zero guess provides a useful comparison with standard lattice formulations
and exhibits even weaker dependence on $\beta_0$,
although it generally requires somewhat larger lattice scales.

A second practical conclusion concerns the tradeoff between the number of sampled DFT coefficients
and the numerical precision required for reconstruction.
The minimal sampling sets established in \cite{levinson2025recovery}
are sufficient for uniqueness,
but the present analysis shows that reconstruction from such sets may require extremely large values of $\beta_2$
and correspondingly high measurement precision.
Sampling even a modest number of additional Fourier coefficients can substantially reduce the required lattice scale,
making recovery possible in standard single- or double-precision arithmetic.
Thus,
uniqueness from minimal Fourier data and numerically practical reconstruction represent distinct aspects of the inverse problem.

Throughout the theoretical analysis, several approximations were made, 
including modeling discrete sets and distributions with continuum analogues 
and applying the asymptotic limiting normal distribution in \cref{eq:CLT_sphere}.
While we did not provide any rigorous error analysis for these approximations,
the numerical simulations overwhelmingly demonstrate that the analysis accurately describes the geometry of the constructed lattices.
The relatively few noticeable discrepancies between the theoretical and empirical results suggest that the approximations used in the theoretical computation of $\rho(\beta_2)$ are not the dominant source of error.  Instead, the discrepancies were attributed to the heuristic model of the LLL-reduced basis
which neglected the approximation factor of the LLL algorithm.
Existing analyses of LLL focus on worst-case upper bounds for the approximation factor,
which are generally overly pessimistic for the lattice bases encountered in practice.   
While developing a more rigorous analysis remains an interesting mathematical question,
we expect the most significant improvements in the accuracy of the $\beta_2$ estimates to come from a better understanding of how LLL behaves on the lattice basis in \cref{eq:lattice_basis}.


\appendix
\section{Proof of \texorpdfstring{\cref{lem:det}}{Lemma 4.2}} 
\label{ap:det}

The proof of \cref{lem:det} uses the Kronecker product of matrices.  If ${\tt A}=\left[A_{ij}\right]$ is an $m\times n$ matrix and ${\tt B}$ is a $p\times q$ matrix, then the Kronecker product ${\tt A}\otimes{\tt B}$ is the $mp\times nq$ block matrix defined by
\begin{equation*}
    {\tt A}\otimes {\tt B} = \left[{\tt A}_{ij}{\tt B} \right]
\end{equation*}
We will use the following standard identities for the Kronecker product~\cite{bernstein2009matrix}
\begin{equation*} 
    ({\tt A} \otimes {\tt B})^T = {\tt A}^T \otimes {\tt B}^T,
    \qquad
    ({\tt A} \otimes {\tt B})({\tt C} \otimes {\tt D}) = ({\tt AC}) \otimes ({\tt BD}) ,
\end{equation*}
and,
if ${\tt A}$ and ${\tt B}$ are square,
\begin{equation} \label{eq:kron_det}
    \qquad
    \det({\tt A} \otimes {\tt B}) = \det({\tt A})^p \det({\tt B})^m.
\end{equation}
\begin{proof}
Let $N=p_1^{\alpha_1}p_2^{\alpha_2}\cdots p^{\alpha_\omega}_\omega$.  For each $1\le r\le \omega$, define the vectors ${\bf v}_{\ell}^{(r)}$ in $\Z^{p_r}$ by
\begin{equation}
\label{eq:vell_def}
     {\bf v}_{l}^{(r)} \coloneqq {\bf e}_{l}^{(p_{r})} - {\bf e}_{l-1}^{(p_{r})} , \qquad
    1\le \ell\le p_r-1 ,
\end{equation}
where ${\bf e}_{\ell}^{(n)}$ is the $\ell$th standard basis vector of $\R^n$.  
For each $r$, the vectors ${\bf v}_1^{(r)},\ldots,{\bf v}^{(r)}_{p_r-1}$ form a basis for  the subspace of $\Z^{p_r}$ whose coordinates sum to zero.
Let  ${\tt B}^{(r)}$ be the $p_r^{\alpha_r} \times p_r^{\alpha_r-1}(p_r-1)$ matrix  defined by the Kronecker product
\begin{equation}
\label{eq:Br_def}
    {\tt B}^{(r)} \coloneqq \begin{bmatrix}
        {\bf v}_1^{(r)} \cdots {\bf v}_{p_r - 1}^{(r)}
	\end{bmatrix}
    \otimes \idmat{p_r^{\alpha_r - 1}} ,
\end{equation}
and define ${\tt B}$ as the Kronecker product of all the ${\tt B}^{(r)}$ matrices,
\begin{equation}
\label{eq:kron_basis}
   {\tt B} \coloneqq {\tt B}^{(1)} \otimes \cdots \otimes {\tt B}^{(\omega)} .
\end{equation}
Since the Kronecker product multiplies row and column dimensions, ${\tt B}$ has dimensions
\begin{equation*}
 \left(\prod_{r=1}^{\omega}p_r^{\alpha_r}\right) \times    \left(\prod_{r=1}^{\omega}p_r^{\alpha_r-1}(p_r-1)\right) = N \times \phi(N).
\end{equation*}
The Kronecker product construction of \cref{eq:kron_basis} is a natural candidate for a basis of $\mathcal{K}$, as each factor corresponds to one of the prime-power zero-sum constraints defining $\mathcal{K}$. However, by the construction of ${\tt B}$ in \cref{eq:kron_basis}, the columns of ${\tt B}$ are naturally indexed by the standard tensor product basis of $\bigotimes_{r=1}^{\omega}\R^{p_r^{\alpha_r}}$,
\begin{equation}
\label{eq:tensorprod_basis}
    {\bf e}_{j_1}^{(p_1^{\alpha_1})}\otimes\cdots\otimes{\bf e}_{j_\omega}^{(p_\omega^{\alpha_\omega})} .
\end{equation}
In contrast, the rows and columns of the constraint matrix ${\tt A}$ are naturally indexed by the standard basis ${\bf e}_n^{(N)}$ of $\R^N$.    To compare these two descriptions, we first introduce a permutation matrix relating the corresponding orderings.  

Let ${\bf j}\colon Z_{N} \to \prod_{r=1}^{\omega} \Z_{p_{r}^{\alpha_{r}}}$ be the Chinese remainder isomorphism defined by 
\begin{equation}
\label{eq:CRT_indexmap}
   {\bf j}(n) =  (j_1(n),\ldots,j_\omega(n))\coloneqq(n \bmod p_1^{\alpha_1},\ldots,n\bmod p_\omega^{\alpha_\omega}),
\end{equation}
We order the tuples ${\bf j}(n)$ lexicographically, thereby inducing a lexicographic ordering of the tensor-product basis in \cref{eq:tensorprod_basis}.
We then define the permutation matrix ${\tt P}$ by
\begin{equation*}
    {\tt P}\left({\bf e}_n^{(N)}\right)= {\bf e}_{j_1(n)}^{(p_1^{\alpha_1})}\otimes\cdots\otimes{\bf e}_{j_\omega(n)}^{(p_\omega^{\alpha_\omega})} ,
\end{equation*}
Since the map $j$ is a bijection, ${\tt P}$ simply reorders the standard basis vectors of $\R^N$, and is therefore a permutation matrix.  

Since ${\tt P}$ is a permutation matrix, it is unimodular and therefore defines an automorphism of the lattice $\Z^N$.  Writing $\hat{\bf x}={\tt P}{\bf x}$ and $\hat{\tt A}={\tt A}{\tt P}^T$, for every ${\bf x}\in\Z^N$ we have 
\begin{equation}
\label{eq:xhat_in_ker}
    {\bf x}\in\mathcal{K} \iff {\tt A}{\bf x} = 0 \iff {\tt A}{\tt P}^T\hat{\bf x} = 0.  
\end{equation}
Therefore ${\bf x}$ is in the lattice $\mathcal{K}$ if and only if $\hat{\bf x}\in\ker(\hat{\tt A})\cap\Z^N$.  We now identify matrices with the same kernels as the transformed blocks of $\hat{\tt A}$, from which it will follow that the columns of ${\tt B}$ form an integer basis for $\ker(\hat{\tt A})$.

Write the block row partition of the constraint matrix as
\begin{equation*}
{\tt A} = \begin{bmatrix}
{\tt A}^{(1)} \\
\vdots \\
{\tt A}^{(\omega)}
\end{bmatrix} ,
\qquad {\tt A}^{(r)} \coloneqq \begin{bmatrix}
    \idmat{N/p_r} &\cdots & \idmat{N/p_r}
\end{bmatrix}
\end{equation*}
and define the transformed blocks by
\begin{equation*}  
\hat{\tt A}^{(r)} =  {\tt A}^{(r)}{\tt P}^T.  
\end{equation*}
To describe the kernels of the transformed blocks $\hat{\tt A}^{(r)}$, we further decompose the $r$th coordinate of ${\bf j}$ from \cref{eq:CRT_indexmap} by writing 
\begin{equation*}
\label{eq:ik_coords}
    j_r(n)=i_r(n)p_r^{\alpha_r-1}+k_r(n) , 
\end{equation*}
where  $0 \le i_{r} < p_{r}$ and $0 \le k_{r} < p_{r}^{\alpha_{r} - 1}$.
Applying this decomposition to the full index tuple ${\bf j}(n)$ yields,
  \begin{equation} \label{eq:j_expand}
      {\bf j}(n)=(j_1, \dots, j_{r-1},(i_r,k_r),j_{r+1},\dots,j_\omega).
  \end{equation}
We next observe that
\begin{equation*}
  {\tt A}^{(r)}{\bf e}_{n}^{(N)} = {\bf e}_{n\bmod N/p_{r}}^{(N/p_{r})},   
\end{equation*}
and hence the action of ${\tt A}^{(r)}$ depends only on the residue classes modulo $N/p_r$. 
We therefore need to characterize these residue classes in the ${\bf j}$-coordinates from \cref{eq:CRT_indexmap}.  By the Chinese remainder theorem, two indices satisfy
\begin{equation}
\label{eq:residue_equiv}
    n=n' \bmod{N/p_r} \iff \begin{cases}
j_s(n)=j_s(n'), & \text{for all } s\neq r,\\
k_r(n)=k_r(n'), 
\end{cases}
\end{equation} 
since $N/p_r = p_r^{\alpha_r-1}\prod_{s\neq r}p_s^{\alpha_s}$.

This observation motivates the auxiliary matrix
\begin{equation}
\label{eq:Atilde_def}
    \widetilde{{\tt A}}^{(r)} \coloneqq (\idmat{p_1^{\alpha_1} \cdots p_{r-1}^{\alpha_{r-1}}}) \otimes (1_{1 \times p_r} \otimes \idmat{p_r^{\alpha_r-1}}) \otimes (\idmat{p_{r+1}^{\alpha_{r+1}} \cdots p_\omega^{\alpha_\omega}}) ,
\end{equation}
whose action is precisely to sum over the coordinate $i_r$ while leaving all remaining coordinate indices in \cref{eq:j_expand} fixed.
We claim that
$\ker(\hat{\tt A}^{(r)})=\ker(\widetilde{\tt A}^{(r)})$.

Take an arbitrary vector $\hat{\bf x}={\tt P}{\bf x}$.    Using the decomposition $j_r=i_rp_r^{\alpha_r-1}+k_r$, define the coordinates of $\hat{\bf x}$ by
\begin{equation*}
\hat{x}(j_1,\ldots,j_{r-1},(i_r,k_r),j_{r+1},\ldots,j_\omega)
=
x_{{\bf j}^{-1}(j_1,\ldots,j_{r-1},(i_r,k_r),j_{r+1},\ldots,j_\omega)}.
\end{equation*} 
Then $\hat{\bf x}$ can be written in the tensor-product coordinates of \cref{eq:tensorprod_basis} as
\begin{align*}
    \hat{\bf x}= \sum_{\substack{ j_s \text{ for } s\ne r \\ i_r, k_r}}\hat{x}(j_1,\ldots,&j_{r-1},(i_r,k_r),j_{r+1},\ldots,j_\omega) \\
    &\cdot\left[{\bf e}_{j_1}^{(p_1^{\alpha_1})}\otimes\cdots\otimes\left({\bf e}_{i_r}^{(p_r)}\otimes{\bf e}_{k_r}^{(p_r^{\alpha_r-1})}\right)\otimes\cdots\otimes{\bf e}_{j_\omega}^{(p_\omega^{\alpha_\omega})}\right] ,
\end{align*}
where the $r$th tensor basis factor  ${\bf e}^{(p_r^{\alpha_r})}_{j_r}$ has been further indexed using the decomposition $j_r=i_rp_r^{\alpha_r-1}+k_r$.
A direct computation gives
\begin{align}
\label{eq:Atilde_comp}
 \widetilde{\tt A}^{(r)}\hat{\bf x}
&=
\sum_{\substack{ j_s \text{ for } s\ne r \\ i_r, k_r}}
\hat{x}(j_1,\ldots,j_{r-1},i_r,k_r,j_{r+1},\ldots,j_\omega) \notag \\
&\qquad\qquad\qquad\qquad\cdot
{\bf e}_{j_1}^{(p_1^{\alpha_1})}\otimes\cdots\otimes
({\bf 1}_{1\times p_r}{\bf e}_{i_r}^{(p_r)})\otimes{\bf e}_{k_r}^{(p_r^{\alpha_r-1})}
\otimes\cdots\otimes
{\bf e}_{j_\omega}^{(p_\omega^{\alpha_\omega})} \notag
\\
&=\sum_{\substack{ j_s \text{ for } s\ne r \\  k_r}}\left(\sum_{i_r=0}^{p_r-1}\hat{x}(j_1,\ldots,j_{r-1},i_r,k_r,j_{r+1},\ldots,j_\omega)\right) 
\\
&\qquad\qquad\qquad\qquad\cdot
{\bf e}_{j_1}^{(p_1^{\alpha_1})}\otimes\cdots\otimes
{\bf e}_{j_{r-1}}^{(p_{r-1}^{\alpha_{r-1}})}\otimes
{\bf e}_{k_r}^{(p_r^{\alpha_r-1})}\otimes
{\bf e}_{j_{r+1}}^{(p_{r+1}^{\alpha_{r+1}})}
\otimes\cdots\otimes
{\bf e}_{j_\omega}^{(p_\omega^{\alpha_\omega})}. \notag
\end{align}
As the tensor-product basis vectors ${\bf e}_{j_1}\otimes\cdots\otimes{\bf e}_{j_{r-1}}\otimes{\bf e}_{k_r}\otimes{\bf e}_{j_{r+1}}\otimes\cdots\otimes{\bf e}_{j_\omega}$ of the codomain of $\widetilde{\tt A}^{(r)}$ are linearly independent,
$\tilde{\tt A}^{(r)}\hat{\bf x}= {\bf0}$ if and only if every coefficient in \cref{eq:Atilde_comp} is zero.  
Therefore, we can conclude that $\tilde{\tt A}^{(r)}\hat{\bf x}= {\bf0}$ if and only if
\begin{equation}
\label{eq:cond_ker_tildeA}
     \sum_{i_r=0}^{p_r-1}\hat{x}(j_1,\ldots,j_{r-1},i_r,k_r,j_{r+1},\ldots,j_\omega) = 0 ,
\end{equation}
for all fixed choices of coordinates $j_1,\dots,j_{r_1},j_{r+1},\dots,j_\omega$ and $k_r$.
By \cref{eq:residue_equiv},
varying $i_r$ while fixing the remaining coordinates enumerates exactly the ${\bf j}$-image of one residue class modulo $N/p_r$.
Hence,
\cref{eq:cond_ker_tildeA} states precisely that the entries of ${\bf x}={\tt P}^T\hat{\bf x}$ sum to zero over every residue class modulo $N/p_r$.
This corresponds exactly with the condition ${\tt A}^{(r)}{\bf x} = {\bf 0}$.
Since $\hat{\bf x}={\tt P}{\bf x}$ and $\hat{\tt A}^{(r)} =  {\tt A}^{(r)}{\tt P}^T$,
the condition  ${\tt A}^{(r)}{\bf x} = {\bf 0}$ is equivalent to $\widehat{\tt A}^{(r)}\hat{\bf x}={\bf 0}$.
Therefore,
we have $\ker(\widetilde{\tt A}^{(r)})=\ker(\hat{\tt A}^{(r)})$.

As $\hat{\bf x}\in\ker(\widetilde{\tt A})$ if and only if $\hat{\bf x}\in\ker(\widetilde{\tt A}^{(r)})$ for all $r$, this immediately yields that 
\begin{equation*}
    \ker(\widetilde{\tt A})=\ker(\hat{\tt A}) ,
\end{equation*}
for the matrix $\widetilde{\tt A}$ defined block-wise by,
\begin{equation*}
    \widetilde{\tt A} = \begin{bmatrix}
    \widetilde{\tt A}^{(1)} \\
    \vdots \\
    \widetilde{\tt A}^{(\omega)}
    \end{bmatrix} .
\end{equation*}
Combining this with \cref{eq:xhat_in_ker}, we have that ${\bf x}\in\mathcal{K}$ if and only if $\hat{\bf x}\in\ker(\widetilde{\tt A})$.  We next need to show that the columns of ${\tt B}$ form an integer basis for $\ker(\widetilde{\tt A})$.

Fix $r$.  Using the definitions of \cref{eq:Br_def,eq:kron_basis,eq:Atilde_def}, a direct computation shows that every column of ${\tt B}$ is in $\ker(\widetilde{\tt A}^{(r)})$ as
\begin{align}
    \tilde{\tt A}^{(r)}{\tt B} &= \left((\idmat{p_1^{\alpha_1} \cdots p_{r-1}^{\alpha_{r-1}}}) \otimes (1_{1 \times p_r} \otimes \idmat{p_r^{\alpha_r-1}}) \otimes (\idmat{p_{r+1}^{\alpha_{r+1}} \cdots p_\omega^{\alpha_\omega}})\right)\cdot \left({\tt B}^{(1)} \otimes \cdots \otimes {\tt B}^{(\omega)} \right) \notag\\ 
  &= {\tt B}^{(1)} \otimes \cdots \otimes {\tt B}^{(r-1)}\otimes\left[(1_{1 \times p_r} \otimes \idmat{p_r^{\alpha_r-1}})\cdot{\tt B}^{(r)}\right]\otimes {\tt B}^{(r+1)}\otimes\cdots\otimes {\tt B}^{(\omega)} \notag \\
  &= {\tt B}^{(1)} \otimes \cdots \otimes {\tt B}^{(r-1)}\otimes\left[1_{1 \times p_r} \cdot \begin{bmatrix}
        {\bf v}_1^{(r)} \cdots {\bf v}_{p_r - 1}^{(r)}
	\end{bmatrix}\right]\otimes\idmat{p_r^{\alpha_r-1}}\otimes {\tt B}^{(r+1)}\otimes\cdots\otimes {\tt B}^{(\omega)} \notag \\
    &= {\tt 0} , \notag
\end{align}
where the last line follows since each ${\bf v}_\ell^{(r)}$ sums to 0.  As this holds for each $r$, every column of ${\tt B}$ lies in $\ker(\widetilde{\tt A})$.  Furthermore, $\nullity(\widetilde{\tt A})=\nullity({\tt A})=\phi(N)$, which is exactly the number of columns of ${\tt B}$.
To show that ${\tt B}$ satisfies the conditions for a basis,
it only remains to show that the columns are linearly independent.
Letting ${\tt G}={\tt B}^T{\tt B}$ be the Gram matrix of ${\tt B}$,
if $\det({\tt G})\neq0$,
then the columns of ${\tt B}$ are linearly independent.
Thus,
linear independence will follow from the determinant computation below.

To compute the determinant of ${\tt G}$,
we first express ${\tt G}$ as a Kronecker product of smaller Gram matrices ${\tt G}^{(r)}={{\tt B}^{(r)}}^T {\tt B}^{(r)}$ by
\begin{align}
\label{eq:G_def}
    {\tt G} &\coloneqq {\tt B}^T {\tt B}= \left({\tt B}^{(1)} \otimes \cdots \otimes {\tt B}^{(\omega)}\right)^T\cdot \left({\tt B}^{(1)} \otimes \cdots \otimes {\tt B}^{(\omega)}\right)\notag\\
    &= \left({{\tt B}^{(1)}}^T {\tt B}^{(1)} \right)\otimes\cdots\otimes\left({{\tt B}^{(\omega)}}^T {\tt B}^{(\omega)} \right) \notag\\
    &\coloneqq {\tt G}^{(1)}\otimes\cdots\otimes{\tt G}^{(\omega)} .
\end{align}
Then for any $r$, 
we can express ${\tt G}^{(r)}$ in terms of the gram matrix of the ${\bf v}^{(r)}_\ell$s,
\begin{align*}
    {\tt G}^{(r)}&=  {{\tt B}^{(r)}}^T {\tt B}^{(r)}\\
    &=\left(\begin{bmatrix}
    {\bf v}^{(r)}_1 \cdots {\bf v}^{(r)}_{p_r-1}
    \end{bmatrix}
    \otimes \idmat{p_r^{\alpha_r - 1}}\right)^T
    \left(\begin{bmatrix}
        {\bf v}_1^{(r)} \cdots {\bf v}^{(r)}_{p_r-1}
    \end{bmatrix}
    \otimes \idmat{p_r^{\alpha_r - 1}}\right) \\
    &= \left(\begin{bmatrix}
        {\bf v}_1^{(r)} \cdots {\bf v}^{(r)}_{p_r-1}
    \end{bmatrix}^T\begin{bmatrix}
        {\bf v}^{(r)}_1 \cdots {\bf v}^{(r)}_{p_r-1}
    \end{bmatrix}\right) \otimes \left(\idmat{p_r^{\alpha_r - 1}}^T\idmat{p_r^{\alpha_r - 1}}\right) \\
    &= \begin{bmatrix}
        {{\bf v}_1^{(r)}}^T{\bf v}^{(r)}_1 & \cdots & {{\bf v}_{p_r - 1}^{(r)}}^T{\bf v}^{(r)}_1 \\
        \vdots & \ddots & \vdots \\
        {{\bf v}^{(r)}_1}^T{\bf v}^{(r)}_{p_r-1} & \cdots & {{\bf v}_{p_r - 1}^{(r)}}^T{\bf v}^{(r)}_{p_r-1} \\
    \end{bmatrix}
    \otimes \idmat{p_r^{\alpha_r-1}}  \\ 
    &\coloneqq {\tt C}_{p_r-1}\otimes \idmat{p_r^{\alpha_r-1}} . 
\end{align*}
To determine the matrix ${\tt C}_{p_r-1}$, we apply the orthonormality of the standard basis to the definition of ${\bf v}^{(r)}_\ell$ in \cref{eq:vell_def}:
\begin{equation*}
    {{\bf v}_\ell^{(r)}}^T{\bf v}^{(r)}_{\ell'} 
    = {{\bf e}_\ell^{(p_r)}}^T{\bf e}_{\ell'}^{(p_r)}
    - {{\bf e}_{\ell-1}^{(p_r)}}^T{\bf e}_{\ell'}^{(p_r)}
    - {{\bf e}_{\ell}^{(p_r)}}^T{\bf e}_{\ell'-1}^{(p_r)}
    + {{\bf e}_{\ell-1}^{(p_r)}}^T{\bf e}_{\ell'-1}^{(p_r)}
    = \begin{cases}
        2 & \ell = \ell' \\
        -1 & \abs{\ell - \ell'} = 1 \\
        0 & \text{otherwise}.
    \end{cases}
\end{equation*}
This gives us the tridiagonal form of ${\tt C}_{p_r-1}$,
\begin{equation*}
    {\tt C}_{p_r-1}=
    \begin{bmatrix}
    2 & -1 &        &        &  \\
    -1 & 2 & -1     &        &  \\
       & \ddots & \ddots & \ddots &  \\
       &        & -1 & 2 & -1 \\
       &        &    & -1 & 2
    \end{bmatrix}, 
\end{equation*}
whose determinant is $\det({\tt C}_{p_r-1})=p_r$~\cite[Fact 3.20.7]{bernstein2009matrix}.
Using the determinant identity across Kronecker products from \cref{eq:kron_det},
we can compute
\begin{equation}
\label{eq:G_r_det}
    \det({\tt G}^{(r)}) = \det({\tt C}_{p_r-1}\otimes \idmat{p_r^{\alpha_r-1}}) = \det\left({\tt C}_{p_r-1}\right)^{p_r^{\alpha_r-1}}\det\left({\tt I}_{p_r^{\alpha_r-1}} \right)^{p_r-1} = p_r^{\,p_r^{\alpha_r-1}}.
\end{equation}
Finally, by repeatedly applying \cref{eq:kron_det} to the formula for ${\tt G}$ in \cref{eq:G_def},
and
substituting in the determinants from \cref{eq:G_r_det},
we obtain 
\begin{align*}
    \det({\tt G})
    = \prod_{r=1}^{\omega}(\det {\tt G}_r)^{\prod_{s\neq r}\phi(p_s^{\alpha_s})}
    = \prod_{r=1}^{\omega}p_r^{p_r^{\alpha_r-1}\prod_{s\neq r}\phi(p_s^{\alpha_s})}
    = \prod_{r=1}^{\omega}p_r^{(\prod_{s}\phi(p_s^{\alpha_s}))/(p_r-1)}
\end{align*}
Using $\phi(N)=\prod_{s}\phi(p_s^{\alpha_s})$ and taking the square root gives,
\begin{equation} \label{eq:G_det}
    \sqrt{\det({\tt G})}=\left(\prod_{r=1}^{\omega}p_r^{\phi(N)/(p_r-1)}\right)^{1/2}.
\end{equation}
This argument computed the determinant using a basis for $\ker(\hat{\tt A})$.
For our lattice of interest,
$\mathcal{K}$,
we consider the basis of ${\tt P}^T{\tt B}$.
The unimodularity of ${\tt P}$ ensures that ${\tt P}^T{\tt B}$ maintains the integer entries and linear independence of ${\tt B}$,
so the relation ${\tt A} = \hat{\tt A}{\tt P}^T$ implies that ${\tt P}^T$ is a basis for $\mathcal{K}$.
Finally,
we complete the proof by computing the determinant of $\mathcal{K}$ in terms of this basis from \cref{eq:G_det},
\begin{equation*}
    \Lambda(\mathcal{K}) 
    = \sqrt{\det\left(({\tt B}{\tt P}^T)^T({\tt B}{\tt P})\right)} = \sqrt{\det({\tt B}^T{\tt B})} = \sqrt{\det({\tt G})}
    =\left(\prod_{r=1}^{\omega}p_r^{\phi(N)/(p_r-1)}\right)^{1/2}.
\end{equation*}
\end{proof}

\bibliographystyle{cas-model2-names}

\bibliography{bib1}

\end{document}